\documentclass[letterpaper, 11 pt]{amsart}

\newcommand{\ignore}[1]{ }

\usepackage{amsmath, amsthm, amsfonts, amssymb}
\usepackage{bbm}
\usepackage[margin=1in]{geometry}
\usepackage{xypic}

\usepackage{ytableau}
\usepackage{tikz}

\usepackage{color}   %May be necessary if you want to color links

\usepackage{hyperref}
\hypersetup{
  colorlinks=true, %set true if you want colored links
   linkcolor=red,  %choose some color if you want links to stand out
}
\numberwithin{equation}{section}

\newtheorem{theorem}{Theorem}[section]
\newtheorem{cor}[theorem]{Corollary}
\newtheorem{lemma}[theorem]{Lemma}

\newtheorem{prop}[theorem]{Proposition}
\newtheorem{remark}[theorem]{Remark}

\usepackage{graphicx}
\usepackage{epstopdf}

\usepackage{caption}
\usepackage{subcaption}

\usepackage[utf8]{inputenc}

\def\e{\mathbb{E \,}}
\def\E{\mathbb{E}}

\def\p{\mathbb{P \,}}
\def\R{\mathbb{R}}
\def\Rplus{\mathbb{R}_+}
\def\a{\alpha(n)}
\def\b{\beta(n)}

\def\A{\mathcal{A}}
\def\B{\mathcal{B}}

\def\C{\mathcal{C}}
\def\Triangle{\mathcal{T}}
\def\Pn{\mathcal{P}_n}
\def\P{\mathcal{P}}
\def\D{\mathcal{D}}
\def\DV{\mathcal{D}_0} % D from Vershik's paper

\def\Pr{\mathbb{P}}
\def\N{\mathbb{N}}
\def\PhiLS{\Phi(t; r, B)}

\def\stable{$(r,B,K)-$stable}

\def\MMstable{$(r,B,K)$--{\textsf{MM-stable}}}
\def\MPstable{$(r,B,K)$--{\textsf{MP-stable}}}
\def\MMstablea{$(r,B,a,K)$--{\textsf{MM-stable}}}
\def\MPstablea{$(r,B,a,K)$--{\textsf{MP-stable}}}
\def\varphiv{\varphi_{\bf v}}
\def\bv{{\bf v}}

\def\bfx{{\bf x}}
\def\bfy{{\bf y}}
\def\bolda{{\bf a}}
\def\boldb{{\bf b}}
\def\MPb{{\textsf{MP--bijection }}}
\def\MMb{{\textsf{MM--bijection }}}
\def\MPbns{{\textsf{MP--bijection}}}
\def\MMbns{{\textsf{MM--bijection}}}
\def\Tv{T_{\bf v}}

\def\smooth{unrestrictedly smooth }
\def\rsmooth{restrictedly smooth }

\def\rhohat{\widehat{\rho}}
\def\Dhat{\widehat{D}}
\def\Shat{\widehat{S}}

\def\O{\mathcal{O}}

\def\xofn{x(n)}

\def\th{\thinspace}

\def\.{\hskip.06cm}
\def\ts{\hskip.03cm}
\def\mts{\hspace{-.04cm}}

\usepackage{ytableau}
\usepackage{young}

\newcommand*{\Scale}[2][4]{\scalebox{#1}{$#2$}}%
\def\la{\lambda}

\def\vp{\varphi}

\def\ca{\mathcal A}

\def\cl{\mathcal L}

\def\<{\langle}
\def\>{\rangle}

\def\0{{\mathbf 0}}

\def\implies{\Rightarrow}

\def\.{\hskip.06cm}
\def\ts{\hskip.03cm}
\def\mts{\hspace{-.04cm}}

\def\bv{\textbf{\textit{v}}}

\def\nin{\noindent}

\title{Limit shapes via bijections}

\author[Stephen DeSalvo, \. Igor Pak] {Stephen DeSalvo$^\ast$ \ \, and \, \  Igor Pak$^\ast$}

\thanks{\today}

\thanks{\thinspace ${\hspace{-.45ex}}^\star$Department of Mathematics,
UCLA, Los Angeles, CA~90095.
\hskip.06cm
Email:
\hskip.06cm
\texttt{\{stephendesalvo,\ts{pak}\}@math.ucla.edu}}

\begin{document}

\begin{abstract}
We compute the limit shape for several classes of restricted integer partitions, 
where the restrictions are placed on the part sizes rather than the multiplicities.
Our approach utilizes certain classes of bijections which map limit shapes continuously in the plane.
% One class of bijections are those outlined in~\cite{PakNature}, where the scaling function for the limit shape is $\sqrt{n}$.
We start with bijections outlined in~\cite{PakNature}, and 
extend them to include limit shapes with different scaling functions.
\end{abstract}

\maketitle

\setcounter{tocdepth}{1}
\tableofcontents

%\tableofcontents

% ***********************************************Introduction***********************************************
\section{Introduction}

\subsection{Preliminaries}
The study of random combinatorial objects is an amazing success story,
which grew from ad hoc problems and simple ideas to a flourishing field of
its own, with many astonishing results, advanced tools and important
applications (see e.g.~\cite{AK}).  Since the early work of Erd\H{o}s and
collaborators, it has long been known that the ``typical'' objects
often have an interesting and unexpected structure often worth exploring.
Some such random combinatorial objects such as various classes of random graphs
are understood to a remarkable degree (see e.g.~\cite{JLR}), while others,
such as random finite groups remain largely mysterious (see e.g.~\cite{BNV}).
\emph{Random integer partitions} occupy the middle ground, with a large mix of
known results and open problems (cf.\ Section~\ref{final:remarks}), and are
a subject of this paper.

\smallskip

\emph{Partition Theory} is a classical subject introduced by Euler nearly
three hundred years ago, which has strong connections and applications to
a variety of areas ranging
from Number Theory to Special Functions, from Group Representation Theory
to Probability and Statistical Physics.  A \emph{partition} of~$n$ is
an integer sequence $\ts (\la_1,\ldots,\la_\ell)$, such that
\ts $\la_1\ge \ldots \ge \la_\ell$ \ts and \ts $\la_1+ \ldots + \la_\ell=n$.
Much of this paper is concerned with asymptotic analysis of
various classes of partitions, some known and some new.

Let us begin by noting that already the number $p(n)$ of integer partitions
of~$n$ is a notoriously difficult sequence; it is not even known if it
is equidistributed modulo~$2$.  There are several specialized tools, however,
which allow a detailed information about random partitions.  First,
the analytic tools by Hardy--Ramanujan and Rademacher give a precise
asymptotics of~$p(n)$, see e.g.~\cite{Andrews}.  Second, the \emph{Boltzmann sampling},
in this case invented earlier by Fristedt~\cite{Fristedt},
allows a uniform sampling of random partitions of~$n$ with
a high degree of independence between part sizes.  This paper
introduces a new combinatorial tool, which is best used in
conjunction with these.

Our main emphasis will be not on enumeration of various classes of partitions,
but on the limit shapes of random such partitions.  It is a well known phenomenon
that random combinatorial objects, e.g.\ random graphs, tend to have unusually
interesting parameters, of interest in both theory and applications.
When represented geometrically, these random objects illuminate the
underlying 0/1 laws of their asymptotic behavior; examples included
various \emph{arctic circle} shapes for domino tilings and perfect matchings
(see e.g.~\cite{CEP,KOS}), and for alternating sign matrices (see~\cite{CP}).

Note that there are many different notions of the \emph{limit shape},
which differ depending on the context and the geometric representation.
First, one needs to be careful in the choice of the \emph{scaling}, by
which combinatorial objects of different sizes are compared to each other.
In this paper we use different scaling for different classes of partitions.

Second, there is more than one notion of convergence to the limit shape,
some stronger than others.  Roughly speaking, they describe how close are
the random objects to the limit shape.  Here we work with two different notions,
convergence in expectation, which we refer to as the \emph{weak notion of
the limit shape}, and convergence in probability, which we refer to as the
\emph{strong notion of the limit shape}, which is implied by a stronger large
deviation result as well as a local central limit theorem (see below).

As we mentioned earlier the main result of the paper is the computation of
limit shapes for several well known classes (\emph{ensembles}) of random partitions for which
this was neither known nor even conjectured before.  Most our results are not
obtainable by previously known techniques. We postpone the precise formulation
of our results until the next subsection; for now let us briefly survey the
long history of ideas leading to this work.

%\smallskip
\subsection{Brief history}

Arguably, partition theory began with Euler's work establishing equality
for the number of partitions in different classes by means of
algebraic manipulation of generating functions.  The next crucial
step was made by Sylvester and his students which introduced
\emph{partition bijections} and used them to prove many results by Euler
and others.  Roughly speaking, a partition is represented as
a \emph{Young diagram} and the squares are rearranged accordingly.
We refer to~\cite{Pak} for a detailed survey of this approach.

The study of the number $p_r(n)$ of partitions of $n$ with at most
$r$ parts goes back to Cayley, Sylvester and MacMahon.  Hardy and
Ramanujan obtained a remarkable asymptotic formula, further
refined by others.  The first result on limit shapes is the
\emph{Erd\H{o}s--Szekeres Theorem}, which can be formulated as follows:
\[
\e \bigl[ \# i : \lambda_i < \alpha\sqrt{n}\bigr] \. = \. \beta \sqrt{n}\ts \bigl(1+o(1)\bigr),
\quad
\text{where} \ \ e^{-\alpha c} + e^{-\beta c} \.= \.1 \ \ \, \text{and} \ \ c = \frac{\pi}{\sqrt{6}}\..
\]
These results were strengthened and extended in a number of papers,
see e.g.~\cite{LD,Fristedt,Vershik,VershikYakubovich}.

There are two prominent types of partition classes considered in the literature,
which are essentially conjugate to each other.  The classes are defined by restrictions
on parts~$\la_i$, and those with restriction on the number of times $m_i(\la)$ part $i$
appear in the partition, with no restrictions relating parts to each other.
For example, the set of partitions into odd numbers fits into the latter class, which
corresponds to a restriction of the form $m_i(\la) = 0$ for all $i$ even.
Sometimes this is less obvious, for example, for partitions into distinct parts,
the restriction is that each part appears at most once, i.e., $m_i(\la) \leq 1$, rather
than the fact that they are distinct, i.e., $\lambda_1 > \lambda_2 > \ldots > \lambda_\ell$.

There is an extensive body of research related to limit shapes of integer partitions
for which the restrictions are solely placed on the multiplicities $m_i(\la)$.
In this paper we consider restrictions on part sizes, such as the
\emph{convex partitions}, which satisfy
$\ts \la_1 - \la_{2} \ge  \la_{2} - \la_{3} \ge \ldots$ \ts These are
equinumerous with partitions into triangular parts $\binom{k}{2}$,
i.e., $m_i(\la) = 0$ whenever $i$ is not a triangular number,
which have a well understood limit shape.  The bijection between these
two classes of partitions is also well known~\cite{CS,PakGeo}, and
we demonstrate how to obtain one limit shape from the other.

The idea of this paper is to use partition bijections to transfer
the limit shape results from one class of partitions onto another.
Making this formal is rather subtle.  In~\cite{PakNature}, the second
author introduces a necessary formalism essentially allowing such
transfer, with the aim of proving that there is no natural proof of
the classical \emph{Rogers--Ramanujan identities}.  This paper
continues this approach, but can be read independently of~\cite{PakNature} and other earlier work.

\subsection{Limit shapes}
\label{intro:limit:shapes}
\begin{figure}
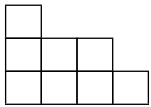

\[
\begin{Young}
   \cr
 &  & \cr
 &  &  &  \cr
 \end{Young}
\]
\caption{The Young diagram for the partition $4, 3, 1$. }
\label{young:diagram}
\end{figure}
To understand the notion of the limit shape, we start with a geometric interpretation of an integer partition.
Each integer partition can be visualized as a curve in the first quadrant of $\mathbb{R}^2$, with a height function determined by the part sizes; see Figure~\ref{young:diagram}.
By simultaneously plotting all of the Young diagrams of every integer partition of a very large fixed size~$n$, for each $n$ one observes empirically a curve which is close to most of the diagram functions, namely,
\[ e^{-c\, \frac{y}{\sqrt{n}}} + e^{-c\, \frac{x}{\sqrt{n}}} = 1, \quad \text{where} \quad x>0, \ y>0, \ \ \text{and} \ \ c=\frac{\pi}{\sqrt{6}}\.. \]
Normalizing the $x$- and $y$-axes by $\sqrt{n}$, a single curve emerges in the limit as $n\to \infty$.
This curve is a law of large numbers for the joint distribution of part sizes in a random integer partition.
Our definition of a limit shape, made precise in \S \ref{diagram functions}, is a limit \emph{in probability} as $n\to \infty$, of appropriately scaled part sizes of a uniformly random integer partition of size~$n$.

The main technique for finding and proving the form of the limit shape is via a probabilistic model.
The joint distribution of part sizes in a random integer partition of size~$n$ is approximated by a certain joint distribution of \emph{independent} random variables, which produces a random integer partition of random size.
A weak notion of a limit shape is the convergence in expectation of the joint distribution of the independent random variables.
It is easy to calculate the limit shape in this setting, as the calculation simplifies to a Riemann sum.
We show in Section~\ref{sect prob} that the weak notion of a limit shape agrees with the usual limit shape in the context of integer partitions.
This is also known as the \emph{equivalence of ensembles}.

We start with two pivotal results in the field of limit shapes of integer partitions: the classical limit shape, and the limit shape of partitions into distinct parts.
There is a long and rather involved history involving limit shapes, which is confounded by the weak and strong notions of limit shapes; see \S \ref{history} for a more detailed account.
We denote the limit shape of unrestricted integer partitions by $\Phi(x)$.

\begin{theorem}[See~\S\ref{history}]\label{unrestricted:limit:shape}
Let $y = \Phi(x)$ denote the limit shape of integer partitions.  Then:
\begin{equation}\label{phi} e^{-c\, x} + e^{-c\,y} = 1, \qquad x>0, \quad \text{where} \quad c = \frac{\pi}{\sqrt{6}}\, . \end{equation}
\end{theorem}

A similar result is true for the limit shape of integer partitions into distinct parts, which we denote by~$\Psi(x)$.

\begin{theorem}[See~\S\ref{history}]
Let $y = \Psi(x)$ denote the limit shape of integer partitions into distinct parts.  Then:
\begin{equation}\label{psi}  e^{d\, y} - e^{-d\, x} = 1, \qquad x>0, \quad \text{ where } \quad d = \frac{\pi}{\sqrt{12}}\, . \end{equation}
\end{theorem}

One can also place certain types of restrictions on the part sizes in an integer partition, and ask whether or not a limit shape exists.
We assume the part sizes are restricted to a set $U = \{u_1, u_2, \ldots\}$, where $u_k \sim B\ts k^r$ for some $r \geq 1$ and $B>0$ (subject to other technical assumptions, see~\S\ref{credit}).
Then, analogously as before, most partitions of size~$n$ are close to a certain curve, which satisfies
\[
\frac{y'}{n^{\frac{1}{1+r}}} \. = \.
\frac{\bigl(x \ts n^{\frac{-r}{1+r}}\bigr)^{\frac{1}{r}-1}}{r\ts B^{\frac{1}{r}}}\. \cdot \.
\frac{e^{-c\, \bigl(x \ts n^{\frac{-r}{1+r}}\bigr)}}{1-e^{-c\, \bigl(x \ts n^{\frac{-r}{1+r}}\bigr)}}\, , \quad \text{where} \quad x>0,\ y>0,
\]
and $c$ is a constant which makes the area under the curve equal to one.
In this case, the scaling of the $x$-axis is by $n^{\frac{r}{1+r}}$ and the scaling of the $y$-axis is by $n^{\frac{1}{1+r}}$.

\begin{theorem}[See~\S \ref{history}]\label{rB:theorem}
Let $y = \Phi_{r,B}(x)$ denote the limit shape of integer partitions with part sizes restricted to $u_1, u_2, \ldots$, with $u_k \sim B\ts k^r$ for some $r \geq 1$ and $B>0$.
Denote by $y'$ the derivative of $y$ with respect to $x$.  Then:
\begin{equation}\label{eq:unrestricted}
y' \. = \. \frac{x^{\frac{1}{r}-1}}{r\ts B^{\frac{1}{r}}}\. \cdot \. \frac{e^{-c\, x}}{1-e^{-c\, x}}, \qquad x>0,
\end{equation}
where $c$ is given by equation~\eqref{Uber c}.
\end{theorem}

A similar generalization holds for partitions into distinct parts from a set $u_1, u_2, \ldots$\ .
An alternative parameterization is by Andrews~\cite{Andrews}, who defined the restrictions on part sizes of an integer partition by a sequence $(a_1, a_2, \ldots)$, where the multiplicity of the parts of size$~i$ is restricted to be strictly less than $a_i$.
For example, the set of unrestricted partitions corresponds to the sequence $(\infty, \infty, \ldots)$, and partitions into distinct parts corresponds to $(2,2,\ldots)$.
When parts sizes are restricted to $u_1, u_2, \ldots$\. , this corresponds to the sequence $a_{u_i} = \infty$, $i \geq 1$, and $a_j = 1$ for $j \notin \{u_1, u_2, \ldots \}$.

There are, however, certain restrictions on integer partitions for which a limit shape does not exist.
For example, if we consider the sequence $a_{2^i} = 2$, $i \geq 1$, and $a_j = 0$ if $j$ is not a power of 2, then there is exactly one integer partition for each $n$; namely, the base-2 representation of the number $n$.
For another example, let $a_i = \infty$ for $i=1,\ldots,k$, and $a_i = 1$ for $i > k$ for some positive $k$.
This sequence corresponds to partitions of $n$ with largest part at most $k$.
For $k$ fixed, the number of such partitions only grows polynomially with $n$, and there is no limit shape.
In contrast, the case  $k = t\sqrt{n}$ was handled by Romik~\cite{Romik}, as well as Vershik and Yakubovich~\cite{VershikYakubovich}, and the limit shape does indeed exist.

There are other classes of partitions which cannot be described by Andrews's parameterization.
In particular, for the example of convex partitions described above, it is not immediately
clear whether or not a limit shape exists, or if it does, what is the proper scaling.
Nevertheless, we show in~\S\ref{convex} that indeed the limit shape of convex partitions
exists and demonstrate how to compute it.

\subsection{New results}
The goal of this paper is to obtain new limit shapes for integer partitions with restrictions which do not fit into the usual framework described above.
The form of the restrictions given by Andrews is a subset of restrictions called \emph{multiplicative restrictions} by Vershik~\cite{Vershik}, owing to the form of the generating function as a product of terms; it also refers to the fact that under the weak form of the limit shape, part sizes are treated as independent, and so probabilities factor according to part sizes.

Our first example is an application of the various geometric transformations defined in~\cite{PakNature}.

\begin{cor}
Let $\mathcal{L}$ denote the set of partitions $\mu$ with consecutive parts $\mu_i - \mu_{i-1} \geq 2$ for even part sizes $\mu_i$, and $\mu_i - \mu_{i-1} \geq 4$ for all odd part sizes $\mu_i$, $i \geq 1.$
Let $m(x)$ denote the limit shape of~$\mathcal{L}$, and let $w = e^{\frac{\pi}{2} m(x)}$, and $u = e^{\frac{\pi}{4}x}$.  Then the limit shape satisfies
\[ u \. = \, \frac{w+1}{w^2-w}\. . \]
\end{cor}

Let $\mathcal{R}$ denote the set of partitions into distinct parts which are congruent to~$0, 1,$ or~$2$ modulo~$4$.
We demonstrate in Section~\ref{Lebesgue} how to apply these transformations to obtain the limit shape of~$\mathcal{L}$.
Moreover, once we have the limit shape, certain statistics immediately follow.
\begin{cor}\label{number:of:parts:cor}
Let $k_n$ denote the number of parts in a random partition of size~$n$ in $\mathcal{L}$.
We have
\begin{equation}\label{number:of:parts}
\qquad \qquad   \frac{k_n}{\sqrt{n}} \, \longrightarrow_P \, \ts \frac{2 \log\left(1+\sqrt{2}\right)}{\pi} \, = \. 0.561099852 \ldots \. ,
% 0.5610998523
\end{equation}
where $\ts\rightarrow_P\ts$ denotes convergence in probability.
\end{cor}

The Durfee square of a partition $\lambda$ is the number of $\lambda_j$ such that $\lambda_j \geq j$.
It is a well-known statistic which plays a fundamental role in some of the bijections exploited in this paper, and is also described as the largest square which can fit entirely inside the Young diagram of an integer partition.
\begin{cor}\label{durfee:square:cor}
Let $\delta_n$ denote the size of the largest Durfee square in a random partition of size~$n$ in $\mathcal{L}$.
Let \. $y_\circ=4.171195932\ldots$ \ts denote the real--valued solution to
\[ -1+2y-9y^2-7y^3-2y^4+y^5 \ts = \ts 0\ts .
\]
Then we have:
%\[ z = 4.17119593 \ldots, \]
% 4.171195932
\begin{equation}\label{durfee:square}
\qquad  \frac{\delta_n}{\sqrt{n}}\. \ts \longrightarrow_P \.\ts
\frac{4}{\pi } \. \log \ts\frac{1}{14} \bigl(5-30 y_\circ-24 y_\circ^2-9 y_\circ^3+4 y_\circ^4\bigr)
\, = \. 0.454611067 \ldots \.,
% 0.45461106717745470152
\end{equation}
where $\ts\rightarrow_P\ts$ denotes convergence in probability.
\end{cor}

There are, however, certain bijections which are not covered by Pak's geometric transformation approach.
Recall the example of convex partitions, which are partitions which satisfy $\lambda_1 - \lambda_2 \geq \lambda_2 - \lambda_3 \geq \ldots$ \ts
Let $\mathcal{C}_n$ denote the set of convex partitions of size~$n$.
Let $\mathcal{T}_n$ denote the set of partitions of size~$n$ with parts in the set $u_k = \binom{k+1}{2}$, $k \ge 1$, with unrestricted multiplicities.
\begin{theorem}[Andrews~\cite{AndrewsII}]\label{Andrews:convex}
We have $|\mathcal{T}_n| = |\mathcal{C}_n|$ for all $n \geq 1$.
\end{theorem}
The limit shape of $\mathcal{T}$ is a special case of Theorem~\ref{rB:theorem}, using $r=2$ and $B = \frac{1}{2}$.
It is thus natural to ask if it is possible to obtain the limit shape of $\mathcal{C}$ from the limit shape of $\mathcal{T}$.
The answer is affirmative, as witnessed in Section~\ref{sect r}, \emph{due to the form of the bijection}.
In fact, Corollary~\ref{convex:corollary} below is a special case of a much larger class of partition bijections which map limit shapes continuously, which we describe in \S\ref{transfer:theorems}.

\begin{cor}\label{convex:corollary}
Let $T(x)$ denote the limit shape of~$\mathcal{T}$, and let $C(x)$ denote the the limit shape of~$\mathcal{C}$.  Then we have
\begin{equation}\label{convex:relation}
e^{2\, C^{(-1)}(x)}\,\left(e^{\frac{1}{2}x^2} - 1\right) =\, e^{-2\ts x\, T\left(\frac{x^2}{2}\right)}, \qquad x>0,
\end{equation}
where $C^{(-1)}$ is the compositional inverse of $C$.
\end{cor}

Romik's unpublished manuscript~\cite{RomikUnpublished} contains an example which uses a single geometric transformation to obtain the limit shape of the set of partitions whose parts differ by at least 2.
It is a specialization of Pak's geometric transformations, and requires a further refinement, in that it maps partitions of size~$n$ with exactly $k$ parts.

\begin{cor}[cf.~\S \ref{Romik:example}]
Let $\mathcal{B}$ denote the set of partitions whose parts differ by at least 2, and denote its limit shape by $B(x)$.  For $x>0$, let $w = e^{c\ts B(x)}$, $u = e^{-c\ts x}$, and $v = e^{c\ts\gamma}$,  where $\gamma = \frac{1+\sqrt{5}}{2}$, and $c$ is a normalizing constant.  Then, the limit shape satisfies
\[ u \. = \,
\frac{v^2 - v}{w^2-w}\..
\]
\end{cor}

\subsection{Contents of the paper}

The rest of the paper is organized as follows.
We start with three motivating examples: partitions with no consecutive parts, self-conjugate partitions, and convex partitions.
Then we make precise all of the necessary notions of limit shapes, bijections, and asymptotic stability in Section~\ref{sect notation}, and state the new limit shape theorems in Section~\ref{sect special cases}.

We then present several examples demonstrating the application of the theorems.
In Section~\ref{sect r}, we present a generalization of convex partitions.
Section~\ref{sect odd} presents various applications of the theorems to Euler's classical odd-distinct bijection, including a generalization due to Stanton.
Sections~\ref{Lebesgue}~and~\ref{diff d} demonstrate step-by-step how to apply Pak's geometric transformations in two important cases.
Section~\ref{sect further} contains examples where the number of summands is restricted.

Section~\ref{sect prob} contains the formal probabilistic setting and proofs of the main lemmas, including a large deviation principle.
Section~\ref{sect main} contains the proof of the main results.
Finally, Section~\ref{sect remarks} contains historical remarks, and Section~\ref{final:remarks} is a collection of final remarks.

% Section: Three motivating examples
\bigskip\section{Three motivating examples}
\label{sect examples}

% Partitions with no consecutive parts
\subsection{Partitions with no consecutive parts}
\label{Romik:example}

We first consider Romik's example, which demonstrates the utility of considering bijections between classes of partitions.

\begin{theorem}[\cite{RomikUnpublished}]\label{Romik:theorem}
Let $\mathcal{A}$ denote the set of partitions such that no parts differ by exactly 1, and there are no parts of size 1.
Let $a = \frac{\pi}{3}$, and let $A(x)$ denote the limit shape of $\mathcal{A}$.
Then we have
\[ A(x) = \frac{1}{2a}\log\left(\frac{1+e^{-a\,x}+\sqrt{1+2e^{-a\,x}-3e^{-2a\,x}}}{2(1-e^{-a\,x)}}\right), \quad \text{where} \quad x>0. \]
\end{theorem}

The \emph{existence} of the limit shape of $\mathcal{A}$ in Theorem~\ref{Romik:theorem} follows by considering the set of partitions which are conjugate to those in $\mathcal{A}$.
\begin{lemma}
Let $\mathcal{A}_n$ denote the set of partitions of size~$n$ such that no parts differ by exactly 1, and there are no parts of size 1.
Let $\mathcal{B}_n$ denote the set of partitions of size~$n$ such that no part has multiplicity 1.
The conjugation map gives a bijection between partitions $\lambda \in \mathcal{A}_n$ and $\mu \in \mathcal{B}_n$ for each $n \geq 1$.
\end{lemma}

The limit shape of $\mathcal{B}$, the set of partitions such that no part has multiplicity~1,
has a restriction which is multiplicative, and its form follows in a straightforward manner
from Theorem~\ref{special unrestricted}.

\begin{lemma}[cf.~\S\ref{history},\.\S\ref{natural:generalization}]
Let $\mathcal{B}$ denote the set of partitions such that no part has multiplicity 1.
Let $a = \frac{\pi}{3}$.
Then:
\[ B(x) = \frac{1}{a}\log\left(\frac{1-e^{-a\,x}+e^{-2a\,x}}{1-e^{-a\,x}}\right), \quad \text{where} \quad x>0. \]
\end{lemma}

As was noted by Romik~\cite{RomikUnpublished}, even though the limit shapes satisfy $(A \circ B)(x) = x$, $x >0$, it is not straightforward to compute the explicit form of $A(x)$ using algebraic techniques, even with a computer algebra system.

% Self-conjugate partitions
\subsection{Self-conjugate partitions}

A general approach of finding limit shapes via a class of continuous transformations was developed in~\cite{PakNature}, where many well-known partition bijections are presented as certain geometric transformations acting on Young diagrams, which also act continuously on the corresponding limit shapes.

An example where the limit shape is apparent in a weak sense, but lacks a rigorous argument, is the set of \emph{self-conjugate partitions} $\mathcal{S}$, consisting of all partitions $\lambda$ such that $\lambda \in \mathcal{S}$ implies also $\lambda' \in \mathcal{S}$.

\begin{theorem}[cf.~\S \ref{history}]\label{Sylvester}
Let $\mathcal{S}$ denote the set of partitions which are self-conjugate.  The limit shape of $\mathcal{S}$ satisfies equation~\eqref{phi}.
\end{theorem}

One can reason heuristically that the limit shape of $\mathcal{S}$ should coincide with the limit shape of unrestricted integer partitions, however, we are unaware of any simple, rigorous methods to do so.
The first difficulty is that the number of self-conjugate partitions of size~$n$ is of an exponentially smaller proportion than the number of unrestricted partitions of size~$n$.
The second difficulty is that we lose the asymptotic independence between part sizes.

There are several different ways to obtain a rigorous proof of Theorem~\ref{Sylvester}, starting with the following bijection.

\begin{theorem}[{\cite[Prop.~7.1]{PakNature}}]
Let $\mathcal{A}_n$ denote the number of partitions of $n$ into odd, distinct parts.  Let $\mathcal{S}_n$ denote the number of self-conjugate partitions of $n$.
We have $|\mathcal{A}_n| = |\mathcal{S}_n|$ for all $n \geq 1$.
\end{theorem}

One construction of the bijection is to break the principle hooks (see~\cite[Figure~13]{PakNature}) and stretch them out.
When the partition $\lambda$ is self-conjugate, each of the principle hooks must necessarily consist of an odd number of squares, and each principle hook consists of at least two squares more than the previous principle hook.
In this instance, we have an explicit formula for the map which sends partitions $\lambda \in \mathcal{S}$ to partitions $\mu$ into distinct, odd parts, namely,
\begin{equation}\label{hooks}\mu_i = 2(\lambda_i -i) + 1, \qquad 1 \leq i \leq \delta_0, \end{equation}
where $\delta_0$ is the size of the largest Durfee square in the partition $\lambda$.

A formal procedure to obtain Theorem~\ref{Sylvester} from this bijection is as follows, see~\cite[Section~7]{PakNature}.
For the remainder of this section, $c = \pi / \sqrt{6}$.
Let $b(t)$ denote the limit shape of partitions into odd, distinct parts, which is given by Theorem~\ref{rB:theorem} using $r = 1$ and $B = 2$.
We shall need the explicit form of the limit shape in the calculations that follow, which is given  by
\[ b(t) = \sqrt{2}\, \Psi\left(\sqrt{2}\, t\right) = \frac{1}{c} \log\left(1+e^{-c\, t}\right), \qquad t>0\, . \]
Recall also that the limit shape of unrestricted integer partitions is given by the explicit formula
\[ \qquad \qquad \qquad \Phi(t) = -\frac{1}{c} \log(1-e^{-c\, t}), \qquad t>0\, . \]

Denote by $a(t)$ the limit shape of self--conjugate partitions, which exists by \cite[Prop.~7.1]{PakNature}.
After an appropriate scaling by $\sqrt{n}$, by Equation~\eqref{hooks}  we have
\[\begin{array}{ccccc}
 \frac{1}{\sqrt{n}}\mu_{t\, \sqrt{n}} & = &  \frac{2}{\sqrt{n}}(\lambda_{t\, \sqrt{n}} - t\, \sqrt{n}) + 1/\sqrt{n}, & & 1 \leq t \sqrt{n} \leq \frac{\log(2)}{\sqrt{2}\, d} \sqrt{n}.   \\
 \downarrow & & \downarrow & & \downarrow\\
 b^{-1}(t) & = & 2(a^{-1}(t)-t), & & 0 \leq t \leq \frac{\log(2)}{c}.
\end{array}
\]
By rearranging and solving for $a^{-1}(t)$, we obtain $a^{-1}(t) = \Phi(t)$ for $0\leq t \leq \log(2)/c$.  The value $\log(2)/c$ is precisely the value for which $\Phi(t) = t$, $t>0$, which cuts the limit shape in half via the line $y=t$.  Thus, for $t > \log(2)/c$, we can invoke the symmetry of self--conjugate partitions about the line $y=t$ and conclude that $a(t) = \Phi(t)$ for all $t>0$.

This solution is uncharacteristically simple, and most bijections are not as obliging as to allow for elementary algebraic manipulations of their part sizes.
Fortunately, this bijection can also be defined by a natural geometric bijection, and we next demonstrate each step in the process of mapping the limit shape from one set of partitions to another.
Figure~\ref{parallel:mapping:Young} demonstrates how a geometric bijection from \cite{PakNature} acts on the Young diagrams, and Figure~\ref{parallel:mapping:limit} demonstrates how a geometric bijection acts on the limit shapes. %corresponds to a similar set of mappings of the limit shape.

\ignore{
\begin{figure}
\includegraphics[scale=0.1, angle=-90]{transform.jpg}
\caption{A bijection via a sequence of geometric bijections and the corresponding operations on the limit shape.}
\label{parallel mapping}
\end{figure}}

\ignore{
\begin{figure}
\[
\resizebox{15em}{.3em}{
\xymatrix{
& \Phi(t) & \\
&\ar[dl]^{P\Phi(t)} \includegraphics[scale=0.2]{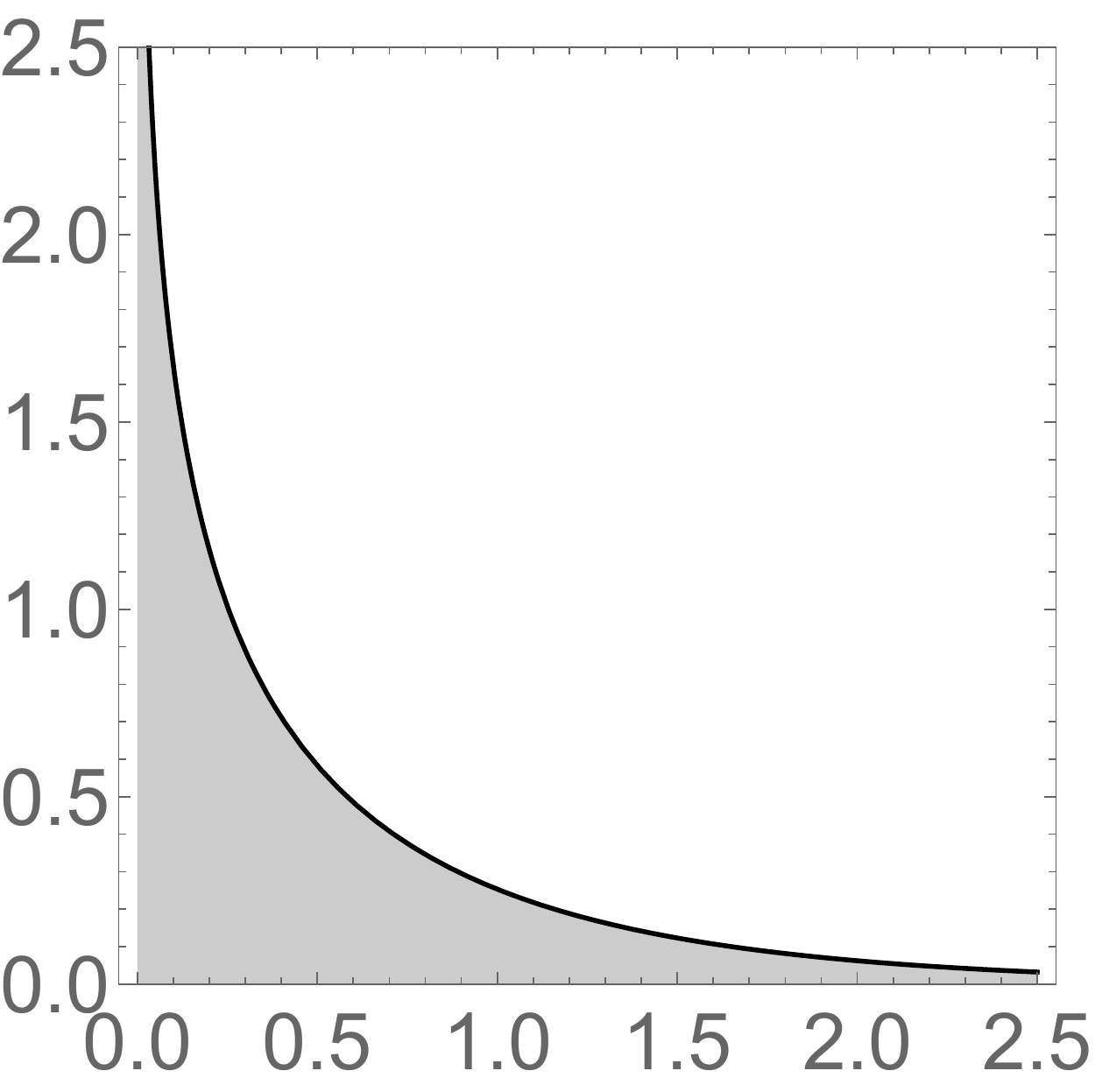} \ar[dr]_{Q \Phi(t)} & \\
\ar[d]^{\Phi_1(t) := P\Phi(t)-t}\includegraphics[scale=0.2]{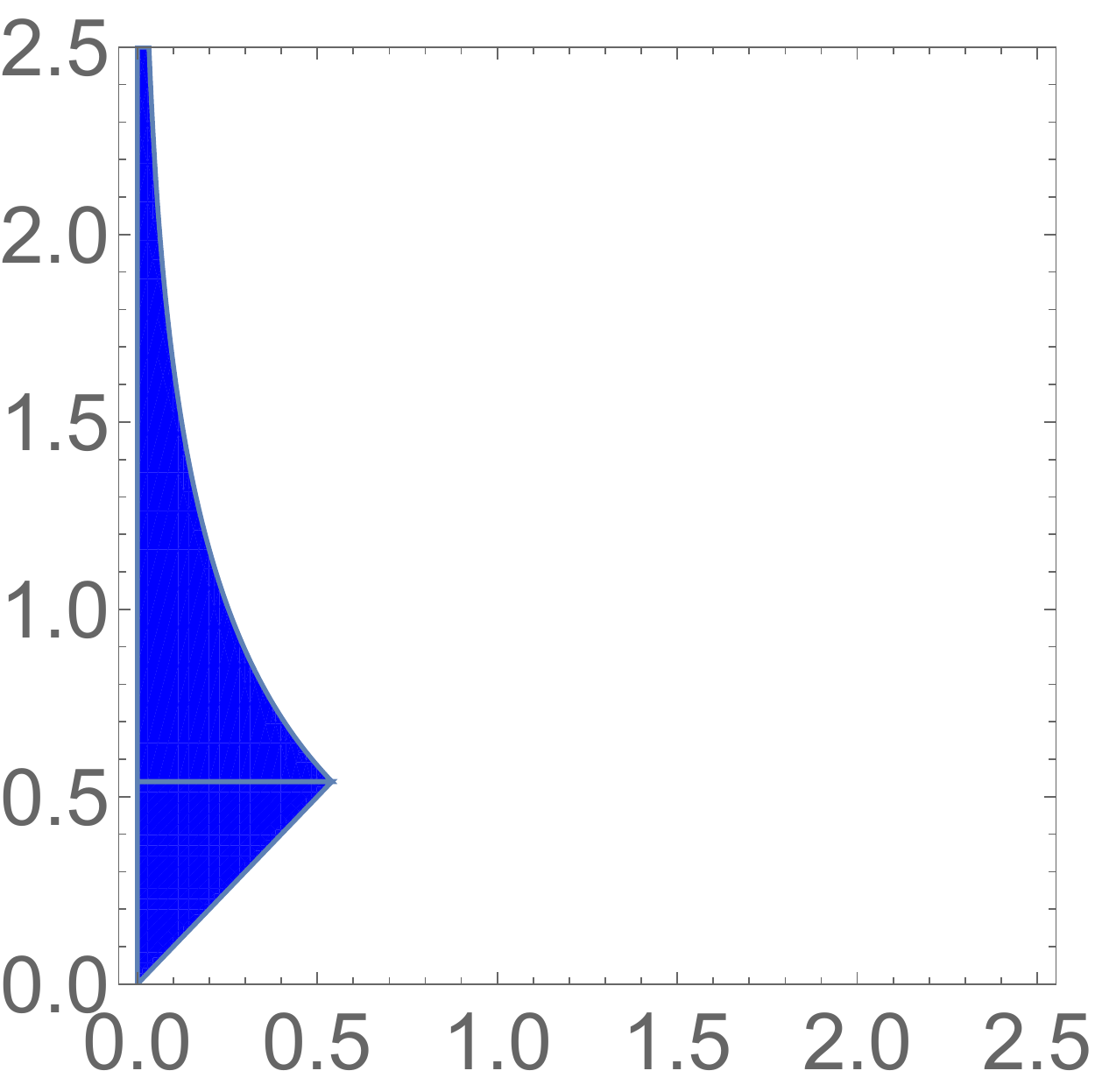} & & \includegraphics[scale=0.2]{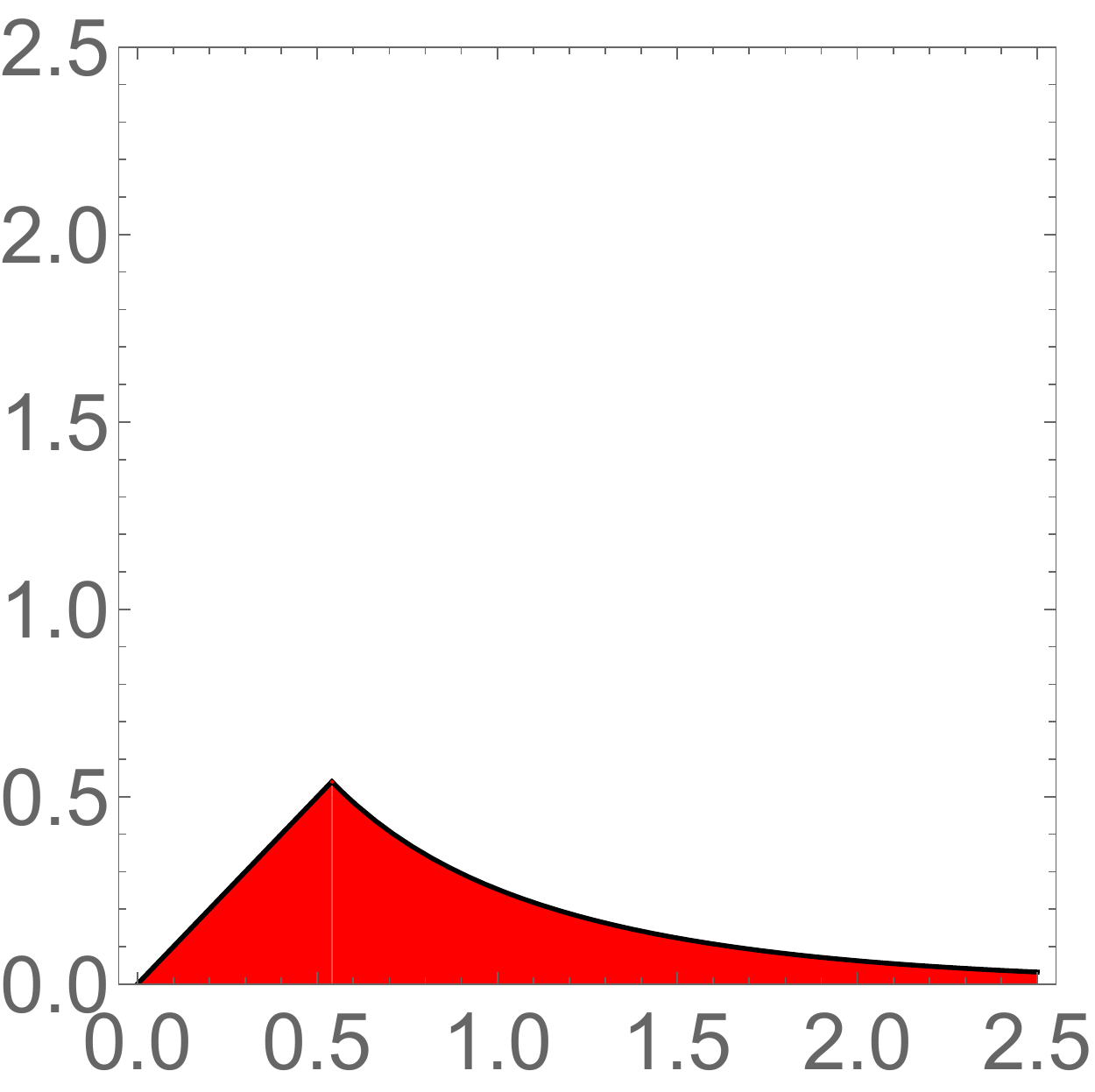}\ar[d]^{\Phi_2(t) := \left((Q\Phi(t))^{\<-1\>} - t\right)^{\<-1\>}} \\
\ar[ddr]\includegraphics[scale=0.2]{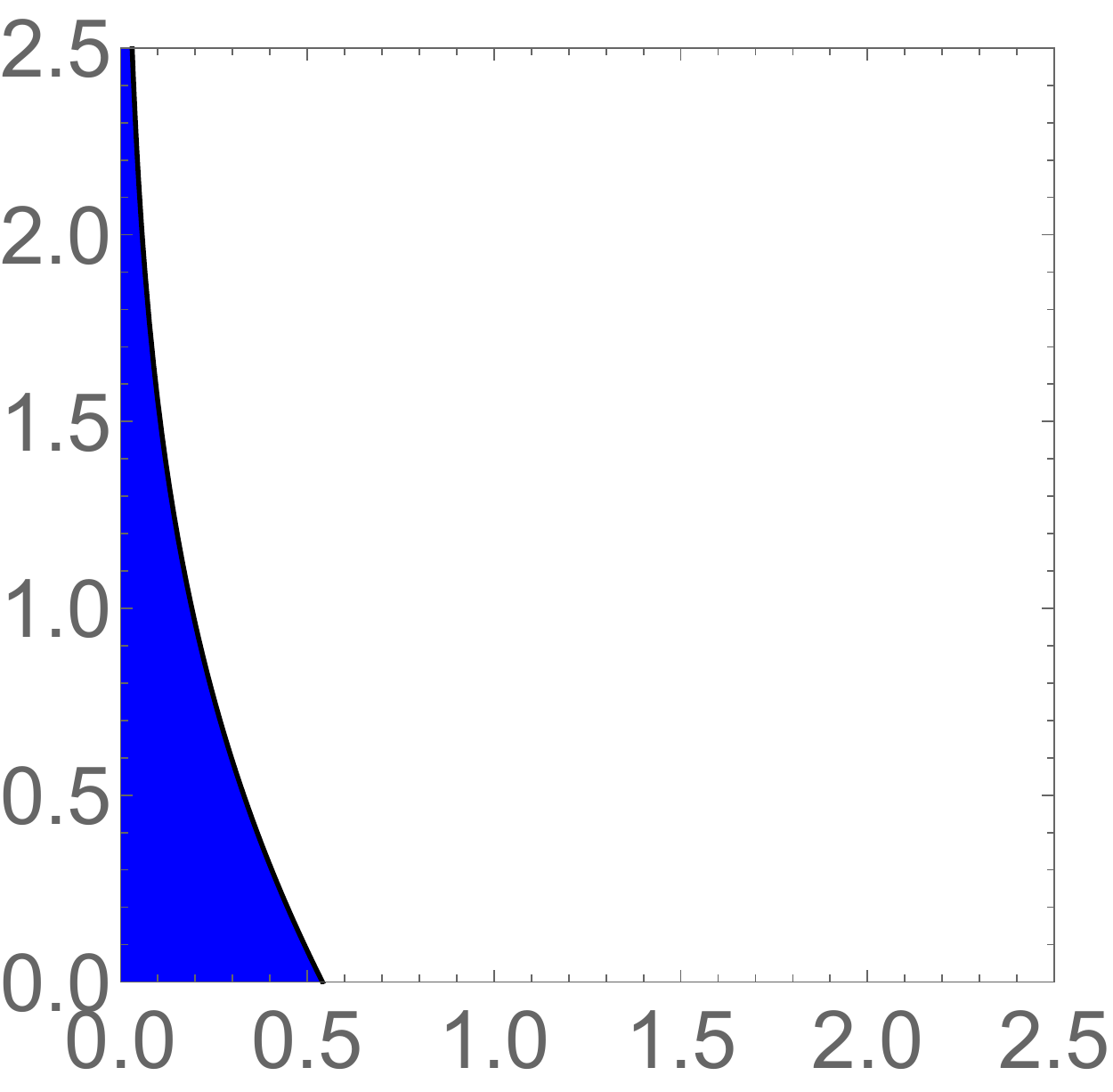} & & \includegraphics[scale=0.2]{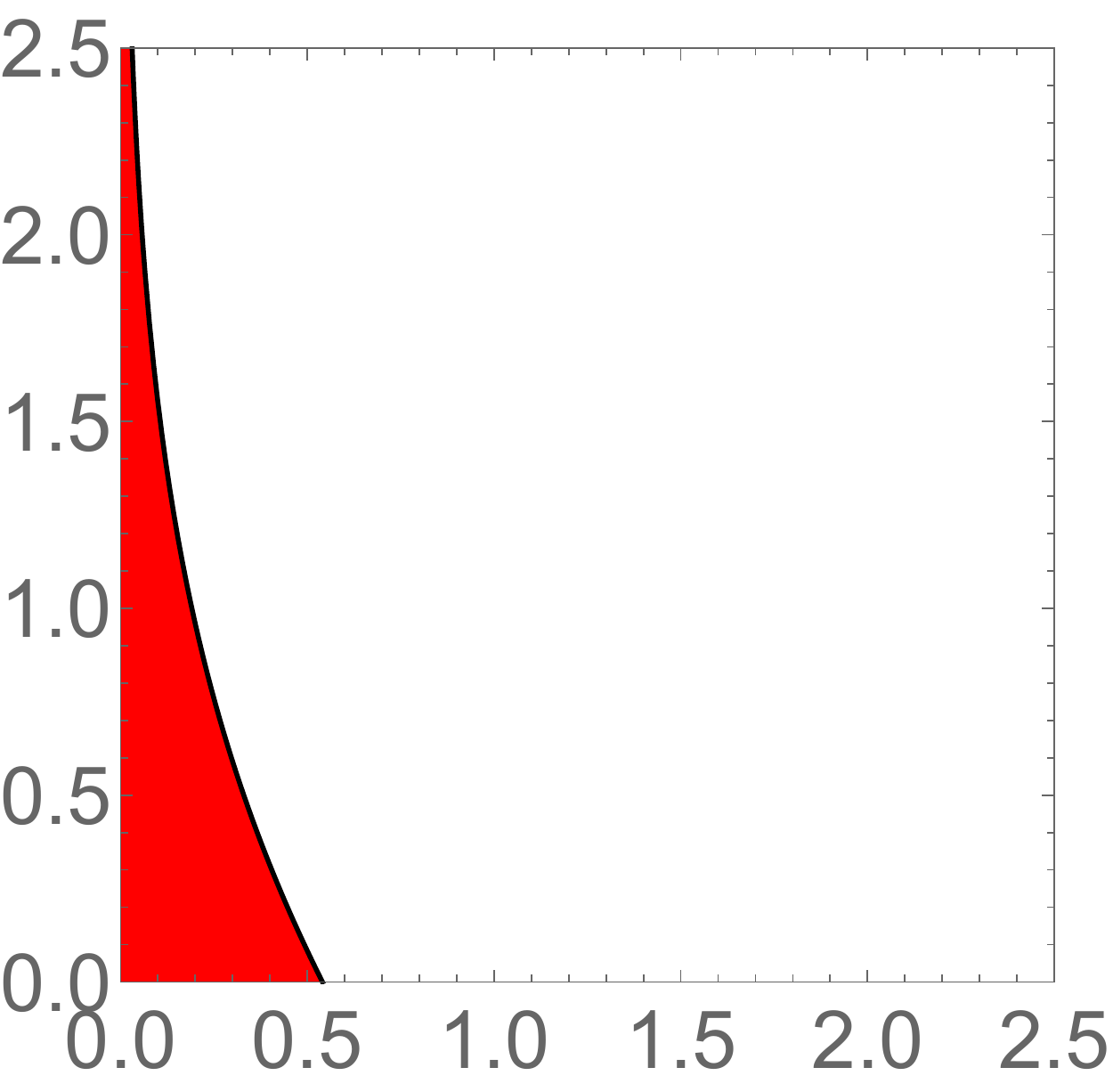}\ar[ddl]  \\
%& {\frac{(P\Phi(t)-t) + \left((Q\Phi(t))^{-1} - t\right)^{-1}}{2}} &  \\
& {\Phi_1(t) + \Phi_2(t)} &  \\
& \mbox{$\begin{array}{c} \includegraphics[scale=0.2]{phi4-2.pdf} \\ \frac{1}{\sqrt{2}}\Psi\left(\frac{t}{\sqrt{2}}\right) \end{array}$ } &
}
}
\]
\caption{The bijection between odd, distinct partitions and self-conjugate partitions as a sequence of geometric transformations acting on the limit shape.}
\label{parallel:mapping:limit}
\end{figure}
}

\begin{figure*}[t]
	\begin{subfigure}[t]{0.3 \textwidth}
\[
%\resizebox{15em}{.267em}{
\hskip -1in\scalebox{.55}{
\xymatrix{
&\ar[dl] \mbox{\begin{tabular}{c}\includegraphics[scale=0.282]{phi1.pdf} \\  \mbox{\Large $\Phi(t)$} \end{tabular}}  \ar[dr] & \\
%\mbox{\Huge $P\Phi(t)$} & & \mbox{\Huge $Q\Phi(t)$} \\
\ar[d]\mbox{\begin{tabular}{c}\includegraphics[scale=0.282]{phi2b.pdf}\\\mbox{\Large $P\Phi(t)$} \end{tabular}} & & \mbox{\begin{tabular}{c}\includegraphics[scale=0.282]{phi2a.pdf}\\\mbox{\Large $Q\Phi(t)$}\end{tabular}} \ar[d] \\
\ar[ddr]\mbox{\begin{tabular}{c} \mbox{\Large $P\Phi(t)-t$}  \\ \includegraphics[scale=0.282]{phi3b.pdf}\\  \mbox{\Large $\Phi_1(t)$}\end{tabular}} & &
\mbox{\begin{tabular}{c} \mbox{\Large $\left((Q\Phi(t))^{\langle -1\rangle} - t\right)^{\langle -1\rangle}$} \\ \includegraphics[scale=0.282]{phi3a.pdf} \\  \mbox{\Large $\Phi_2(t)$}  \end{tabular}}   \ar[ddl]  \\
%& {\frac{(P\Phi(t)-t) + \left((Q\Phi(t))^{-1} - t\right)^{-1}}{2}} &  \\
& \mbox{\Large $\Phi_1(t) + \Phi_2(t)$} &  \\
& \mbox{$\begin{array}{c} \includegraphics[scale=0.282]{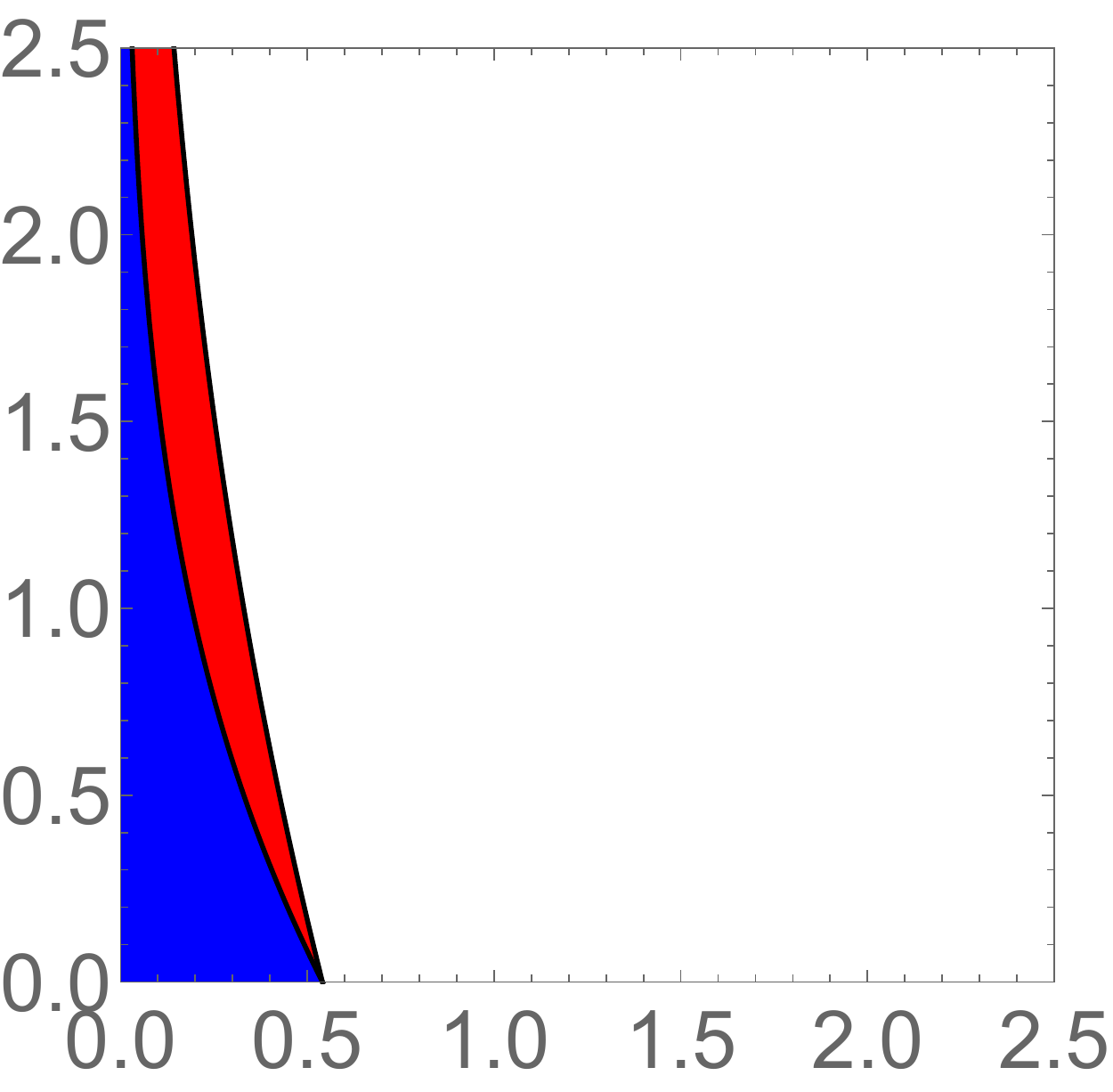} \\ \mbox{\Large $\frac{1}{\sqrt{2}}\Psi\left(\frac{t}{\sqrt{2}}\right)$} \end{array}$ } &
}
}
\]
\caption{Geometric transformations acting on the limit shape.}
\label{parallel:mapping:limit}
\end{subfigure}
\hskip .9in
	\begin{subfigure}[t]{0.3 \textwidth}
\[
%\resizebox{20em}{.5em}{
\scalebox{0.4}{
\xymatrix@C=1em@R=1em{
&\ar[dl]
\mbox{\begin{tabular}{c}\ydiagram[*(red)]
{1+0,2+0,2+0,2+0,2+1,1+5,0+7}
*[*(blue)]{1,2,2,2,2,2,1} \\ \\  \mbox{\Huge Cut} \end{tabular}} \ar[dr]& \\
\ydiagram[*(blue)]{1,2,2,2,2,1} \ar[ddd]|-{\mbox{\Huge Shift}} &  &  \ydiagram[*(red)]{2+1,1+5,7} \ar[ddd]|-{\mbox{\Huge Shift}} \\ \\ \\
\ydiagram[*(blue)]{1,1,2,2,2,2} \ar[ddddd]&  &\ \ \ \ \ \ \  \ydiagram[*(red)]{1,5,7}\ar[dddd]|-{\mbox{\Huge Transpose}}  \hskip .85in \\ \\ \\ \\
 & & \ydiagram[*(red)]{1,1,2,2,2,2,3} \ar@*{[thicker]}[dll] \\
 \mbox{\begin{tabular}{cc} \ydiagram[*(red)]
{0+1,0+1,0+1,0+1,0+2,0+2,0+2,1+1,1+1,2+0,2+0,2+0,2+1}
*[*(blue)]{0,0,0,0,0,0,0,1,1,2,2,2,2} & \mbox{\Huge Add} \end{tabular}} & &
}}
\]
\caption{Geometric transformations acting on Young diagrams rotated $90^\circ$ counterclockwise.}
\label{parallel:mapping:Young}
\end{subfigure}
\caption{The bijection between self-conjugate partitions and partitions 
into distinct odd parts.}
\end{figure*}

\ignore{
\begin{figure}
\[
\resizebox{20em}{.5em}{
\xymatrix{
&\ar[dl] \ydiagram[*(red)]
{7,1+5,2+1}
*[*(blue)]{7,6,3,2,2,2,1} \ar[dr]& \\
\ydiagram[*(blue)]{1,2,2,2,2,1} \ar[d] &  \mbox{\Huge Cut} &  \ydiagram[*(red)]{7,1+5,2+1} \ar[d] \\
\hskip .85in \ydiagram[*(blue)]{2,2,2,2,1,1} \ar[dd]\ \ \ \  \mbox{\Huge Shift}&  & \mbox{\Huge Shift}\ \ \ \  \ydiagram[*(red)]{7,5,1}\ar[dddl] \hskip .85in \\ \\
\hskip 1in \ydiagram[*(blue)]{6,4} \ \ \  \mbox{\Huge Transpose} \ar[dr] & \mbox{\Huge Add}  & \\
&  \ydiagram[*(red)]
{7+6,5+4,1}
*[*(blue)]{13,9,1} &
}}
\]
\caption{The bijection between odd, distinct partitions and self-conjugate partitions as a sequence of geometric transformations acting on the Young diagram.}
\label{parallel:mapping:Young}
\end{figure}
}

\ignore{
\begin{figure}
\[\tiny
\xymatrix{
&\ar[dl] \ydiagram[*(red)]
{7,1+5,2+1}
*[*(blue)]{7,6,3,2,2,2,1} \ar[dr]&  &  &\ar[dl] \includegraphics[scale=0.2]{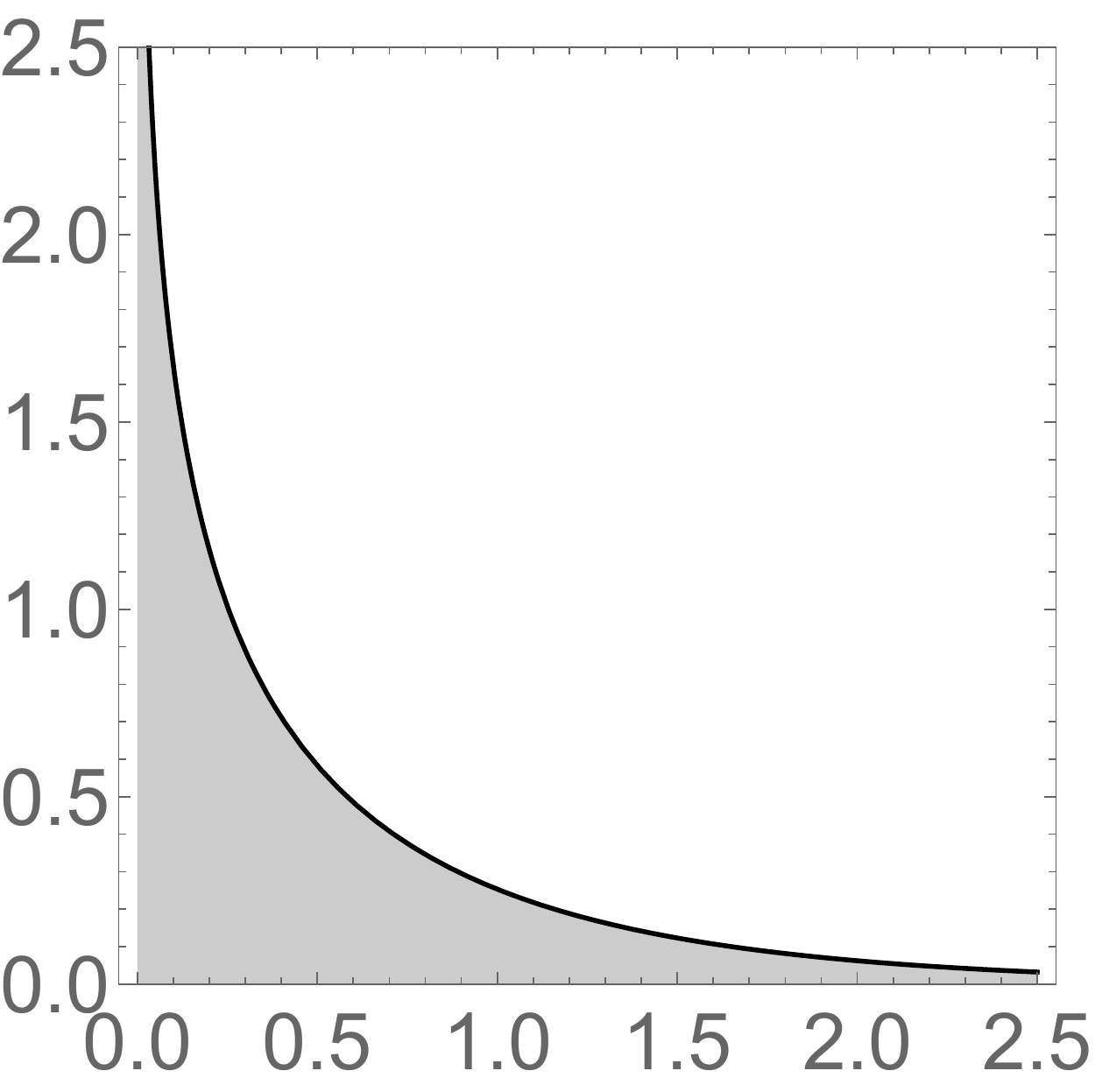} \ar[dr] & \\
\ydiagram[*(blue)]{1,2,2,2,2,1} \ar[d] & &  \ydiagram[*(red)]{7,1+5,2+1} \ar[d]  & \ar[d]\includegraphics[scale=0.2]{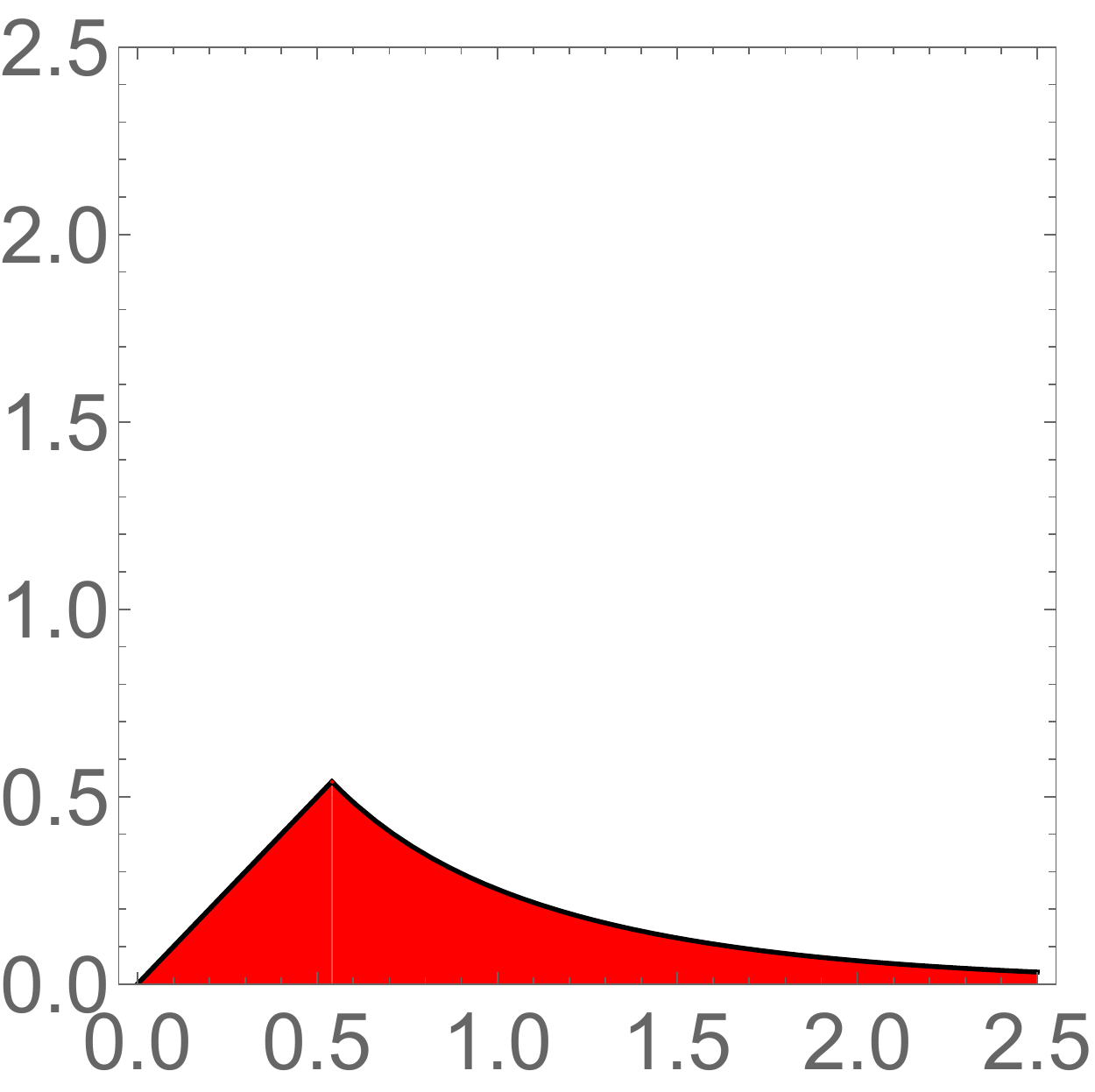} & & \includegraphics[scale=0.2]{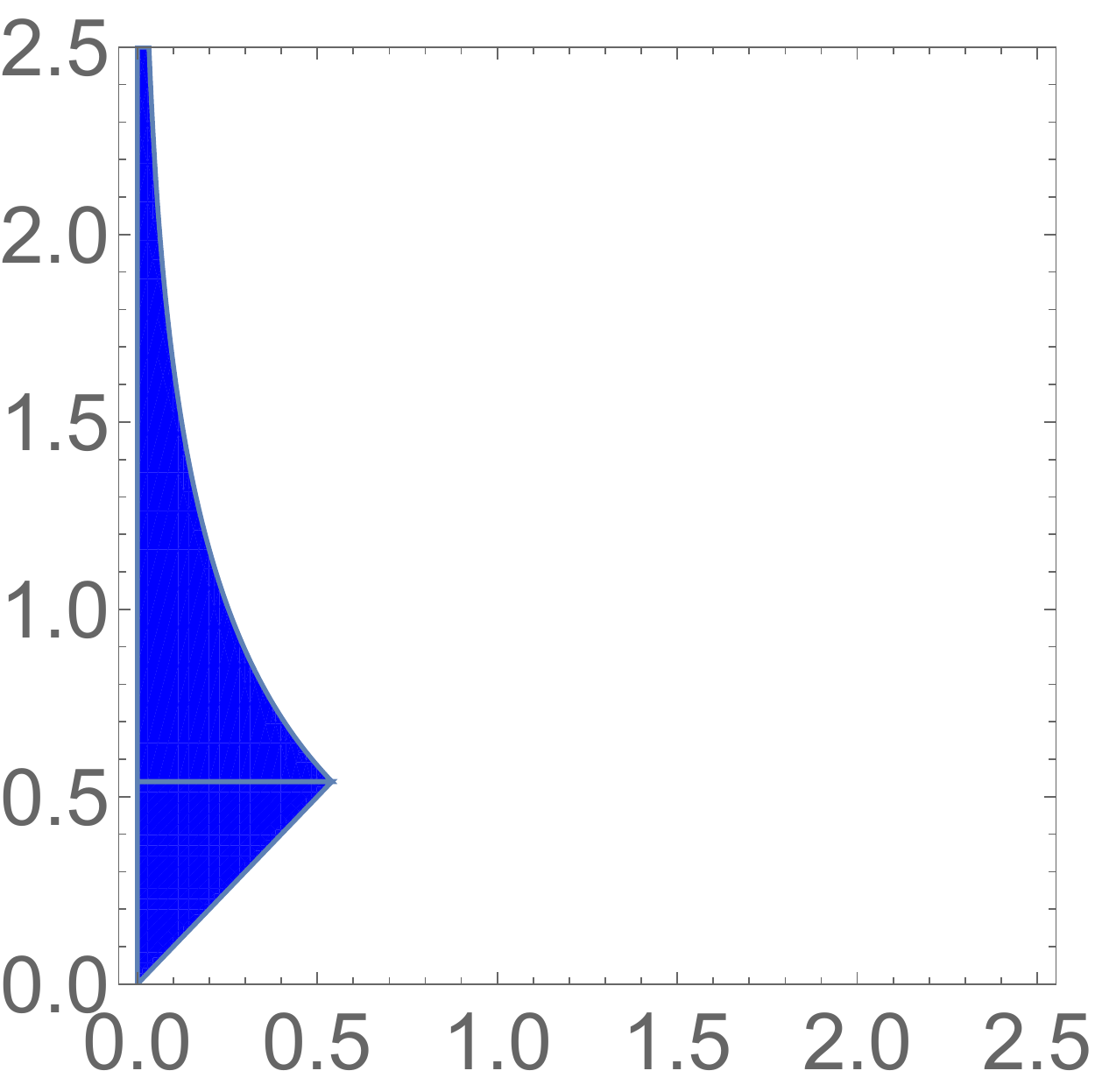}\ar[d] \\
\ydiagram[*(blue)]{2,2,2,2,1,1} \ar[dd] & &  \ydiagram[*(red)]{7,5,1}\ar[dddl]  & \ar[ddr]\includegraphics[scale=0.2]{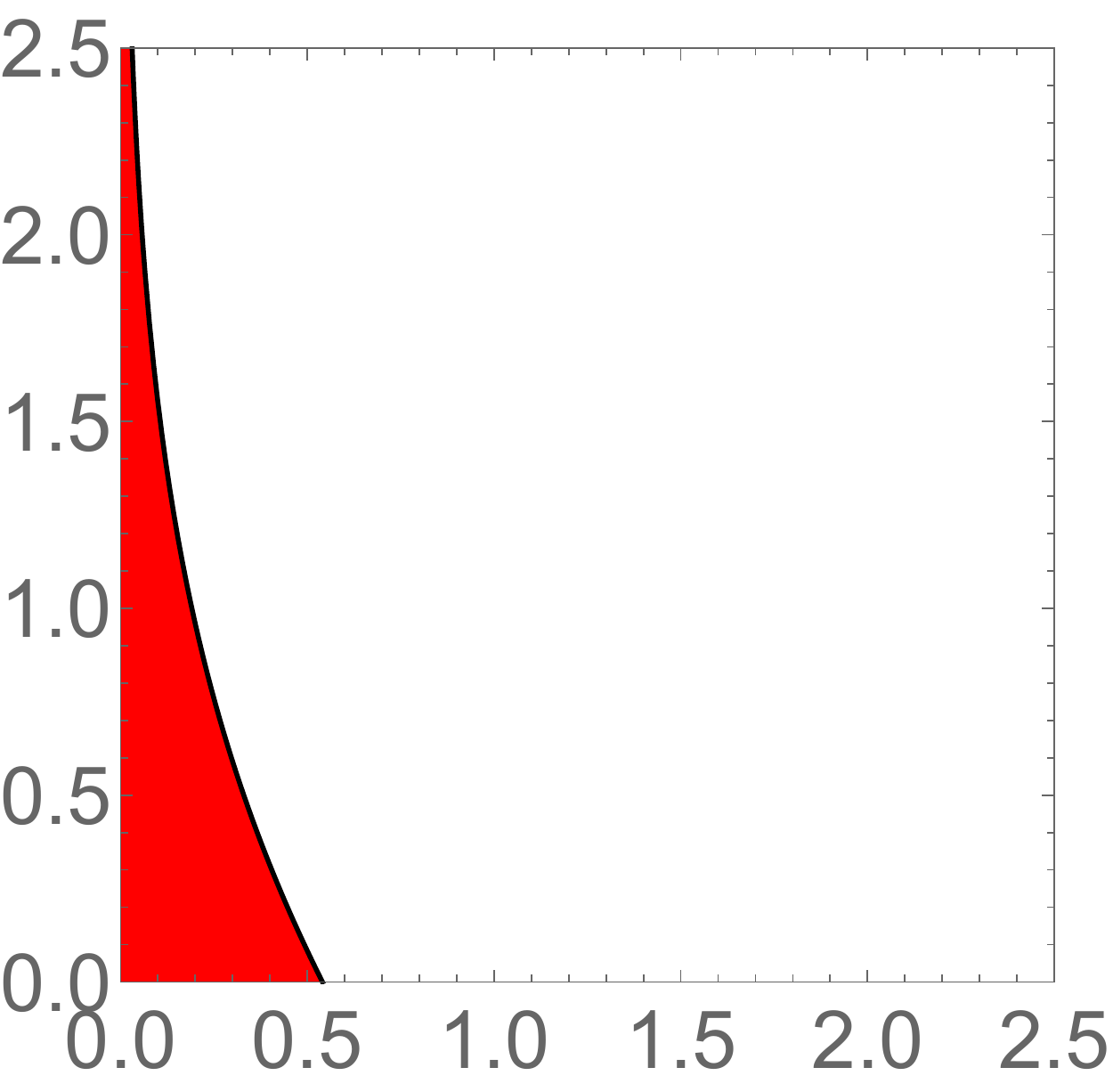} & & \includegraphics[scale=0.2]{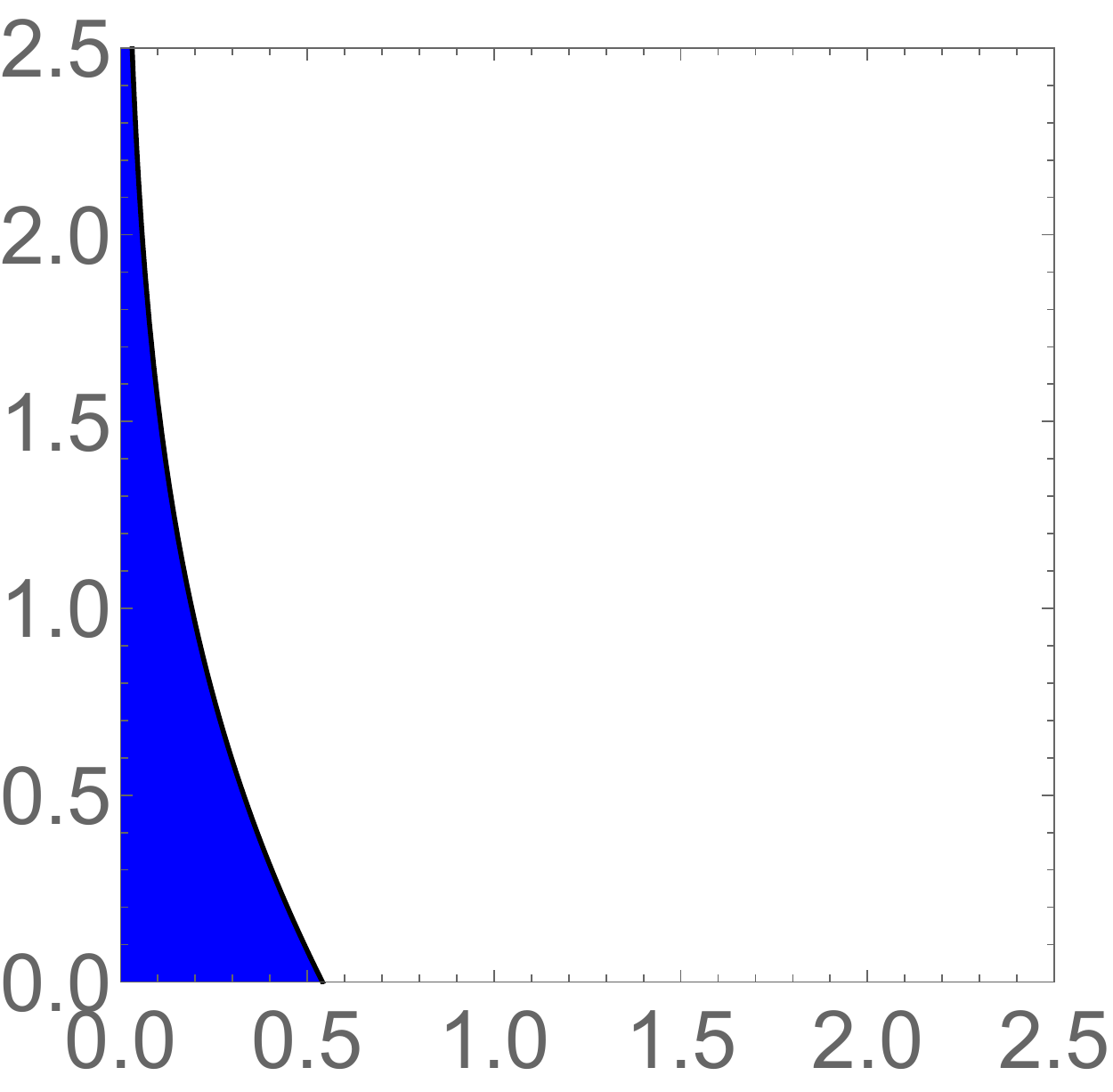}\ar[ddl] \\ \\
\ydiagram[*(blue)]{6,4} \ar[dr] & & & & \includegraphics[scale=0.2]{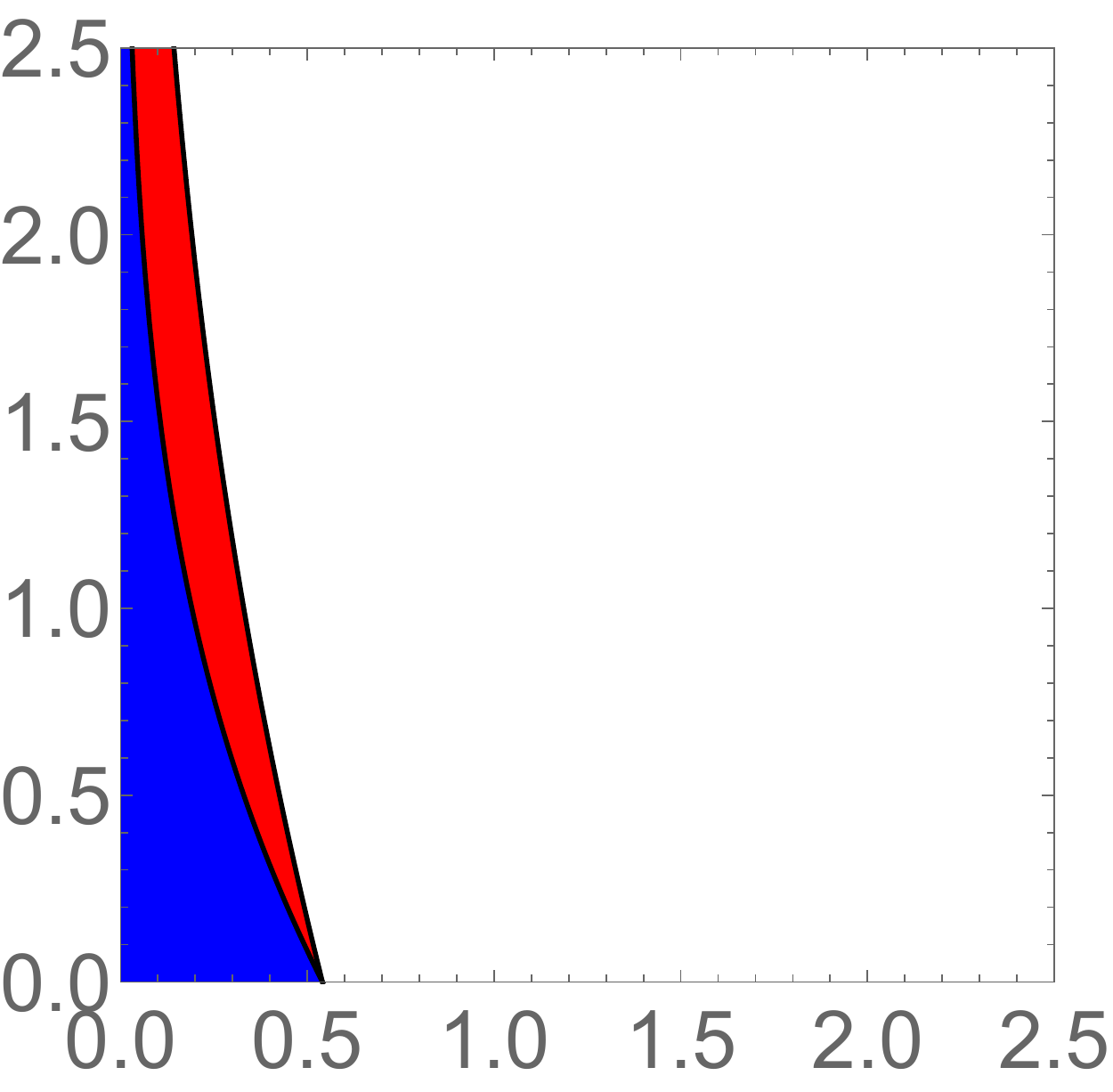} & \\
&  \ydiagram[*(red)]
{7+6,5+4,1}
*[*(blue)]{13,9,1} &
}\]
\end{figure}
}

\ignore{
\begin{figure}
\begin{tikzpicture}[scale=0.2]
\draw (0.1,11.9723)--(0.2,9.67549)--(0.3,8.35232)--(0.4,7.42762)--(0.5,6.72114)--(0.6,6.1526)--(0.7,5.67918)--(0.8,5.27532)--(0.9,4.92451)--(1.,4.61552)--(1.1,4.3403)--(1.2,4.09293)--(1.3,3.8689)--(1.4,3.66472)--(1.5,3.47761)--(1.6,3.30534)--(1.7,3.14607)--(1.8,2.99829)--(1.9,2.86074)--(2.,2.73233)--(2.1,2.61215)--(2.2,2.49942)--(2.3,2.39344)--(2.4,2.29363)--(2.5,2.19947)--(2.6,2.11048)--(2.7,2.02626)--(2.8,1.94646)--(2.9,1.87074)--(3.,1.7988)--(3.1,1.7304)--(3.2,1.66528)--(3.3,1.60323)--(3.4,1.54406)--(3.5,1.48758)--(3.6,1.43364)--(3.7,1.38207)--(3.8,1.33275)--(3.9,1.28554)--(4.,1.24032)--(4.1,1.197)--(4.2,1.15545)--(4.3,1.1156)--(4.4,1.07735)--(4.5,1.04063)--(4.6,1.00535)--(4.7,0.971446)--(4.8,0.93885)--(4.9,0.907501)--(5.,0.87734)--(5.1,0.848312)--(5.2,0.820365)--(5.3,0.79345)--(5.4,0.767522)--(5.5,0.742537)--(5.6,0.718454)--(5.7,0.695236)--(5.8,0.672845)--(5.9,0.651246)--(6.,0.630407)--(6.1,0.610296)--(6.2,0.590885)--(6.3,0.572144)--(6.4,0.554048)--(6.5,0.536571)--(6.6,0.519688)--(6.7,0.503377)--(6.8,0.487616)--(6.9,0.472383)--(7.,0.457659)--(7.1,0.443424)--(7.2,0.429661)--(7.3,0.416352)--(7.4,0.40348)--(7.5,0.39103)--(7.6,0.378985)--(7.7,0.367332)--(7.8,0.356057)--(7.9,0.345145)--(8.,0.334585)--(8.1,0.324363)--(8.2,0.314469)--(8.3,0.30489)--(8.4,0.295616)--(8.5,0.286636)--(8.6,0.27794)--(8.7,0.269518)--(8.8,0.261362)--(8.9,0.253462)--(9.,0.245809)--(9.1,0.238396)--(9.2,0.231214)--(9.3,0.224256)--(9.4,0.217514)--(9.5,0.21098)--(9.6,0.20465)--(9.7,0.198514)--(9.8,0.192568)--(9.9,0.186805)--(10.,0.181219);
\end{tikzpicture}
\end{figure}
}
\ignore{\begin{figure}
\begin{Young}
 & & & & & & &  \cr
 & & & & & & \cr
 &  &   \cr
 & \cr
  & \cr
   & \cr
  \cr
 \end{Young}
\end{figure}}

We summarize the steps of Figure~\ref{parallel:mapping:limit} by the following set of mappings:
\[
\begin{array}{cll}
\frac{1}{\sqrt{2}}\Psi\left(\frac{t}{\sqrt{2}}\right)&  \longmapsto   \left( \frac{\Psi(t\, \sqrt{2})}{\sqrt{2}},  \frac{\Psi(t\, \sqrt{2})}{\sqrt{2}} \right) & (t>0, t>0) \\
 & \longmapsto  \left( \left(\frac{\Psi(t\, \sqrt{2})}{\sqrt{2}}\right)^{\<-1\>}+t,  \left(\frac{\Psi(t\, \sqrt{2})}{\sqrt{2}}\right)^{\<-1\>}+t \right) & (0<t\leq \frac{\ln(2)}{c}, 0 < t \leq \frac{\ln(2)}{c})\\
 & \longmapsto  \left( \left(\frac{\Psi(t\, \sqrt{2})}{\sqrt{2}}\right)^{\<-1\>}+t,  \left(\left(\frac{\Psi(t\, \sqrt{2})}{\sqrt{2}}\right)^{\<-1\>}+t\right)^{\<-1\>} \right) & (0<t\leq \frac{\ln(2)}{c}, t \geq \frac{\ln(2)}{c}) \\
 & \longmapsto  \Phi(t) & (0<t<\infty). \\ 
 \end{array}
 \]
The first map splits the curve into two curves of the same shape but scaled to half the original area.  The next map shifts the coordinates up so that the functions evaluated at the point $x_0 = \frac{\ln(2)}{c}$ are equal to $\Phi(x_0)$.  The third map reflects the first coordinate about the line $y=t$, so that the curves now lie in complementary regions in the first quadrant.
Finally, the last map pastes together the functions at $x_0.$

The mappings defined above are presented in a way which emphasizes how the height functions are transformed with respect to the axes, whereas they are actually transforming regions in the plane.  From a technical point of view, the mappings should be considered as transformations on the plane, as they are presented in \cite{PakNature}, but for the purpose of finding explicit formulas for limit shapes, one can very often more simply track how the mappings transform the height functions and their inverses directly.

% Convex partitions
\subsection{Convex partitions}
\label{convex}
Our last example of this section exemplifies another class of bijections, those which can be realized as a linear transformation of multiplicities of part sizes.

Consider the set $\C$ of all \emph{convex} partitions, i.e., partitions whose parts $\lambda_1 \geq \lambda_2 \geq \ldots \geq \lambda_\ell>0$ also satisfy the convexity condition: $\lambda_1 - \lambda_2 \geq \lambda_2 - \lambda_3 \geq \ldots \geq \lambda_\ell > 0$.  This set of partitions does not exclude any particular set of sizes, nor does there appear to be any simple, a priori transformation, e.g., conjugation, to such a set.
Nevertheless, we find the limit shape of convex partitions.

\begin{theorem}\label{conj convex}
Let $C(x)$ denote the limit shape of convex partitions $\C$.  Then:
\begin{equation}\label{convex:limit:shape}
C^{(-1)}(x) \, = \, \int_{x}^\infty (y-x) \. \frac{e^{-c\, \frac{1}{2} y^2}}{1-e^{-c\, \frac{1}{2} y^2}}\,\. dy, \quad x>0,
\quad \text{where} \ \ c \. = \, \frac{1}{2} \. \pi^{1/3}\th \zeta\left(\frac{3}{2}\right)^{2/3}.
\end{equation}
\end{theorem}

\begin{proof}[Sketch of proof]

Computing this limit shape requires several steps.  First, recall Andrews's Theorem~\ref{Andrews:convex}, that there is a bijection between partitions into parts of sizes $u_k = \binom{k+1}{2}$, $k \geq 1$ and convex partitions.
Simply knowing a bijection exists is not sufficient, however, and it is in fact the particular form of this bijection which allows us to find the limit shape.

Recall $\mathcal{T}$ denotes the set of partitions into parts with sizes in $u_1, u_2, \ldots$\ , with $u_k = \binom{k+1}{2}, k \geq 1$.
Let $b_k$ denote the number of parts of size $\binom{k+1}{2}$ in a partition $\lambda\in\mathcal{T}$, $k \geq 1$.
The following map $\varphi$ was defined in~\cite{AndrewsII}.
\begin{align}
\label{convex bijection}
\varphiv: \ \binom{2}{2}^{b_1} \binom{3}{2}^{b_2} \ldots \binom{k+1}{2}^{b_k} \longrightarrow\ \  &  b_1\ts (1,0,\ldots,0)\ +  \
b_2\ts (2,1,0,0,\ldots 0)\\
\nonumber &\qquad  + \ \ldots \ + \ b_k \ts (k,k-1,k-2,\ldots, 2, 1)\ts.
\end{align}
\begin{lemma}[\cite{Andrews}]
Map $\varphiv$ is a bijection between $\Triangle_n$ and $\C_n$ for each $n \geq 1$.
\end{lemma}

Each $\lambda \in \Triangle_n$, with $b_i$ parts of size~$\binom{i+1}{2}$, corresponds to a $\mu\in\C_n$ with part sizes given by
\begin{align*}
\mu_1 &\, = \, b_1 + 2b_2 + 3b_3 + \ldots + k b_k  \\
\mu_2 &\, = \,  \ \ \ \ \ \ \ \ b_2 + 2b_3 + \ldots + (k-1) b_k \\
\mu_3 &\, = \,  \ \ \ \ \ \ \ \ \ \ \ \ \ \ \ \ b_3 + \ldots + (k-2) b_k \\
\ \, \vdots & \ \ \  \,\ \ \ \ \ \ \ \ \ \ \ \ \ \vdots \\
\mu_i & \, = \,  \sum_{j\geq i} \. (j-i+1) \ts b_j\..
\end{align*}

We can immediately write the weak form of the limit shape by appealing to Riemann sums.
In this case, the number of parts of size~$i$, $i \geq 1$, is generated from independent geometric random variables $Z_i \in \{0,1,\ldots\}$ with parameter $1-x^i$, where $x = e^{-c/\a}$.
The constant $c$ is given by equation~\eqref{Uber:c:no:a} with $r = 2$, $B = \frac{1}{2}$, and $\a = n^{2/3}$.
In this case, however, we set $Z_j = 0$ if $j \notin \{1, 3, 6, 10, \ldots\}$.
It will also be apparent in the calculation below that we need two different, but complementary, scaling functions.
We have:
\begin{align*}
\E\left[ \frac{\b}{n}\, \mu_{x\,\b}\right] & \, = \, \frac{\b}{n}\, \sum_{j \ts\geq\ts \b \ts x} \. (j-x\,\b+1)\, \E\left[Z_{\binom{j+1}{2}}\right] \\
   & \, = \, \frac{\b}{n}\, \sum_{j \ts\geq\ts \b \ts x} \, \bigl(j-x\,\b+1\bigr) \, \frac{\exp\left(-c\, \frac{\binom{j+1}{2}}{\a}\right)}{1-\exp\left(-c\, \frac{\binom{j+1}{2}}{\a}\right)}\..
\end{align*}
We let $y_j$ satisfy $\ts \frac{1}{2}y_j^2 = \frac{\binom{j+1}{2}}{\a}\ts$. Hence, \ts
$y_j \sim \frac{j}{\sqrt{\a}}$, \ts $\Delta y_j \sim \frac{1}{\sqrt{\a}}$, and
\begin{align*}
 \frac{\b}{n}\, \mu_{x\,\b}  \, = \, &\  \frac{\b^2\sqrt{\a}}{n} \.
 \sum_{\bigl(\sqrt{\a}/\b\bigr)\ts y_j \ts \geq\ts x} \.
 \left(\frac{\sqrt{\a}}{\b}\, y_j - x\right)\. \frac{e^{-c\, \frac{1}{2} y_j^2}}{1-e^{-c\, \frac{1}{2} y_j^2}}\. \Delta y_k \\
 \longrightarrow &\ \int_x^\infty \. (y-x)\.\frac{e^{-c\, \frac{1}{2}y^2}}{1-e^{-c\, \frac{1}{2}y^2}}\, dy, \quad \mbox{ where } \quad x>0,
 \quad c \ts = \ts \frac{1}{2}\. \pi^{1/3}\th \zeta\left(\frac{3}{2}\right)^{2/3}.
 \end{align*}
The final limit is valid if and only if $\a = \b^2$, and $\b^3 = n$, thus we obtain scaling factors $\a = n^{2/3}$ and $\b = n^{1/3}$.

There is one final step, however, since the calculation above only implies the formula given in equation~\eqref{convex:limit:shape} satisfies \emph{the weak notion of a limit shape}.
We need some form of concentration result, i.e., large deviations, to claim that the above formula is indeed the limit shape for the set of partitions.
As we shall see, however, the same concentration results valid for the limit shape of~$\Triangle$ also apply immediately to the image~$\C$,
see~\S\ref{concentration}.
\end{proof}

The bijection $\varphiv$ is a special case of a more general example discussed in Section~\ref{sect r}.
In addition, this example is typical of a general method to find limit shapes described in \S\ref{transfer:theorems}.

% Section: Notations and basic results
\bigskip\section{Notations and basic results}
\label{sect notation}

% Notation
\subsection{Notation}

We denote the set of positive real numbers by $\R_+ = \{ x > 0\}$.
The set of all probability measures absolutely continuous with respect to Lebesgue measure on the positive real line is denoted by $L_+^1(\R_+)$.
Let $\mathbb{H}$ denote the set of real--valued operators on the Banach space of real--valued, countable sequences.
We think of transformations in $\mathbb{H}$ as infinite--dimensional, real--valued matrices.

The expression $a_k \sim b_k$  means
\[   \lim_{k\to\infty} \frac{a_k}{b_k} = 1.
\]
We denote the positive part of $x$ by $x_+:= \max(0,x).$  The ceiling function, or smallest integer larger or equal to $x$, is denoted by $\lceil x \rceil.$
Let $\Gamma(z)$ denotes the \emph{Gamma function}, and $\zeta(z)$ the \emph{Riemann zeta function}.

Define the family of constants
\begin{equation}
\label{Uber c}
d(r,B,a)  \. := \.
\left[\frac{(1-a^{-1/r})\. \zeta\left(1+\frac{1}{r}\right)\. \Gamma\left(1 + \frac{1}{r}\right)\, }{r\, B^{1/r}}\right]^{r/(1+r)} \quad \text{for all } \ \. a \geq 2, \ r \geq 1, \ B>0\ts.
\end{equation}
We also define
\begin{equation}\label{Uber:c:no:a}
d(r,B) \. := \. \lim_{a\to\infty} d(r,B,a)\,  \quad  \text{for all } \ r \geq 1, \ B>0\ts.
\end{equation}
The letter $c$ is not restricted to a single value, but will be defined locally.\footnote{This is to emphasize the secondary role of the constants, and to make the expressions more readable.}

Given a sequence $y_k$, $k\ge 1$, we denote the forward difference operator by $\Delta y_k = y_{k+1}-y_k.$
We also define the $k^{th}$ difference recursively as
\[ \triangle_i^k(\lambda) =
\begin{cases}
\lambda_i, & \text{if } \ i=\ell \ \text{ or } \ k = 0, \\
\triangle_i^{k-1}(\lambda) - \triangle_{i+1}^{k-1}(\lambda), & \text{otherwise}.
\end{cases}
\]

The row vector  $m = (m_1, m_2, \ldots)$ has a transpose given by a column vector and is denoted by $m^T = (m_1, m_2, \ldots)^T$.

From this point on, we adopt the notation that matrix indices are always integers, i.e.,
\[ v(t,y) \equiv v(\lceil t\rceil, \lceil y \rceil) \quad \text{ for all real } \ \. t>0, \ y>0.\]

For a given limit shape $F$, we denote its compositional inverse by $F^{\<-1\>}$, or simply $F^{-1}$ when the context is clear.  In our setting, the domain and range of a limit shape is always the positive real numbers, and limit shapes are always monotonically decreasing, hence $F^{\<-1\>}$ is also a function.

% Integer partitions
\subsection{Integer partitions}

	An integer partition of a positive integer $n$ is an unordered list of nonnegative integers whose sum is $n$.  It is standard notation to list the parts in descending order, $\lambda_1 \geq \lambda_2 \geq \ldots \geq \lambda_\ell>0$, and we denote the number of parts by $\ell = \ell(\lambda)$.  We say that $\lambda = (\lambda_1, \lambda_2, \ldots, \lambda_\ell)$ is an integer partition of \emph{size}~$n$ if $|\lambda| := \sum_i \lambda_i = n$.  Each partition $\lambda$ has a unique \emph{conjugate} partition $\lambda'$ that satisfies $\lambda'_i = \#\{j: \lambda_j \geq i\}$, $i\geq 1$.  We denote by $\Pn$ the set of partitions of size~$n$, and $\P = \cup_n \Pn$ denotes the set of partitions of all sizes.   A subset of partitions $\A \subset \P$ is defined similarly, with $\A_n = \A \cap \P_n$.
		
	 Each partition $\lambda\in \P$ has a corresponding Ferrers diagram, which is a collection of points on a two-dimensional lattice corresponding to parts in the partition; see Figure \ref{Ferrers} and see also~\cite{Pak} for a further explanation.  We note, in particular, that the Ferrers diagram for $\lambda'$ is a reflection of the Ferrers diagram for $\lambda$ about the line $y=x$.
	
	Let $m_i(\lambda) = \#\{{\rm parts\ of\ size} \ i\ {\rm in}\ \lambda\}$ denote the multiplicity of parts of size $i$, where $i \geq 1$.  There is a natural one-to-one correspondence between partitions $\lambda\in\P_n$ and sequences $m = (m_1,m_2,\ldots)$ with $\sum i\th m_i = n$.
When there are no restrictions on the multiplicities of parts, we call such a partition \emph{unrestricted}.

	 We define
\[ \P_\bolda = \{ 1^{m_1} 2^{m_2} \ldots \mbox{ s.t. }  m_i < a_i , \ i\geq 1\},
\]
to be the set of partitions where parts of size $i$ can occur at most $a_i-1$ times, where $\bolda = (a_1, a_2, \ldots)$, $a_i \in [1,\infty]$, $i \geq 1$.  The case $a_i = 2$ for all $i \geq 1$ corresponds to distinct part sizes, and the case $a_i = \infty$ for all $i \geq 1$ corresponds to unrestricted integer partitions.
We use the term \emph{Andrews class partitions} to refer to sets $\P_\bolda$.

\begin{figure}
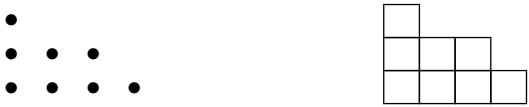

	\begin{subfigure}[l]{0.3 \textwidth}
\[
\begin{array}{cccc}
\bullet   \\
\bullet  & \bullet & \bullet\\
\bullet& \bullet&\bullet & \bullet \\
\end{array}
\]
\end{subfigure}
\begin{subfigure}[bl]{0.3\textwidth}
\[
\begin{Young}
   \cr
 &  & \cr
 &  &  &  \cr
 \end{Young}
\]
\end{subfigure}
\caption{The partition~$\lambda = (4,3,1)$ as a Ferrers Diagram (left), a Young Diagram (middle), and the diagram function $D_\lambda(x)$ (right).}
\label{Ferrers}
\end{figure}

% Asymptotic enumeration of integer partitions
\subsection{Asymptotic enumeration of integer partitions}

We consider sequences \[U=\{u_1,u_2,\ldots\},\] where $u_k \geq 1$ are strictly monotone increasing and are polynomials in $k$; that is,
\begin{equation}\label{U1} u_k = a_r k^r + a_{r-1} k^{r-1} + \ldots + a_0, \qquad a_r > 0, r \in \N.  \end{equation}
We can, in fact, consider more general sequences, however, there are added technical conditions which unnecessarily complicate the statement of our theorems, and since all of our examples take the simpler form in~\eqref{U1}, we forego this added generality.
See~\S\ref{Ingham Form} for a discussion on the consequences of such generalizations.

When {\tt gcd}$(U) = b\geq1$, then in what follows, by $n\to\infty$ we mean taking $n$ to infinity through integer multiples of $b$. .

\begin{theorem}[\cite{Ingham, RothSzekeres, rth}]\label{Ingham:theorem}
Assume the set $U$ satisfies~\eqref{U1}, for some $r \in \N$, and let $B \equiv a_r$ and $E \equiv a_{r-1}$.  Let $p_U(n)$ denote the number of partitions of $n$ into part sizes in $U$. Then we have
\begin{equation}\label{Ingham:formula} p_U(n) \sim \frac{\exp\mts\left[{(1+r)\,d(r,B)\, B^{-1/(1+r)}\, n^{1/(1+r)}}\right]}{c_1\, n^{\frac{B\,r + E}{B(r+1)}+\frac{1}{2}}}\ts,\end{equation}
where
\begin{align*}
c_1 & \, = \,  d(r,B)^{1+E/(B\,r)} \ts
B^{\frac{1}{2}+E/(B\,r)}\. (1+r^{-1})^{-1/2}\. (2\pi)^{-(r+1)/2} \, \prod_{j=1}^r \. \Gamma(1+\rho_j)
\end{align*}
is a constant, and where $\{\rho_j\}_{j=1}^d$ denotes the negatives of the roots of $u_k$.
\end{theorem}

\smallskip

\begin{theorem}[\cite{RothSzekeres}]\label{RothSzekeres}
Assume the set $U$ satisfies~\eqref{U1} for some $r \geq 1$, and let $B \equiv a_r$ and $E \equiv a_{r-1}$.
Let $p_U^d(n)$ denote the number of partitions of $n$ into distinct parts from $U$.
Then we have
\begin{equation}\label{RS:formula}
p_U^d(n) \, \sim \, \frac{\exp\mts\left[{(1+r)\,d(r,B,2)\, B^{-1/(1+r)}\, n^{1/(1+r)}}\right]}{c_2\, n^{\frac{B\,r + E}{2B(r+1)}+1}}\ts,\end{equation}
where
\begin{align*}
c_2 & \, = \, 2^{-(1+E/(r\,B))} \. d(r,B,2)^{1+E/(B\,r)} \. B^{\frac{1}{2}+E/(B\,r)}\. (1+r^{-1})^{-1/2}\pi^{-1/2}
\end{align*}
is a constant.
\end{theorem}

Whereas only the leading two coefficients of $u_k$ are relevant for Theorem~\ref{RothSzekeres} above, the values of the roots of $u_k$ are also needed in Theorem~\ref{Ingham:theorem}.
However, in terms of limit shapes only the exponent~$r$ and leading coefficient $B \equiv a_r$ will play a fundamental role.
A set of partitions $\P_U$ taking parts only from the set $U$ will be called \emph{\smooth with parameters $r$ and $B$} if the multiplicities of parts is unrestricted and $U$ satisfies~\eqref{U1} with $u_k = B k^r \bigl(1+o(1)\bigr)$.
A set of partitions $\P_U$ taking parts only from the set $U$ will be called \emph{restrictedly smooth with parameters $r$, $B$ and $a$} if the multiplicity of each part size is strictly bounded from above by $a$ and $U$ satisfies~\eqref{U1} with $u_k = B\, k^r \bigl(1+o(1)\bigr)$.

% Section: Limit shapes
%\bigskip\section{Limit shapes}

% Diagram Functions
\subsection{Diagram functions and limit shapes}
\label{diagram functions}
For each $\lambda\in \P$, the \emph{diagram function} of $\lambda$ is a right-continuous step function $D_\lambda:(0,\infty) \longrightarrow \R_+$ defined as
\begin{align*} D_\lambda(t) = \sum_{k\geq t} \thinspace m_k(\lambda), & \qquad \text{with} \qquad  \int_0^\infty D_\lambda(t)\, dt = |\lambda|. \end{align*}
Define a \emph{scaling function} to be any positive, monotonically increasing function $\alpha:\Rplus\to\Rplus$, such that $\alpha(n)\to\infty$ as $n\to\infty$.
We define the \emph{$\alpha$-scaled diagram function} of $\lambda$ as
\begin{align*} \Dhat_\lambda(t) \. =  \. \frac{\a}{n}\. D_{\lambda}(\a\, t), & \qquad \text{with} \qquad \int_0^\infty \Dhat_\lambda(t)\, dt = 1. \end{align*}

Our definition of limit shape is from \cite{Yakubovich}, and is as follows.
Let $\{\Pr_n\}$ denote a family of measures on subsets of partitions $\A_n \subset \Pn$, where $n \geq 1$.
In this paper we consider only $\Pr_n = 1/|\A_n|$, the uniform distribution.
Suppose there exists a scaling function $\a$  and a piecewise continuous function $\Phi: \R_+\longrightarrow\R_+$ such that $\int_0^\infty \Phi(t) dt = 1$.
Suppose further that for any finite collection $0<t_1< \ldots < t_k$ of continuity points of $\Phi$, and any $\epsilon>0$,
\begin{equation}
\label{limit shape definition}
\Pr_n\left( \left| \Dhat_\lambda(t_j) - \Phi(t_j)\right| < \epsilon\  \thinspace \text{for  all}\  \thinspace j=1, \ldots, k\right) \to 1, \ \  \text{as }\thinspace n\to\infty.
\end{equation}
We say that $\Phi$ is the \emph{limit shape of} $\A$ under scaling function $\a$.

Similarly, define the \emph{conjugate diagram function} $\rho_\lambda:(0,\infty)\longrightarrow\R_+$, such that
\[
\rho_\lambda(x) := \begin{cases} \lambda_{\lceil x\rceil}, & \text{if}\ 0<x\leq \ell \\
0, & \text{if}\ x>\ell.
\end{cases}
\]
We similarly define the \emph{$\alpha$-scaled conjugate diagram function}
\[ \rhohat_\lambda(x) \. := \. \frac{\a}{n} \. \rho_\lambda\bigl(\a x\bigr) \qquad \text{ with} \qquad \int_0^\infty \widehat{\rho}_\lambda(t) \, dt = 1.
\]
Suppose there exists a scaling function $\a$ and a piecewise continuous function \ts $\Phi: \R_+\longrightarrow\R_+$ \ts
such that \ts
$\int_0^\infty \Phi(t) \ts dt = 1$. Suppose further, that for any finite collection $0<t_1<\ldots<t_k$ of continuity points of~$\Phi$,
  and any $\epsilon>0$, we have
\[ \Pr_n\left( \left| \rhohat_\lambda(t_j) - \Phi^{-1}(t_j)\right| < \epsilon\  \thinspace \text{for  all}\  \thinspace j=1, \ldots, k\right) \to 1, \ \  \text{as }\thinspace n\to\infty.
\]
Then we say that $\Phi^{-1}$ is the \emph{conjugate limit shape of} $\A$ under scaling function~$\a$.

\begin{remark} {\rm
The scaling function $\a$ in the definition of limit shape is not unique, since it can be replaced with any scaling function $\b$ such that $\b\sim\a$.
One could consider an equivalence class of scaling functions in order to uniquely specify such a class of functions, but we do not pursue this further.
In either case, when the limit shape of $\A\subset \P$ exists, it is unique since it must integrate to 1.
}
\end{remark}

Following \cite{Vershik}, define the space $\DV = \{(\alpha_0, \alpha_\infty, p)\}$  of triples $\alpha_0, \alpha_\infty \in \mathbb{R}_+$, and nonnegative monotonically decreasing functions $p\in L^1_+(\R_+)$, such that
$$
\alpha_0 \. + \. \alpha_\infty \. + \. \int_0^\infty \. p(t)\. dt \, = \, 1.
$$
One can think of $\DV$ as the set of measures on \ts $\mathbb{R}_+$ \ts of the form \ts
$\alpha_0 \delta_0 + \alpha_\infty \delta_\infty + p(t)\. dt$,
where $\delta_0$ and $\delta_\infty$ are point masses at~0 and~$\infty$.
Note that $(0,0,\Dhat_\lambda(t))\in\DV$ for all $\lambda\in\P$ and scaling functions~$\a$, and
similarly, $(0,0,\rhohat_\lambda(t))\in\DV$ for all $\lambda\in\P$ and scaling functions~$\a$.
The space~$\DV$ is closed, so all limits of diagram functions exist in~$\DV$,
even when limit shapes do not exist.

\begin{remark}\label{vanishing:mass}{\rm
Sometimes the limit shape does not exist, even though there is a scaling for which a limit \emph{curve} exists.
Consider, as in~\cite[Example~3.4]{PakNature}, the set of partitions $\A = \{k^k 1^{k^2}\}$, $k \geq 1$, consisting of a $k \times k$ square block and a $1 \times k^2$ tail.
For each $n$ of the form $n = 2k^2, k \geq 1$, there is exactly one partition in $\A$ of size~$n$.
In fact, one can take the limit of the scaled diagram function through values of $n$ indicated, and obtain the limit \emph{curve}
\[
a(x) = \begin{cases} \frac{1}{\sqrt{2}}, & 0 \leq x \leq \frac{1}{\sqrt{2}} \\
0, & \mbox{otherwise,}\end{cases}
\]
which does not have unit area.  This corresponds to the element $ \left(\frac{1}{2}, 0, a(x) \right) \in \DV$.  All other scaling functions would send the scaled diagram function to either $(1,0,0)\in \DV$ or $(0,1,0)\in \DV$ in the limit.
}\end{remark}

% U theorems
\subsection{Two $U$ theorems}

The theorems below indicate the types of limit shape results already known, which we shall use to obtain new ones.
They are specializations of~\cite[Th.~8]{Yakubovich}.

% Unrestricted Theorem
\begin{theorem}[\cite{Yakubovich}]  \label{special unrestricted}
	Suppose $\P_U$ is \smooth with parameters $r\geq 1$ and $B>0$.
	Let $c=d(r,B)$.  Then the limit shape of $\P_U$ under scaling function $\a=n^{r/(1+r)}$ exists and is denoted by $\Phi(t; r, B)$, and we have 	
	\begin{equation}
\label{unrestricted limit shape}
\PhiLS \, = \, \int_{(t/B)^{1/r}}^\infty  \. \frac{e^{-c\, B\, y^{r}}}{1-e^{-c\, B\, y^{r}}}\. dy, \quad t>0.
\end{equation}
\end{theorem}

Theorem~\ref{special unrestricted} above covers a large class of examples of interest in our present setting.
It is also the starting point for finding limit shapes of partitions with non-multiplicative restrictions in \S\ref{transfer:theorems}.
We will also need another particular specialization, which is when all multiplicities of allowable part sizes are strictly bounded from above by some fixed number $a \geq 2$.
Typically the value $a=2$, i.e., distinct parts, is sufficient, but~\S\ref{od general} contains an example where we need further generality.

% Distinct Parts Limit Shape
\begin{theorem}[\cite{Yakubovich}] \label{special distinct}
	Suppose $\P_U$ is \rsmooth with parameters $r\geq 1,$ $B>0,$ and $a \geq 2$.
Let $c=d(r,B,a)$.  Then the limit shape of $\mathcal{A}$ under scaling function $\a=n^{r/(1+r)}$ exists and is denoted by $\Phi(t; r, B, a)$, and we have
\begin{equation*}
\Phi(t; r, B,a) \, = \, \int_{(t/B)^{1/r}}^\infty \.
\frac{e^{-c\, B\, y^{r}}+2e^{-2\, c\, B\, y^{r}}
+ \cdots + (a-1)e^{-(a-1)\, c\, B\, y^{r}}}{1+e^{-c\, B\, y^{r}}+e^{-2\, c\, B\, y^{r}}+
\cdots+e^{-(a-1)\, c\, B\, y^{r}}} \. dy, \quad t>0\ts.
\end{equation*}
\end{theorem}
%\frac{x+2x^2+3x^3+\ldots (a-1)\,x^{a-1}}{1+x+x^2+x^3+\ldots+x^{a-1}}
In particular, by taking $f(x) = 1/(1-x)$ and $f(x) = (1+x+\ldots + x^{a-1})$, respectively,
in~\cite[Th.~8]{Yakubovich}, we recover Theorem~\ref{special unrestricted} and Theorem~\ref{special distinct}.

% MM and MP bijections
\subsection{\MMb and \MPb}
\label{bijections}

Two subsets of partitions $\A, \B\subset \P$ are called equinumerous if $|\A_n| = |\B_n|$ for all $n \geq 1$.
For sets $\A_n$, $\B_n$ $\subset \P_n$, we consider one-to-one correspondences $\vartheta_n:\A_n \to \B_n$ which can be expressed as a real--valued linear transformations.

We define \MPb (multiplicities to parts) to be a one-to-one correspondence that maps, using a linear transformation, multiplicities of parts of partitions $\lambda\in\A_n$ to parts of partitions $\mu\in\B_n$ for all $n$.
The transformation has the form
\begin{equation}
\label{v transform}
\mu := \left(\sum_{j=1}^n v(1,j)\, m_j, \sum_{j=1}^n v(2,j)\, m_j, \ldots, \sum_{j=1}^n v(n,j)\, m_j\right) \geq 0,
\end{equation}
for some set of coefficients~$\bv = \{v(i,j): 1 \leq i,j\leq n\}$ such that
\[ \sum_i v(i,j) = \begin{cases} j, & \text{ when } j\in U \\
 0, & \text{ otherwise, } \end{cases} \qquad \text{ for all } j=1,\ldots, n\ts.
\]
To avoid excessive subscripting we have opted for the functional notation of indices, i.e., $v(i,j)$ instead of $v_{i,j}$.

Similarly, we define \MMb (multiplicities to multiplicities) to be a one-to-one correspondence $\vartheta_n (m_1,m_2,\ldots, m_n)$ that maps, using a linear transformation, multiplicities of parts of partitions $\lambda\in\A_n$ to multiplicities of parts of partitions $\mu\in\B_n$.  Define the vector~$\bfx = (x_1, x_2, \ldots, x_n)\in \R^n$, with
\[ x_i = \begin{cases}
i, & i \in U \\
0, & \text{ otherwise},
\end{cases} \qquad \text{for all } i=1, 2, \ldots, n.
\]

Denote the multiplicities of $\mu$ by $(m_1', m_2',\ldots, m_n')$, then the transformation has the form
\begin{equation}
\label{v transform MM}
(m_1', m_2', \ldots, m_n') =  \left(\sum_j v(1,j)\, m_j, \sum_j v(2,j)\, m_j, \ldots, \sum_j v(n,j)\, m_j\right) \geq 0,
\end{equation}
and the vector~$\bfx$ is a left eigenvector of the matrix~$\bv$ with eigenvalue~1.

We define $\varphiv := \vartheta_n$ when the correspondence $\vartheta_n$ is an \MMb or \MPb with coefficients~$\bv$, to emphasize the role of $\bv$ in the bijection, and to distinguish \MMbns s and \MPbns s from general one-to-one correspondences between sets of partitions (see for example~\cite{Pak}).

We next extend the notion of \MMb and \MPb  to countable sequences.
  In this setting, the matrix of coefficients~$\bv = \{v(i,j): 1 \leq i, j \}$ lies in $\mathbb{H}$, which we recall is the set of real--valued operators on the Banach space of real--valued, countable sequences.
Define also the vector~$\bfy~=~(y_1, y_2, \ldots)$, with
\[ y_i = \begin{cases}
i, & i \in U \\
0, & \text{ otherwise},
\end{cases}  \qquad  \text{for all }  i\geq 1.
\]

  For \MPbns s, we have
  \[ \sum_i v(i,j) = \begin{cases} j, & \text{ when } j\in U \\
 0, & \text{ otherwise, } \end{cases} \qquad \text{for all } j \geq 1.
 \]

Similarly, for \MMbns s, we have the vector~$\bfy$ is a left eigenvector of matrix~$\bv$ with eigenvalue~1.
Note that it is not necessary for every entry of $v$ to be nonnegative.

% Stability
\subsection{Stability}
   Suppose $\A \subset \P$ denotes a set of partitions which is \smooth with parameters $r \geq 1,$ $B>0.$
We denote the limit shape of $\A$ by $\Phi(x; r, B),$ and, furthermore, we define $\phi(y; r, B)$ as the continuous function $\phi(\,\cdot\,; r, B)~:~\R_+ \longrightarrow \R_+,$ which satisfies
\[ \Phi(x; r, B) = \int_{(x/B)^{1/r}}^\infty \phi(y; r, B)\, dy,  \quad \text{where} \quad x>0.
\]
Suppose there exists an \MMb or \MPb between a set of partitions $\A$ and some other set of partitions $\B \subset \P$, with coefficients~$\bv$.  Since partitions in $\A$ only have parts in $U$, the nonzero values of coefficients $\bv$ are of the form $v(i, u_k)$, $i \geq 1$, $k \geq 1$.

We require a condition on the coefficients $\bv$ as follows.  We let
\[ y_k := \left(\frac{u_k}{B\, \a}\right)^{1/r},\quad k\ge 1,\]
so that
\[ y_k \sim \frac{k}{\a^{1/r}}\. , \qquad \Delta y_k \sim \left(\frac{1}{\a}\right)^{1/r}\. . \]

We say the coefficients $\bv$ are \MPstable\ if there exists a
bounded and piecewise continuous function  $K:\R_+\times \R_+ \longrightarrow \R_+$ such that for all $t>0$ and $k \geq 1,$ we have
\begin{equation} \label{v nice}
\begin{array}{c}
\displaystyle v(\lceil \b t\rceil , u_k )\, \phi(u_k; r, B) \sim  \beta(n)\th K(t,y_k, \phi(y_k; r, B)).
\end{array}
\end{equation}

Similarly, we say the coefficients $\bv$ are \MPstablea\ if there exists a
bounded and piecewise continuous function  $K:\R_+\times \R_+ \longrightarrow \R_+$ such that for all $t>0$ and $k \geq 1,$ we have
\begin{equation} \label{v nicea}
\begin{array}{c}
\displaystyle v(\lceil \b t\rceil , u_k )\, \phi(u_k; r, B, a) \sim  \beta(n)\th K(t,y_k, \phi(y_k; r, B, a)).
\end{array}
\end{equation}

We say the coefficients $\bv$ are \MMstable\ if there exists a
bounded and piecewise continuous function  $K:\R_+\times \R_+ \longrightarrow \R_+$ such that for all $t>0$ and $k \geq 1,$ we have
\begin{equation} \label{v nice2}
\begin{array}{c}
\displaystyle \sum_{k=1}^{\lfloor u^{-1}(n)\rfloor } v(\lceil \b t\rceil , u_k )\, \phi(u_k; r, B) \sim  \beta(n)\th K(t,y_k, \phi(y_k; r, B)).
\end{array}
\end{equation}

Similarly, we say the coefficients $\bv$ are \MMstablea\ if there exists a
bounded and piecewise continuous function  $K:\R_+\times \R_+ \longrightarrow \R_+$ such that for all $t>0$ and $k \geq 1,$ we have
\begin{equation} \label{v nice2a}
\begin{array}{c}
\displaystyle \sum_{k=1}^{\lfloor u^{-1}(n)\rfloor } v(\lceil \b t\rceil , u_k )\, \phi(u_k; r, B,a) \sim  \beta(n)\th K(t,y_k, \phi(y_k; r, B,a)).
\end{array}
\end{equation}

When an \MMb $\varphiv$ is defined by an \stable\ set of coefficients $\bv$, we say that bijection $\varphiv$ is \stable, and we denote the limit shape of $\B$ by $\Phi(t; r, B, K)$.
Similarly, when an \MPb $\varphiv$ is defined by an \stable\ set of coefficients~$\bv$, we say that bijection $\varphiv$ is \stable, and we denote the limit shape of $\B$ by $\Psi(t; r, B, K)$.

A bijection defined in terms of linear combinations of multiplicities of part sizes acts as a Markov operator on the diagram function; see for example \cite[Section 2.3.39]{Bobrowski}.  Recall that a Markov operator {\rm P} is a linear operator that sends functions $f \in L^1_+(\R_+)$ to $L^1_+(\R_+)$ and is such that
\begin{enumerate}
\item ${\rm P}f \ts \geq \ts 0$ \, for\, $f \geq 0$, and
\item $\int_0^\infty {\rm P}f(x) \. dx \. = \. \int_0^\infty \ts f(x) \. dx$\ts.
\end{enumerate}
The coefficients $\bv$ in Equation~\eqref{v nice} correspond to an explicit form of the operator {\rm P}.  Note that while we must have $\sum_j v(i,j) m_j \geq 0$ for all $i\ge 1$, there is no restriction that each entry $v(i,j)$ be nonnegative.
Most noteworthy, however, is the fact that Markov operators are linear contractions, and hence preserve a form of the large deviation principle, which we revisit in~\S\ref{concentration}.

\begin{remark}{\rm
Not all bijections give a means to compute the limit shape.  For example, it is not apparent how to utilize Dyson's iterated map for odd-distinct parts to compute limit shapes~\cite{PakNature}.
A more natural example would be the \emph{Garsia--Milne involution principle} bijections of Rogers--Ramanujan identities~\cite{GarsiaMilneRR} (see also~\cite{Pak}).  Both sides have a limit shape, however, we are not aware of any way to draw conclusions about one limit shape from the other using the involution, as it would seem to be difficult to do in general~\cite{Konvalinka}.
}\end{remark}

%  Section -- Transfer theorems
\bigskip\section{Transfer theorems}
\label{sect special cases}

\subsection{\MMb and \MPb}
\label{transfer:theorems}
We now state our key technical theorems, which specify conditions under which the limit shapes map continuously, and provide the means for computing an explicit formula.

% Transformed Unrestricted Theorem
\begin{theorem}[First transfer theorem] \label{MT unrestricted}
	Suppose $\P_U$ is \smooth with parameters $r\geq 1$ and $B>0$.
	Let $c~=~d(r,B)$.
	
	\begin{enumerate}
	\item If there exists a set of partitions $\B\subset \mathcal{P}$ and an \MMb
	$\varphi_{\bf v}: \P_U \to \mathcal{B}$ that is \stable, then the limit shape of $\B$ under scaling function $\b~=~n^{1/(1+r)}$ is given by
\[ \int_0^\infty K\left(t,y,\ts \frac{e^{-c \, B\, y^{r}}}{1 - e^{-c \, B\, y^{r}}}\right) \ts dy, \qquad t>0.
\]
	\item If there exists a set of partitions $\B\subset \mathcal{P}$ and an \MPb
	$\varphi_{\bf v}: \P_U \to \mathcal{B}$ that is \stable, then the conjugate limit shape of $\B$ under scaling function $\b~=~n^{1/(1+r)}$ is given by
\[ \int_0^\infty K\left(t,y,\ts \frac{e^{-c \, B\, y^{r}}}{1 - e^{-c \, B\, y^{r}}}\right)\, dy, \qquad t>0.
\]
\end{enumerate}
\end{theorem}

Note that the integrand function $K$ in the theorem is not necessarily multiplicative with respect to its final argument, as it is in the case of convex partitions in \S\ref{convex}; see also Section~\ref{sect r}.
It is more generally a distribution (in the analysis sense) on $L_+(\R_+)$, and we present an explicit example demonstrating the need for this generality in \S\ref{sect Ohara}.

An analogous result holds for partitions with multiplicities all strictly bounded by some $a \geq 2$.

% Transformed Distinct Parts Limit Shape
\begin{theorem}[Second transfer theorem] \label{thm dist}
	Suppose $\P_U$ is \rsmooth with parameters $r\geq 1,$ $B>0$, $a \geq 2.$
	Let $c~=~d(r,B,a)$.
	\begin{enumerate}
	\item If there exists a set of partitions $\mathcal{B}\subset \P$ and an \MMb $\varphi_{\bf v}: \P_U \to \mathcal{B}$, defined by a set of coefficients ${\bf v}$ that is \stable, then
\begin{equation}
\label{distinct transform}
\small \Phi(t; r, B, a, K) = \int_0^\infty K\left(t,y,\th  \frac{e^{-c\, B\, y^r} + 2e^{-2 c\, B\, y^r} + \ldots + (a-1)\, e^{-(a-1)\, c\, B\, y^r}}{1+e^{-c\, B\, y^r}+e^{-2 c\, B\, y^r} + \ldots + e^{-(a-1)\, c\, B\, y^r}}\right)\th dy, \quad t>0.
\end{equation}
is the limit shape under scaling function $\b~=~n^{1/(1+r)}$.
\item If there exists a set of partitions $\mathcal{B}\subset \P$ and an \MPb $\varphi_{\bf v}: \P_U \to \mathcal{B}$, defined by a set of coefficients ${\bf v}$ that is \stable, then
\begin{equation}
\label{distinct transform1}
\small \Phi^{-1}(t; r, B, a, K) = \int_0^\infty K\left(t,y,\th  \frac{e^{-c\, B\, y^r} + 2e^{-2 c\, B\, y^r} + \ldots + (a-1)\, e^{-(a-1)\, c\, B\, y^r}}{1+e^{-c\, B\, y^r}+e^{-2 c\, B\, y^r} + \ldots + e^{-(a-1)\, c\, B\, y^r}}\right)\th dy, \quad t>0.
\end{equation}
is the conjugate limit shape under scaling function $\b~=~n^{1/(1+r)}$.
\end{enumerate}
\end{theorem}

The case $a=2$ is most common, corresponding to partitions into distinct parts, but the analysis is the same for any $a \geq 2$, and we shall utilize the full generality in \S\ref{od general}.

% geometric bijections
\subsection{Geometric bijections}
\label{geometric transformations}

In \cite{PakNature}, a general paradigm for representing partition bijections is presented involving the composition of transformations acting on the Young diagram of an integer partition.
These transformations map points in the plane via a bijection $\phi$ in such a way that under certain conditions, the limit shape along with certain other statistics of interest are mapped continuously via these transformations.
This property is referred to as \emph{asymptotic stability}, and the limit shape is the statistic of the Young diagram corresponding to the boundary.
We briefly recount the transformations as they pertain to the mapping of the limit shape, and defer the interested reader to the more comprehensive treatment in~\cite{PakNature}.

For example, \emph{Stretch and paste} corresponds to combining several components, say with boundaries $f_1, \ldots, f_m$, which exist on the same domain and with the same area, by taking the sum with an appropriate scaling to keep the sum of the resulting expression constant.  We think of this operation as a shuffle, which corresponds to
\begin{equation*}
f_1(m\,t) + \ldots + f_m(m\,t).
\end{equation*}

\smallskip
\nin
\textsf{(1)} \ts \emph{Conjugation} is equivalent to taking a functional inverse: $f(t) \longmapsto  f^{-1}(t)$.

\smallskip
\nin
\textsf{(2)} \ts\emph{Move} adds a constant: $f(t) \longmapsto  f(t)+a$.

\smallskip
\nin
\textsf{(3)} \ts\emph{Shift} adds a linear function $a\, t$: $f(t) \longmapsto  f(t) + a\, t$.

\smallskip
\nin
\textsf{(4)} \ts\emph{Shred} and move acting together on $r$ components effectively splits up the limit shape into $r$ pieces, each having shape scaled by $r$ with area $1/r$, which corresponds to  $f(t) \longmapsto  (f(r\, t), \ldots, f(r\, t)).$

\smallskip
\nin
\textsf{(5)} \ts A \emph{union}, or \emph{sort}, between two components corresponds to the addition $f(t) + g(t)$.

\smallskip
\nin
\textsf{(6)} \ts The \emph{+} operator between two components corresponds to the inverse to the addition of the inverses: $(f^{-1}(t)+g^{-1}(t))^{\<-1\>}$.

\smallskip
\nin
\textsf{(7)} \ts A \emph{cut} is somewhat more complicated, and requires treating the limit shape as defining a region in the positive quadrant of~$\R^2$, which can be split in two pieces, one of which is a piecewise function, the other being the inverse of a piece-wise function; see for example, Figure~\ref{parallel:mapping:limit}.  It partitions the curve into two curves, say via projections $P$ and $Q$, as $f(t) \longmapsto  ( P\, f(t), Q\, f(t) ).$

\smallskip
\nin
\textsf{(8)} \ts A paste \emph{by itself} corresponds to combining two curves whose domains partition some region in $\R^2$ into a single curve.
Under appropriate conditions, paste acts as the inverse to a cut, i.e., $( P\, f(t), Q\, f(t) ) \longmapsto  f(t)$.

\medskip

\begin{theorem}[Third transfer theorem, \cite{PakNature}] \label{transformations}
{\rm
The geometric transformations on integer partitions correspond to the following transformations on limit shapes.\\
\begin{center}
\begin{tabular}{|l|l|} \hline
Transformation & Limit shape map \\ \hline \hline
Stretch and paste & $f_1(m\,t) + \ldots + f_m(m\,t)$ \\ \hline
Cut & $f(t) \longmapsto  ( P\, f(t), Q\, f(t) )$  \\ \hline
Conjugation & $f(t) \longmapsto  f^{-1}(t)$ \\ \hline
Move & $f(t) \longmapsto  f(t)+a$ \\ \hline
Shift & $f(t) \longmapsto  f(t) + a\, t$ \\ \hline
Shred and move & $f(t) \longmapsto  (f(r\, t), \ldots, f(r\, t))$ \\ \hline
Paste & $( P\, f(t), Q\, f(t) ) \longmapsto  f(t)$ \\ \hline
Union & $f(t) + g(t)$ \\ \hline
$+$ & $(f^{-1}(t)+g^{-1}(t))^{\<-1\>}$ \\ \hline
\end{tabular}
\end{center}
}
\end{theorem}

% *************************Convex Partitions*****************************

%\bigskip\section{First Example: Convex Partitions}
\bigskip\section{Partitions with nonnegative $r^{th}$ differences}
\label{sect r}
Throughout this section, $r$ is any fixed integer, $r \geq 2$.

% nonnegative r-th differences theorem
\begin{theorem} \label{nonneg}
Let $\C^r$ denote the set of partitions with parts that have nonnegative $r^{th}$ differences; i.e., $\triangle_i^r(\mu) \geq 0$, $i \geq 1$ for each $\mu \in \C^r$.
The conjugate limit shape of $\C^r$ under scaling function $\b~=~n^{1/(1+r)}$ is given by
\[ \Phi\left(t; r, 1/r!, \frac{( y-t)^{r-1}_+}{(r-1)!} \phi(y; r,1/r!) \right), \]
with explicit formula
\[ \int_{t}^\infty \frac{\bigl(y-t\bigr)^{r-1}}{(r-1)!} \frac{e^{-c\, y^r/r!}}{1-e^{-c\, y^r/r!}}\th dy, \qquad t>0, \quad \mbox{ where } \quad c~=~d(r,1/r!). \]
\end{theorem}

When $r=2$, we obtain Corollary~\ref{conj convex}.  A similarly defined set of partitions with multiplicative restrictions is utilized to obtain this result.

% r-th triangular theorem
\begin{theorem}[\cite{Yakubovich}, cf.~\S\ref{history}] \label{r-th triangular}
Let $F^r$ denote the set of partitions into parts of sizes given by $u_k = \binom{r+k-1}{r}$, $k\geq 1$.
The limit shape of $F^r$ under scaling function $\a = n^{r/(1+r)}$ is given by $\Phi(t; r, 1/r!)$.
\end{theorem}

We now define the bijection $\varphi_r : F^r_n \to \C^r_n$ in \cite{rth}, which is a straightforward generalization to the one in \S\ref{convex}.  For each $\lambda \in F^r_n$, let $m_j(\lambda)$ denote the number of parts of size~$\binom{r-1+j}{r}$ in~$\lambda$.  Then $\mu = \varphi_r(\lambda)$ is given by
\[
\mu_i \, = \, \sum_{j\geq i} \. \binom{r-1+j-i}{r-1} \. m_j(\lambda)\. , \quad i \geq 1\. .
\]

\begin{lemma}
\label{rth lemma}
The bijection $\varphi_r$ is an \MPbns.  Moreover, $\varphi_r$ is $(r, 1/r!, K)$--\textsf{MP-stable}, with
\[
K\left(t,y,\phi(y; r,1/r!)\right) \, = \, \frac{( y-t)^{r-1}_+}{(r-1)!} \. \phi(y; r,1/r!)\. .
\]
\end{lemma}

\begin{proof}
For $i\geq1$ and $j \geq 1$, we have $v(i,u_j) = \binom{r-1+j-i}{r-1}$, and 0 otherwise.  We have
\[
\sum_{i} v(i,u_j) \. = \. \sum_{i=1}^{j} \binom{r-1+i-1}{r-1} \. = \. \binom{r-1+j}{r} \. = \. u_j\. .
\]
With $\ts \a = n^{r/(r+1)}$, $\ts \b = n^{1/(r+1)}$, we have
\[
\frac{1}{r!} \. y_j^r \. = \. \frac{u_j}{ \a} \. \sim \. \frac{j^r}{r!\, \a}\. ,
\]
so
\[
y_j~\sim~\frac{j}{\b}, \qquad \Delta y_j~\sim~\frac{1}{\b}\ .
\]
Hence,
\[ 
v\bigl( \b\, t , u_k\bigr) \ts \phi(u_k; r, 1/r!) \. \sim \. \beta(n) \ts K\bigl(t,y_k,\phi(y_k; r, 1/r!)\bigr)\. , 
\]
where $\beta(n) \sim n^{1/(r+1)}$, and
\[ K(t, y, \phi(y; r,1/r!)) \. \sim \.  \binom{r-1 + \b\, y_k - \beta t}{r-1} \phi(y; r, 1/r!) \. \sim \.  
\frac{( y-t)^{r-1}_+}{(r-1)!} \. \phi(y; r,1/r!)\. , \]
which implies the result.
\end{proof}

% **************************** Odd vs. Distinct ******************************
\bigskip\section{Euler's theorem and generalizations}
\label{sect odd}

% Glaisher's bijection
\subsection{Glaisher's bijection}

Let $\O$ denote the set of partitions with all parts odd and $\D$ denote the set of partitions with all parts distinct.  We show in this section that Glaisher's bijection $\varphi: \D \to \O$ is an \MMbns, and present several extensions.

Recall the following classical result.

\begin{theorem}[Euler's Theorem] For each $n \geq 1$, we have $|\O_n| = |\D_n|$.  \end{theorem}

Glaisher's bijection $\varphi$ is defined as follows.  Start with $\lambda\in\D$.  Replace each part of size $m\, 2^r$ with $2^r$ parts of size $m$, for all $r\ge 1$ and $m$ odd.

Define the sets $\O^r = \{$odd parts with multiplicities $<2^r\}$ and $\D^r = \{$distinct parts not divisible by $2^r\}$.  The bijection $\varphi: \D \to \O$, restricted to $\D^r$ and $\O^r$, shows that $|\O^r| = |\D^r|$ for all $r\geq 1$.
The limit shapes of  $\O^r$ and $\D^r$ exist, and the bijection between them was shown to be a geometric bijection \cite{PakNature}.
An informal calculation in \cite{PakNature} derives the limit shape of $\O$ using the limit shapes $\D^r$ for all $r\geq 1$.  Below we formalize this calculation by showing that Glaisher's bijection $\varphi$ is an \MMbns.

% Glaisher's Bijection is nice
\begin{prop} Glaisher's bijection $\varphi: \D^r \to \O^r$ is an \MMb for all $r\in \N$.  Furthermore, Glaisher's bijection $\varphi: \D \to \O$ is an \MMbns.
\end{prop}
\begin{proof}

For clarity of exposition, we first show cases $r=1,2,$ then we show in full generality.

In the case $r=1$, $\O^1 = \D^1$ and $\varphi = id$, the identity transformation.  The transformation matrix is given by $V_1(2s+1, 2s+1) = 1$ for all $s \geq 0$, and 0 otherwise.

In the case $r=2$, $\D^2$ consists of parts of odd size and even parts not divisible by 4.
The transformation is given by
\[ V_2(i,j) =
\begin{cases}
1, & i=2s+1 \text{ and } j=2s+1 \\
2, & i=2s+1 \text{ and } j=4s+2  \\
0, & \text{otherwise},
\end{cases} \qquad \mbox{ for all $s \geq 0$}\. .  \qquad  \quad
\]

The general case $r=k$ has transformation given by
\[\qquad  V_k(i,j) =
\begin{cases}
1, & i=2s+1 \text{ and } j=2s+1 \\
2, & i=2s+1 \text{ and } j=4s+2 \\
%4, & i=2s+1 \text{ and } j=8s+4 \\
\vdots \ &  \ \  \qquad \vdots \\
2^{k-1}, & i=2s+1 \text{ and } j=2^{k}\, s + 2^{k-1} \\
0, & \text{otherwise},
\end{cases}\qquad \mbox{ for all $s \geq 0$ and $k \geq 1$}.
\]
It follows that $V_k$ is a transformation matrix for an \MMb for all $k \geq 1$.

Using pointwise convergence in $\mathbb{H}$, we have \ts $V_k \to V_\infty$ \ts as \ts $k\to\infty$,
where
\begin{equation}
\label{Glaisher's linear transformation}
V_\infty(i,j)\. = \. \begin{cases} 2^{k-1}, & i = 2s+1 \text{ and } j = 2^k s+2^{k-1}, \\
 0, & \text{otherwise},
 \end{cases}\qquad \mbox{ for all $s \geq 0$ and $k \geq 1$}.
\end{equation}
It is easy to verify that $V_\infty$ has a left eigenvector of $(1,2,3,\ldots)$ corresponding to eigenvalue~1.  Thus, Glaisher's bijection $\varphi: \D \to \O$ is an \MMbns.
\end{proof}

We will refer to the transformation $V_\infty$ from equation~\eqref{Glaisher's linear transformation} as \emph{Glaisher's linear transformation}.

% Iterative Glaisher
\subsection{O'Hara's algorithm}
\label{sect Ohara}
A generalization to Euler's Theorem was given by Andrews \cite{Andrews} (see also \cite{Pak}).  For vectors $\bolda~=~(a_1, a_2, \ldots)$ and $\boldb~=~(b_1, b_2, \ldots)$, let $\A_n$ and $\B_n$ denote the Andrews class partitions $\P_\bolda$ and $\P_\boldb$ of size~$n$, respectively.  Define $\rm{supp}(\bolda)$ to be the set of all $i \geq 1$ such that $a_i < \infty$.  We say that $a \sim b$ if there exists a bijection $\pi: \rm{supp}(a) \to \rm{supp}(b)$ such that $i a_i = j b_j$ for all $j = \pi(i)$.

\begin{theorem}[Andrews's Theorem \cite{Andrews}] \label{Andrews Theorem} If $a \sim b$, then $|\A_n| = |\B_n|$ for all $n \geq 1$.
\end{theorem}
\begin{proof}  We have:
\[ 1 + \sum_{n=1}^\infty |\A_n|t^n \, = \, \prod_{i=1}^\infty  \frac{1 - t^{i a_i}}{1-t^i}
\, = \, \prod_{i=1}^\infty \frac{1-t^{j b_j}}{1-t^j} \, = \, 1 + \sum_{n=1}^\infty |B_n| t^n, \]
where we assume the convention that $t^\infty = 0$.
\end{proof}

The following algorithm, see~\cite{OHara}, performs the bijection given by Andrews classes of partitions, and is known as \emph{O'Hara's algorithm}.  For sets of partitions $\P_\bolda$ and $\P_\boldb$, suppose  $\bolda \sim \boldb$.  Start with $\lambda \in \P_\bolda$.  Replace $a_i$ parts of size $i$ with $b_j$ parts of size $j$, where as before $j = \pi(i)$, $i \geq 1.$  Repeat until $m_i(\lambda) < a_i$ for all $i \geq 1.$
In its full generality, O'Hara's algorithm is not asymptotically stable, and may require exponentially many steps, see~\cite{Konvalinka}.
In the special case below, however, the algorithm is in a sense monotonic, and gives a valid alternative approach to obtaining Glaisher's linear bijection.

Consider $\D = \P_\bolda$, $\bolda = (2, 2, 2, 2, \ldots),$ and $\O = \P_\boldb$, $\boldb = (\infty, 1, \infty, 1, \ldots)$, with $\pi: i \to 2i$.  O'Hara's algorithm exchanges each even part for two parts of half the size.  This is repeated until all parts are odd.  Denote by $V$ the linear transformation of one step in O'Hara's algorithm.  Then
\begin{equation}
V(i,j) = \begin{cases}
1, & i=2m+1,\text{ and } j=2m+1,\\
2, & i=m+1 \text{ and } j=2m+2, \\
0, & \text{otherwise},
\end{cases} \qquad \mbox{ for all $m \geq 0, k \geq 1$\. .}
\end{equation}

Each step in O'Hara's algorithm corresponds to multiplication of the vector of multiplicities by the matrix $V$.  Let $S_k := V^k$ denote $k$ steps in O'Hara's algorithm. We have:
\[ S_k \left(\begin{array}{c} m_1 \\ m_2 \\ \vdots \\ \end{array}\right) = \left( \sum_{i=1}^{k-1} 2^{i-1} m_{2^{i-1}}, 2^k m_{2^{k+1}}, \ldots, \sum_{i=1}^{k-1} 2^{i-1} m_{(2\ell+1) \ts 2^{i-1}}, 2^k m_{(2\ell+1)\ts 2^{k+1}}, \ldots\right)^T. \]

Using pointwise convergence in $\mathbb{H}$, we have:
\[ S_k \to V_\infty \quad \text{ as } \ \, k\to\infty\ts,
\]
where
\[ V_\infty \left(\begin{array}{c} m_1 \\ m_2 \\ \vdots \\ \end{array}\right) \, = \,
\left( \sum_{i=1}^{\infty} 2^{i-1} m_{2^{i-1}}, 0,  \ldots,
\sum_{i=1}^{\infty} 2^{i-1} m_{(2\ell+1) \cdot 2^{i-1}}, 0, \ldots \right)^T.
\]

\subsection{Asymptotic stability}
Let $\psi(y) := \phi(y; 1,1,2)$ denote the integrand for the limit shape of partitions into distinct parts.
\begin{theorem}\label{Glaisher:relation}
Glaisher's linear transformation is \MMstablea\ for $r=1, B=1, a=2,$ with
\[ K\mts\left(t,y,\psi(y)\right) \. = \. \sum_{k=1}^\infty\. 2^{k-1} \psi\left(y/2^{k-1}\right)\mathbbm{1}\mts\mts\left(y \geq t / 2^{k-1}\right), \qquad t > 0\. . \]

Moreover, the limit shapes satisfy
\[ \sqrt{2}\. \Phi\left(x \sqrt{2}\right) \.  = \. \frac{1}{2} \. \sum_{i \geq 1} \. \Psi\mts\left(\frac{t}{2^{i-1}}\right).\]
\end{theorem}

\begin{proof}
Note first that $V_\infty\left(u_j, 2^{k-1} u_j\right) = 2^{k-1}$, so
\[ v\left(\b t, y_j \a \right) \. = \.\begin{cases} 2^{k-1} & \b\th t / 2^{k-1} = y_j\th \a \\ 0 & \text{otherwise.} \end{cases}
\]
Next, we split the terms $\ts v\bigl(\b t, y_j \a\bigr) \. \psi(y_j \a)\ts$ according to the integer~$k$ for which we have \. 
$\b\th t = y_j\th \a\th 2^{k-1}$.  This gives
\begin{align*}
 \sum_{j=1}^\infty v\bigl(\b\ts t, y_j \a\bigr) \ts \psi\bigl(y_j \a\bigr)  & \, = \,
 \sum_{k\geq 1} 2^{k-1} \. \psi\bigl(y_j \a\bigr) \. \mathbbm{1}\mts\mts\left(y_j \geq t/2^{k-1}\right), \quad t>0\. .
\end{align*}

\nin
Note that $V_\infty$ corresponds to a Markov operator {\rm P} acting on the diagram function $\Dhat$,  with
\[
{\rm P} \Dhat(t) \, = \, \frac{\a}{n} \. \sum_{i \geq 1} \sum_{\{j\, :\, 2j-1 \geq t \a/2^{i-1}\}}  2^{i-1} Z_{(2j-1) 2^{i-1}}
\, = \, \sum_{i \geq 1} 2^{i-1} \Dhat\left(t/2^{i-1}\right), \quad t>0\. .
\]
Let $\ts y_j = (2j-1)/\a$ \ts for all $j\ge 1$, so that $\ts \Delta y_j = 2/\a$,
and take the limit of the expectation, as $n\to\infty$. We obtain:
\[
\e \left[{\rm P} \Dhat(t)\right] \, =  \,
\sum_{i\geq 1} \. \frac{\a^2}{2n} \. \sum_{j : y_j \geq t/2^{i-1}} 2^{i-1} \e Z_{y_j \a 2^{i-1}} \Delta y_j \, \.
\longrightarrow \, \. \frac{1}{2} \. \sum_{i \geq 1} \Psi\mts\left(t/ 2^{i-1}\right), \quad t>0\. ,
\]
which is finite for each $t>0$ if and only if $\a = \sqrt{n}$.
Finally, note that also directly we have
\[
\e \left[{\rm P} \Dhat(t)\right] \to \sqrt{2} \. \Phi\left(t\.\sqrt{2}\right),
\]
which is the standard expression for the limit shape of partitions into odd parts.
\end{proof}

This example also demonstrates that the function $K$ in equation~\eqref{v nice} is not always multiplicative with respect to its final argument.

% Geometric bijection
\subsection{Geometric bijection}
Another way to understand the connection between the limit shapes of $\mathcal{O}$ and $\mathcal{D}$ is via the construction of a geometric bijection.  In this case, the bijection is more complicated, see \cite[Figure~16]{PakNature}, but the same rules apply to the limit shape.  A summary of the mappings is below.
Note that instead of mapping the diagram function at each stage, we have instead chosen to demonstrate how the bijection acts on the conjugate diagram function.
We have:
\begin{align*}
\sqrt{2}\Phi(x\sqrt{2}) & \longmapsto  (\sqrt{2}\Phi(x\sqrt{2}), \sqrt{2}\Phi(x\sqrt{2}) ) \\
 & \longmapsto  \Scale[1]{\left( \sqrt{2}\Phi(x\sqrt{2}) - 2x, \left(\frac{\sqrt{2}\Phi(x\sqrt{2})}{2}, \frac{\sqrt{2}\Phi(x\sqrt{2})}{2}\right) \right)} \\
 & \longmapsto  \Scale[1.2]{\left( \left(\frac{\sqrt{2}\Phi(x\sqrt{2}) - 2x}{2}, \frac{\sqrt{2}\Phi(x\sqrt{2}) - 2x}{2} \right), \left( \left(\frac{\Phi(x\sqrt{2})}{\sqrt{2}}\right)^{\<-1\>}, \left(\frac{\Phi(x\sqrt{2})}{\sqrt{2}}\right)^{\<-1\>} \right)\right)} \\
 & \longmapsto  \Scale[1.2]{\left(\frac{1}{2} \left(\frac{\sqrt{2}\Phi(x\sqrt{2}) - 2x}{2} + \frac{\sqrt{2}\Phi(x\sqrt{2}) - 2x}{2} \right), \left( \frac{\Phi(x\sqrt{2})}{\sqrt{2}}-x, \frac{\Phi(x\sqrt{2})}{\sqrt{2}}-x \right)\right)} \\
\small & \longmapsto  \Scale[1.2]{\left(\frac{1}{2} \left(\frac{\sqrt{2}\Phi(x\sqrt{2}) - 2x}{2} + \frac{\sqrt{2}\Phi(x\sqrt{2}) - 2x}{2} \right),  \frac{1}{2} \left(\frac{\Phi(x\sqrt{2})}{\sqrt{2}}-x+ \frac{\Phi(x\sqrt{2})}{\sqrt{2}}-x \right)\right)} \\
 & \longmapsto  \left(\frac{\Phi(x\sqrt{2}) - 2x}{\sqrt{2}},  \frac{\Phi(x\sqrt{2}) - 2x}{\sqrt{2}} \right)  \ \longmapsto \  \left(\frac{\Phi(x\sqrt{2})}{\sqrt{2}}-x \right) \, = \, \Psi(x)\ts,
\end{align*}%
where the final equality comes from a straightforward rearrangement of terms.

% Stanton's generalization
\subsection{Stanton's generalization}
\label{od general}
In this section, $r \geq 1$, and $m \geq 2$ are integers.
The following is a generalization of Theorem~\ref{Glaisher:relation}.

\begin{theorem}\label{Stanton:relation}
Let $\A$ denote the set of partitions into perfect $r^{th}$ powers with multiplicity at most $m^r-1$, with limit shape given by $\Phi(t; r, 1, m^r)$.  Let $\B$ denote the set of partitions into perfect $r^{th}$ powers not divisible by $m^r$, with limit shape given by $\Phi\left(t; r, \left(\frac{m}{m-1}\right)^r \right).$
Then we have
\begin{equation}\label{Stanton:equation} (m-1) \sum_{k \geq 1} m^{r(k-1)-k} \Phi\left(t\, m^{r(k-1)}; r, 1, m^r\right) = \Phi\left(t; r, \left(\frac{m}{m-1}\right)^r\  \right), \qquad t>0\. .
\end{equation}
\end{theorem}

Note that the limit shapes of $\A$ and $\B$ in Theorem~\ref{Stanton:relation} follow from Theorem~\ref{special unrestricted} and Theorem~\ref{special distinct}, respectively.
The relation in equation~\eqref{Stanton:equation} is obtained from the following generalization of Euler's theorem.

\begin{theorem}[Stanton \cite{Stanton}] \label{q prop}Let $\{u_k\}_{k\geq1}$ be a sequence of distinct positive integers, and let $m\geq 2$ be an integer.  If $u_{mk}/u_k = m_k$ is an integer for all $k$, then the number of integer partitions of $n$ into parts of sizes avoiding $\{u_{mk}\}_{k\geq 1}$ is equal to the number of integer partitions of $n$ into parts of size $\{u_k\}_{k\geq1}$, where $u_k$ has multiplicity at most $m_k-1$.
\end{theorem}

Theorem~\ref{q prop} also follows by Andrews's Theorem \ref{Andrews Theorem} using the mapping $\pi: u_{mk} \to m_k u_k$.  The generalization of Glaisher's bijection in this case is replacing parts of size $u_{mk}$ into $m_k$ parts of size $u_k$.   When $m_k = m =2$ and $\{u_k\}_{k\geq 1} = \{2,4,6,\ldots\}$, we obtain Glaisher's bijection.

Define \emph{Stanton's $(r,m)$-bijection} $\varphi_{r,m}:~\A~\to~\B$ as follows.  Replace parts of size $(k\,m^j)^{r}$ into $m^{j\,r}$ parts of size $k^r$, for all $k \geq 1$ not divisible by $m^r$.

\begin{theorem}
Stanton's $(r,m)$-bijection $\varphi_{r,m}$ is an \MMbns.
\end{theorem}
\begin{proof}
Stanton's $(r,m)$-bijection can be written using the transformation matrix $V_{r,m}$ with entries
\begin{equation}\label{Stanton:transformation}V_{r,m}(i,j) = \begin{cases} (m^r)^{ w - 1} & \text{ for  $ i = k^r$ and  $j = m^{r\,(w-1)} k^r$ } \\
 0 & \text{otherwise} \end{cases}, \end{equation}
 for all $k \geq 1$ not divisible by $m^r$ and $w \geq 1$.
\end{proof}
Theorem~\ref{Stanton:relation} then follows from Equation~\eqref{Stanton:transformation} in a similar fashion as the proof of Theorem~\ref{Glaisher:relation}.

% Lebesgue's identity
\bigskip\section{Lebesgue's identity}\label{Lebesgue}

% Bressoud's bijection
\subsection{Limit shapes and Bressoud's bijection}

\begin{figure}
\scalebox{0.8}{
\includegraphics[scale=0.57]{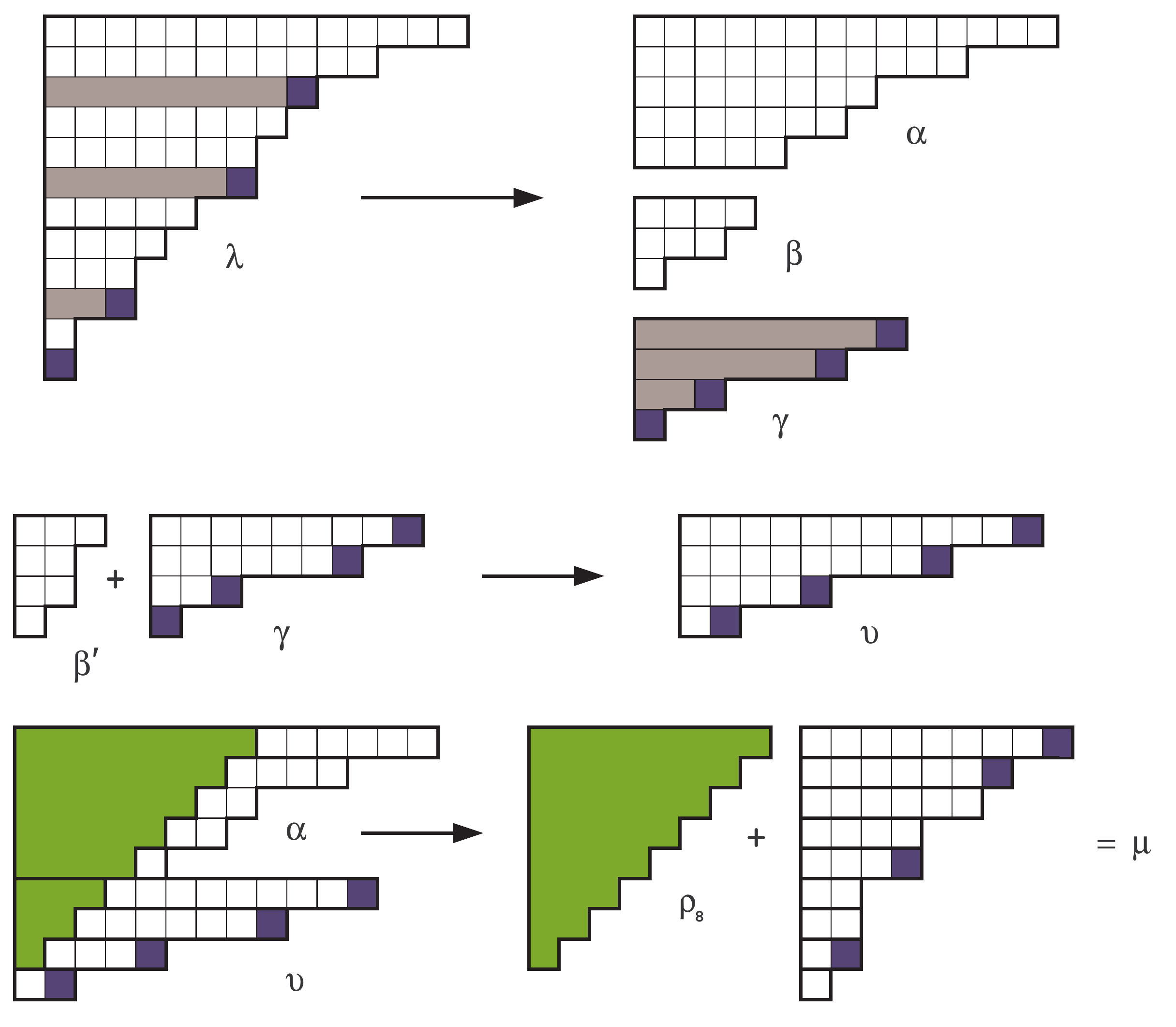}
}
\caption{An example of a bijection $\vp: \la \to \mu$
for  $\la \in \ca_{n}$, $\mu \in \cl_{n}$. }
\label{Bressoud}
\end{figure}

In this section, we use a bijection due to Bressoud, see~\cite[4.3.2]{Pak}, to obtain the limit shape below.

\begin{theorem}\label{Lebesgue:theorem}
Let $1 \leq \ell < k$.  Let $\mathcal{L}^{\ell,k}$ denote the set of partitions $\mu$ into parts congruent to $0$ or $\ell$ mod $k$ such that parts differ by at least $k$, and parts congruent to $\ell$ mod $k$ differ by at least $2k$.
The limit shape of $\mathcal{L}^{\ell,k}$ is given by
\begin{equation}\label{Lebesgue:shape}
\frac{2 \sqrt{2}}{\pi\sqrt{k}} \. \log\left[\frac{1}{2} \left(1+e^{-\frac{\pi  t}{2 \sqrt{2\, k}}}+\sqrt{1+e^{-\frac{\pi  t}{\sqrt{2\, k}}}+6 e^{-\frac{\pi  t}{2 \sqrt{2\, k}}}}\right)\right], \quad t \geq 0\. .
\end{equation}
\end{theorem}

Let $\A_n$ denote the set of partitions of size~$n$ into distinct parts which are congruent to~$0, 1$ or~$2$ modulo~$4$.
Let $\mathcal{L}_n$ denote the set of partitions $\mu$ of size~$n$ with consecutive parts $\mu_i - \mu_{i-1} \geq 2$ for even part sizes $\mu_i$ and $\mu_i - \mu_{i-1} \geq 4$ for all odd part sizes $\mu_i$, $i \geq 1$.  \ts Recall \emph{Bressoud's bijection} \ts
$\varphi : \mathcal{A}_n \to \mathcal{L}_n$  illustrated in Figure~\ref{Bressoud} (see~\cite{Pak}).
The bijection $\varphi$ is used in~\cite[\S4.3.2]{Pak} to prove Lebesgue's identity, given below.

\begin{theorem}[See~\cite{Pak}]  We have:
\[ \sum_{\ell=1}^\infty t^{\binom{\ell+1}{2}}\frac{(1+z\ts t)(1+z\ts t^2) \cdots (1+z\ts t^\ell)}{(1+t)(1+t^2) \cdots (1+t^\ell)} = \prod_{i=1}^\infty \left(1+z\ts t^{2i}\right)\left(1+t^i\right). \]
\end{theorem}

\begin{figure}
\begin{center}\includegraphics[scale=0.35]{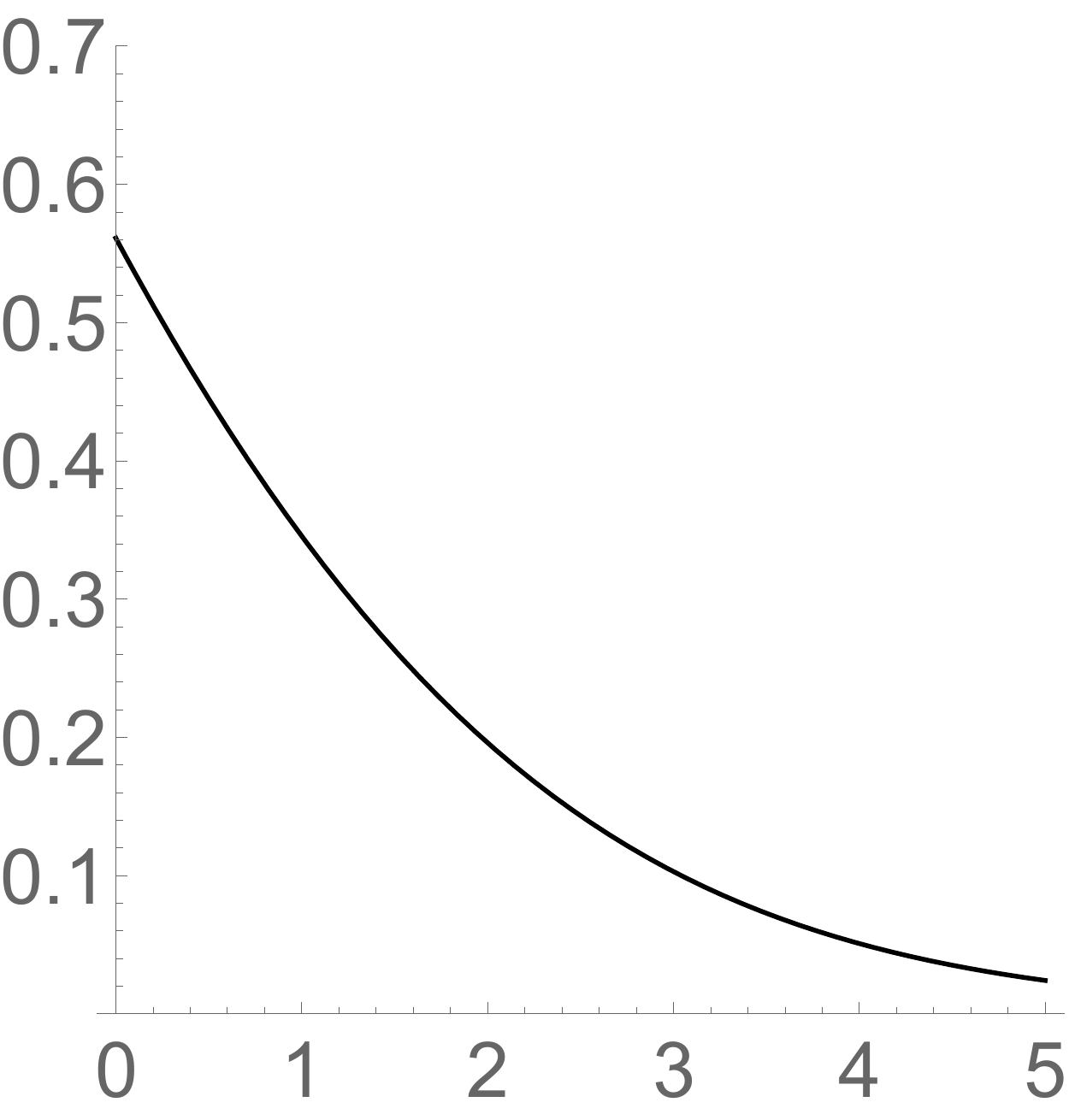}\end{center}
\caption{Plot of the limit shape of $L^{1,2}$, the set of partitions $\mu$ with consecutive parts $\mu_i - \mu_{i-1} \geq 2$ for even part sizes $\mu_i$ and $\mu_i - \mu_{i-1} \geq 4$ for all odd part sizes $\mu_i$, $i\ge 1$.}
\label{Lebesgue:plot}
\end{figure}

The proof of Theorem~\ref{Lebesgue:theorem}, using $k=2$ and $\ell=1$ for simplicity, is explicitly worked out below, and a plot appears in Figure~\ref{Lebesgue:plot}.
Whereas the value of $\ell$ is an important aspect of the bijection, the limit shape is invariant to the particular value of $\ell$ chosen.

% Proof of Lebesgue's theorem
\subsection{Proof of Theorem~\ref{Lebesgue:theorem}}
Let $\A$ denote the set of partitions into distinct parts which are congruent to $0, 1$ or~$2$ modulo~$4$.  
Let~$\mathcal{L}$ denote the set of partitions~$\mu$ with consecutive parts $\ts\mu_i - \mu_{i-1} \geq 2$ \ts 
for even part sizes $\mu_i$ and $\mu_i - \mu_{i-1} \geq 4$ for all odd part sizes~$\mu_i$, $i\ge 1$.  
By Theorem~\ref{special distinct}, the limit shape of $\A$ is given by
$$s(x) \. := \, \sqrt{\frac{3}{4}} \, \Psi\left(x\, \sqrt{\frac{3}{4}}\right) 
\, = \, \frac{3}{\pi} \. \ln\left(1+e^{-\frac{\pi}{4}x}\right). 
$$ 
Since the bijection to~$\mathcal{L}$ is given by a geometric bijection, see~\cite[Corollary 10.2]{PakNature}, 
the limit shape of $\mathcal{L}$ exists, and we now derive its explicit form.
We are unaware of its explicit form appearing in the literature.

In Figure~\ref{Bressoud}, each non-blue square is the equivalent of two dots side by side, and each blue square is the equivalent of one dot.  Asymptotically, this distinction is negligible, but for the bijection it is essential.
The first step in the bijection is to separate all of the parts divisible by three; call this partition $\gamma$, and the partition consisting of the remaining parts by $\alpha\cup \beta$.

The first step of the bijection, see Remark~\ref{stable:remark} below, is the map
\[ \begin{array}{cll}
s(x) & \. \longmapsto \. \left(\frac{2}{3} s(x), \frac{s(x)}{3}\right) & (x\geq 0, x \geq 0)\. ,
\end{array}
\]
which corresponds to separating the limit curves of $\alpha \cup \beta$ and $\gamma$, respectively, in Figure~\ref{Bressoud}. %separates out the parts congruent to 1 modulo 3.
Note that since $s(x) = \frac{3}{\pi} \ln\left(1+e^{-\frac{\pi x}{4}}\right)$, we have $s(0) = \frac{3}{\pi}\ln(2)$\. .

\begin{figure}
	\begin{subfigure}[l]{0.3 \textwidth}
	\includegraphics[scale=.3]{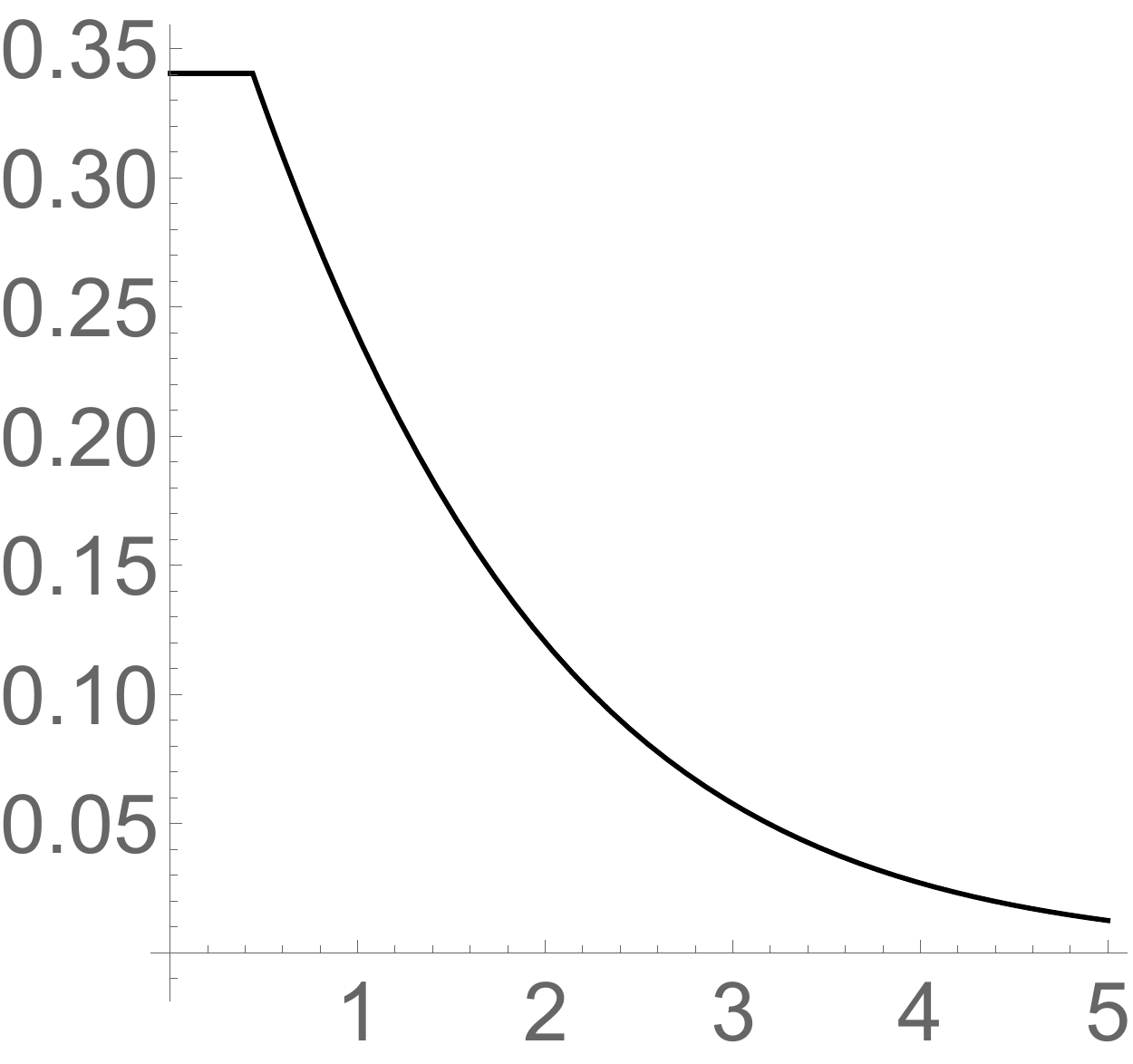}
	\caption{$a(x)$}
	\end{subfigure}
	\begin{subfigure}[l]{0.3 \textwidth}
	\includegraphics[scale=.3]{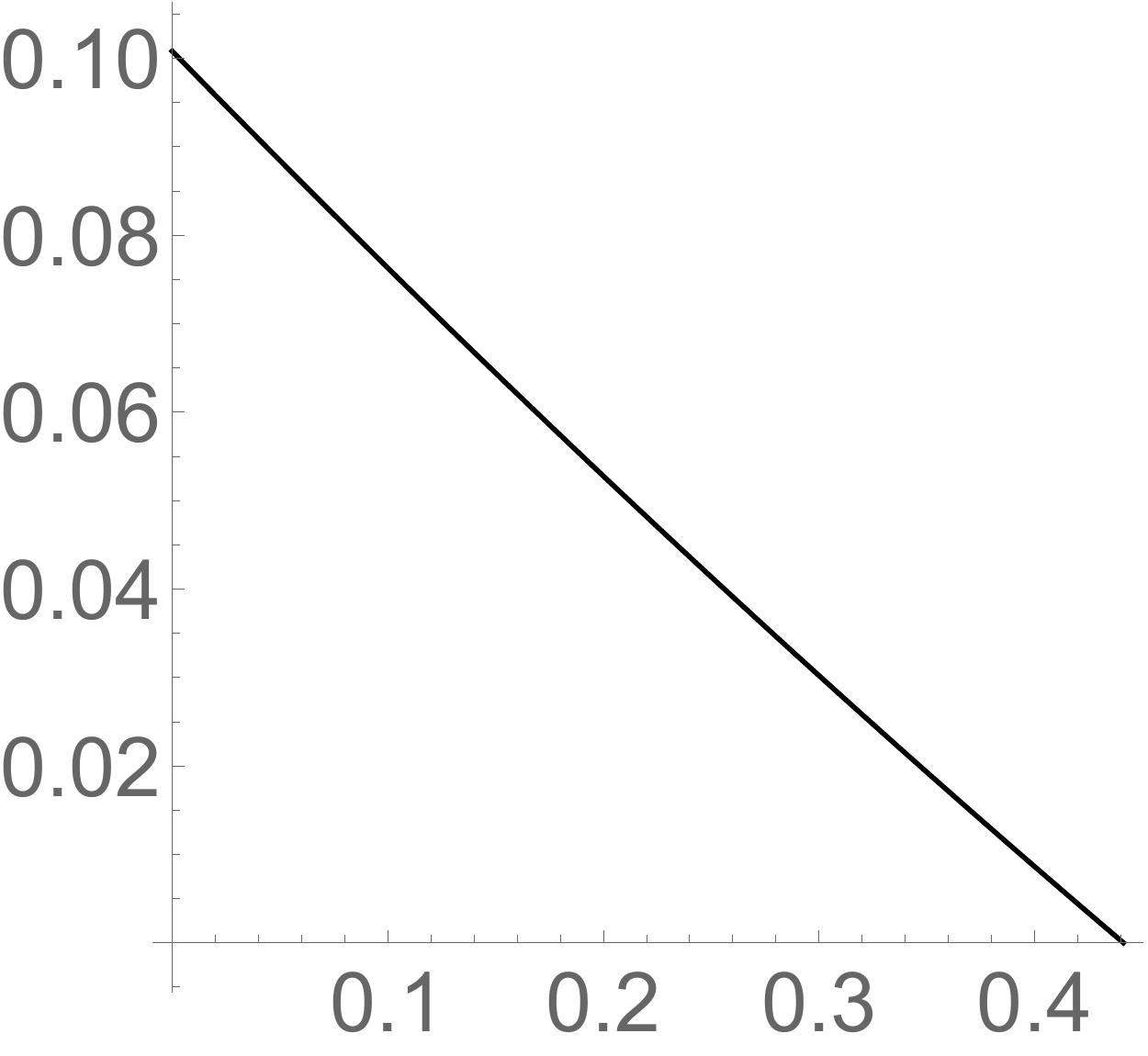}
	\caption{$b(x)$}
	\end{subfigure}
	\begin{subfigure}[l]{0.3 \textwidth}
	\includegraphics[scale=.3]{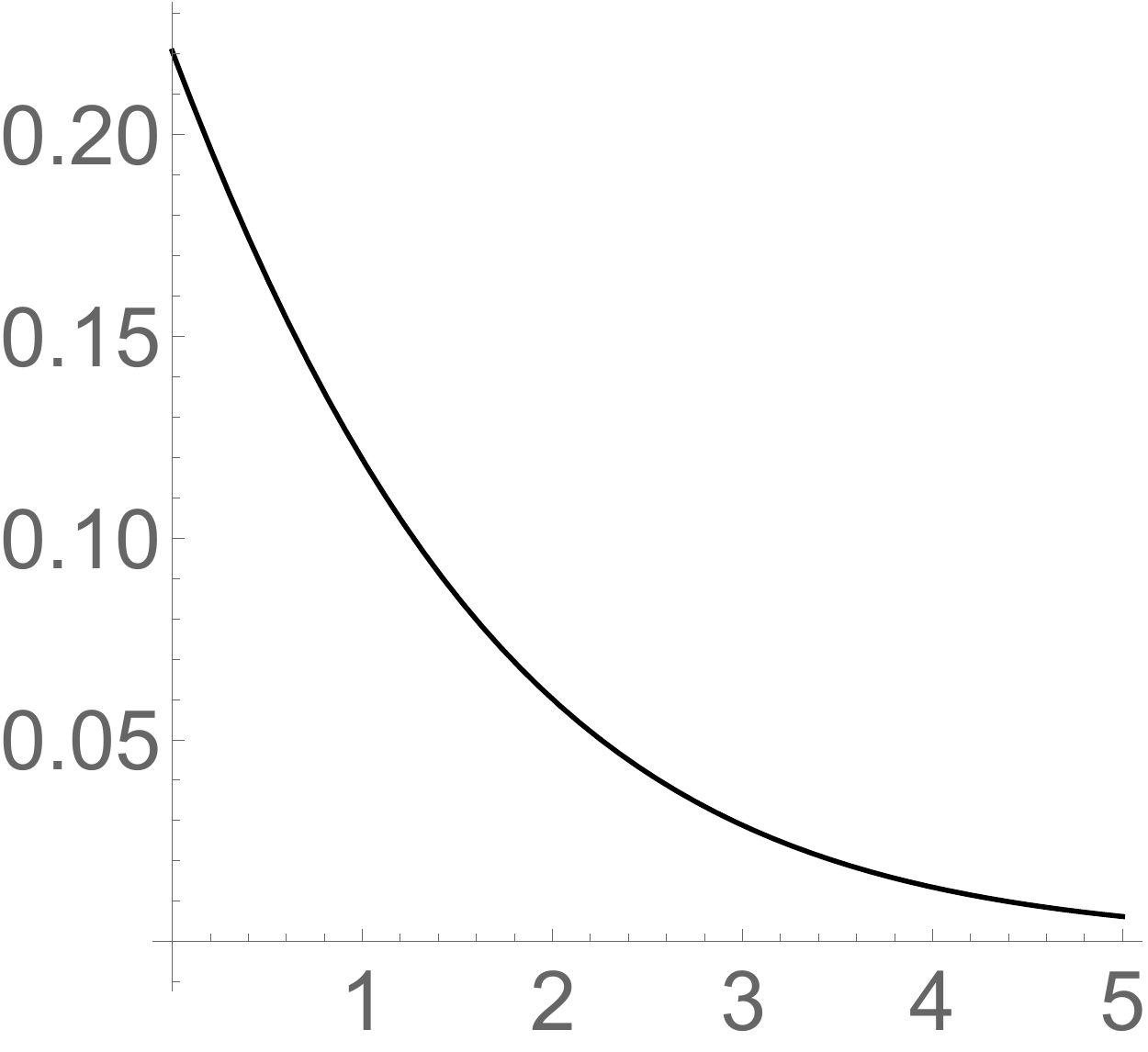}
	\caption{$c(x)$}
	\end{subfigure}
\caption{Plots of limit curves corresponding to the different parts of Bressoud's bijection.}
\end{figure}

The next step in the bijection is to separate out partitions $\alpha$ and $\beta$ from $\alpha\cup \beta$.  The bijection places all parts in $\alpha\cup\beta$ which are $\leq 2\ell(\gamma)$ in $\beta$, and leaves the rest in $\alpha$.  Since $\ell(\gamma)$ is an asymptotically stable statistic, with $2\ell(\gamma)/\sqrt{n} \to 2s(0)/3$, we may further separate limit curves $a(x), b(x), c(x)$ corresponding to $\alpha, \beta, \gamma$ as follows:
\[ s(x) \longmapsto  (a(x), b(x), c(x))\. , \]
where
\begin{align*}
 a(x) \. = \. \begin{cases} \eta_0, \quad & 0 \leq x \leq \frac{2s(0)}{3} \\
 \frac{2}{3}s(x), \quad & x > \frac{2s(0)}{3} \end{cases}, \qquad
 b(x) \. = \. \begin{cases} \frac{2}{3}s(x) - \eta_0, & 0\leq x \leq \frac{2s(0)}{3} \\ 0 & \text{otherwise}\end{cases},
\end{align*}
$c(x) \ts = s(x)/3$ \ts for \ts $x \geq 0$, \ts and
$$
\eta_0 \. := \, \frac{2}{3}\. s\bigl(2s(0)/3\bigr) \, = \, \frac{2}{\pi} \. \ln\left(1+\frac{1}{\sqrt{2}}\right).
$$

\nin
We also  need the following inverses:
\begin{align*}
 a^{-1}(x)  \, = \, \begin{cases} -\frac{4 \ln \left(-1+e^{\frac{\pi  x}{2}}\right)}{\pi}, & 0 \leq x \leq \eta_0 \\
0 & x > \eta_0 \end{cases}, \qquad
 c^{-1}(x) \, =  \, \frac{-4 \ln\left(-1+e^{\pi  x}\right)}{\pi }, \ \   0 \leq x \leq \frac{s(0)}{3}\..
 \end{align*}

\nin
The limit curve of $\nu$ is given by
\begin{align*}
 v(x) \, = \,2b(2x) + c^{-1}(x) \, = \, -2x-\frac{2}{\pi}\ln\left(\frac{3}{2}+\sqrt{2}\right)-\frac{4}{\pi}\log\left(e^{\frac{\pi}{2} x}-1\right),
 \quad 0 \leq x \leq \frac{s(0)}{3}\..
\end{align*}

\nin
The limit curve of $\rho_s$ is the line $2(\eta_0+s(0)/3 - x)$, for $0\leq x \leq \eta_0+s(0)/3$.
We let $x_0 := \eta_0+s(0)/3.$
Now we subtract out the staircase by defining $\delta := \alpha \setminus \rho_s$ and $\xi := \nu \setminus \rho_s$.  Denote the limit curve of $\delta$ by $d(x)$ and the limit curve of $\xi$ by $e(x)$.
We have
\begin{align*}
 d(x) & \, = \, \left( a^{-1}(x) - 2(x_0 - x)\right) = -\frac{4}{\pi} \ln\left(e^{\frac{\pi}{2}x}-1\right) - 2x_0 + 2x, \qquad 0 \leq x \leq \eta_0, \\
 e(x) & \, = \, v(x) - 2(s(0)/3 - x) \, = \, -2x_0 - 2x - \frac{4}{\pi}\ln\left(e^{\frac{\pi}{2} x} - 1\right), \qquad 0 \leq x \leq s(0)/3.
 \end{align*}

\begin{figure}
\includegraphics[scale=0.3]{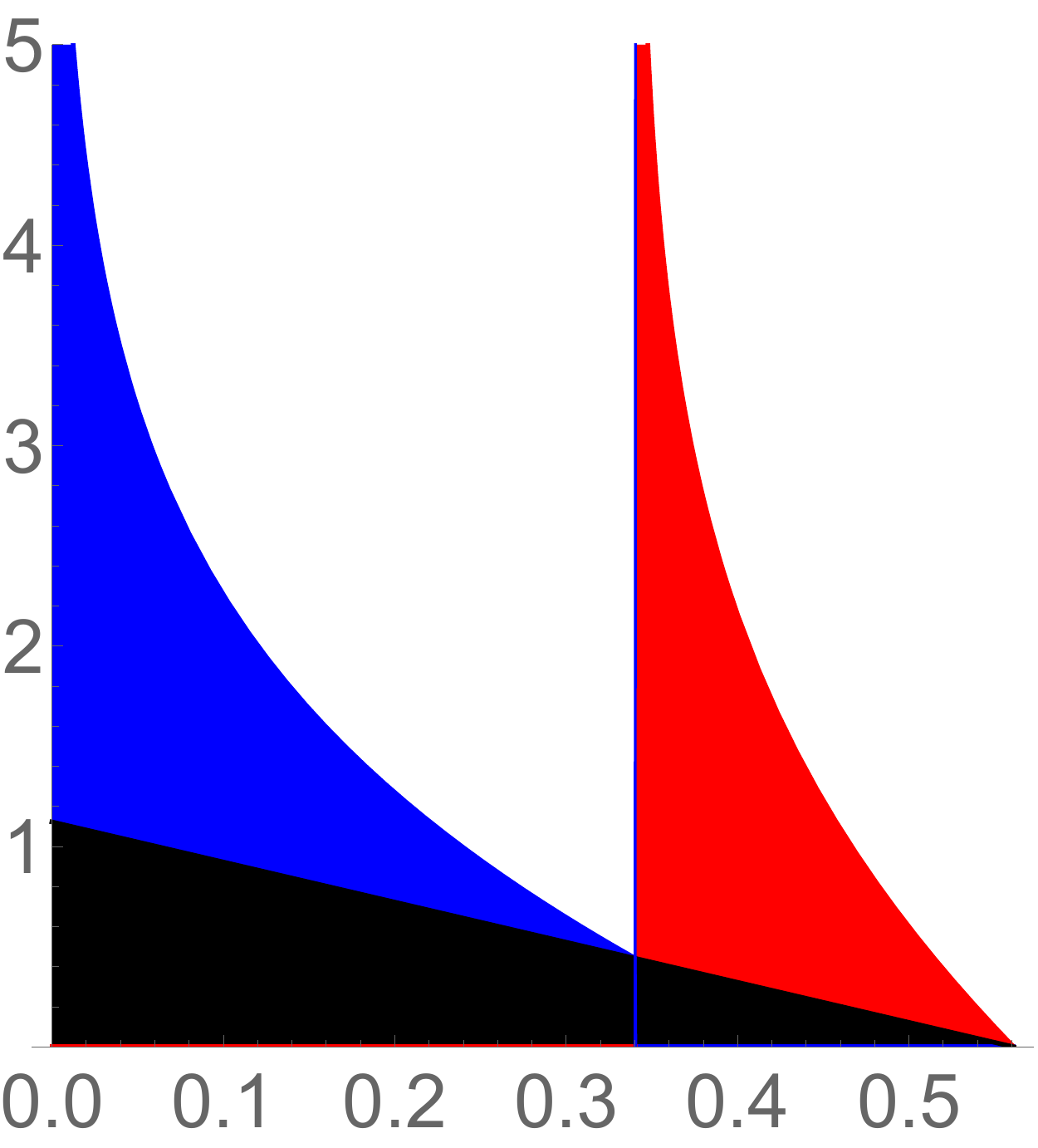}
\caption{Limit curves of $\xi$ ({\color{red}red}), $\rho_s$ ({\color{black}black}), and $\delta$ ({\color{blue}blue}).}
\label{triangle}
\end{figure}

Next, we sort the parts of partitions $\delta$ and $\xi$, which corresponds to a union of parts;
call this partition $\theta$, with corresponding limit curve $t(x)$.
First, we have
\begin{align*}
d^{-1}(x) \, = \, -\frac{2}{\pi} \ln\left(1 - e^{-\frac{\pi}{4}\left(x + x_0\right)}\right), \quad
e^{-1}(x) \, = \, \frac{2}{\pi} \ln\left(1 + e^{-\frac{\pi}{4}\left(x + x_0\right)}\right), \qquad x \geq 0\.,
\end{align*}
and
\begin{align*}
 t^{-1}(x) & \, = \, \left( d^{-1}(x) + e^{-1}(x) \right)  \, = \,
 \frac{2}{\pi}\. \ln\ts \coth\left(\frac{\pi}{8}(x+x_0)\right), \qquad x \geq 0\..
\end{align*}
Note that $\ts\coth^{-1}(x) \. = \. \frac{1}{2} \ts \ln\left(\frac{x+1}{x-1}\right)$ \t for $|x|>1$.  Therefore,
\begin{align*}
t(x) & \, = \, -2x_0\ts + \ts\frac{4}{\pi} \ts \ln\left(\frac{e^{\frac{\pi}{2}x}+1}{e^{\frac{\pi}{2}x}-1}\right), \qquad 0 \leq x \leq x_0\..
\end{align*}

The last step of the bijection adds partitions $\theta$ and $\rho_s$, and corresponds to the conjugate limit shape of $\mathcal{L}$.
Namely,
\begin{align}
\label{minv}
%\nonumber
m^{-1}(x) & \, = \, 2(x_0-x) + t(x) \, = \,
\frac{4}{\pi} \. \ln\left(\frac{e^{-\frac{\pi}{2}x}+1}{e^{\frac{\pi}{2}x}-1}\right), \quad 0 \leq x \leq x_0\..
\end{align}

\nin
We can also invert to obtain the limit shape of~$\mathcal{L}$.
To obtain the inverse, we solve for $w$ in
\[
u \. = \. \frac{w+1}{w^2-w}\.,
\]
where $\ts w = e^{\frac{\pi}{2} m(x)}$, and $\ts u = e^{\frac{\pi}{4}x}$.
The positive solution for positive $w$ is of the form
\[
w \. = \. \frac{1+u+\sqrt{1+6 u+u^2}}{2 u}\. .
\]
The limit shape of $\mathcal{L}$ is thus given by
\begin{equation}\label{local:m}
 m(x) \, = \, \frac{2}{\pi} \.
 \log\left( \frac{1+e^{-\frac{\pi}{4} x}+\sqrt{1+6e^{-\frac{\pi}{4} x}+e^{-\frac{\pi}{2}x}}}{2}\right), \qquad x \geq 0\..
 \end{equation}
This completes the proof of Theorem~\ref{Lebesgue:theorem}.\qed

\begin{remark}\label{stable:remark}{\rm
Note that the limit shapes of $\gamma$ and $\alpha\. \cup\. \beta$ do not follow from any of the transformations defined in \S \ref{geometric transformations}.
Instead, the limit shapes exist in the stated form by Section~\ref{sect prob}.
First, the expected number of parts divisible by three over any collection of $(b-a)\sqrt{n}$ parts, $0<a<b$, in the joint distribution of independent random variables exists and is given by $c(x)$.
Appealing again to Section~\ref{sect prob}, we see that the pointwise convergence of expectation of the independent random variables is sufficient to guarantee convergence in probability for the set of integer partitions.
}\end{remark}

\subsection{Proof of corollaries~\ref{number:of:parts:cor}~and~\ref{durfee:square:cor}}
From the limit shape given in equation~\eqref{local:m}, we can now obtain various statistics.
For example, the expected number of parts in a random partition in $\mathcal{L}$ is given by $m(0)$, which corresponds to equation~\eqref{number:of:parts} in Corollary~\ref{number:of:parts:cor}.

To obtain Corollary~\ref{durfee:square:cor}, we solve either for $m(x) = x$ or $m^{-1}(x) = x$.  In this case, solving for $m^{-1}(x)=x$ using equation~\eqref{minv} is the simpler calculation, from which equation~\eqref{durfee:square} follows.

% ***********************************************Minimal Difference D ***********************************************
\bigskip\section{Minimal difference $d$ partitions}
\label{diff d}
In this section, let $d$ denote a fixed positive integer, i.e.~$d \geq 1$.

\begin{theorem}[cf.~\S\ref{Romik:example} and \S\ref{Romik:mistake}]\label{diff:d:theorem}
Let $\mathcal{D}_d$ denote the set of partitions $\mu$ in which $\mu_i - \mu_{i-1} \geq d$; that is, parts are at least $d$ apart and there is no restriction on the number of parts.
Let $D_d(x)$ denote the limit shape of $\mathcal{D}_d$.  Then we have
\begin{equation}
\label{diff d inv} D^{(-1)}_d(x) \,= \,\frac{w}{c} \.\log \left[ \left(e^{\frac{c\, \gamma}{w}\, d} - e^{\frac{c\, \gamma}{w}\, (d-1)}\right) \frac{e^{-\frac{c\, d}{w} x}}{1-e^{-\frac{c}{w} x}}\right], \qquad x>0\.,
\end{equation}
where
\begin{equation}\label{gamma} w = \sqrt{1 - \frac{d}{2} \gamma^2}\ , \qquad \gamma = \frac{- \log(1-y_d)}{\sqrt{\text{\rm Li}_2(y_d) + \frac{d}{2} \log^2(1-y_d)}} \. , \end{equation}
where $y_d$ is the unique solution to $(1-y)^d=y$ in the interval $(0,1)$, and $c$ is the solution to
\[ \qquad c = \sqrt{\text{\rm Li}_2\left(1 - e^{-\frac{c\, \gamma}{w}}\right)}\ .\]
\end{theorem}

Here $\text{\rm Li}_2(x)$ is the usual \emph{dilogarithm function}, which for positive real-valued inputs $x$ is defined as
\[ \text{\rm Li}_2(x) \. = \, \sum_{k \geq 1} \. \frac{x^k}{k^2}\ . \]

It is easy to obtain a minimal difference $d$ partition: take any partition $\lambda_1 \geq \lambda_2 \geq \ldots \geq \lambda_{\ell}$ and apply the transformation $\lambda_i \longmapsto  \lambda_i + d\cdot (\ell-i)$, for all $i=1,\ldots,\ell(\lambda)$.
This mapping, however, is not size-preserving, as it is between unrestricted partitions of $n$ into exactly $k$ parts, and minimal difference $d$ partitions of $n+d\binom{k}{2}$ into exactly $k$ parts.
Geometrically, we are adding a right triangle with slope $d$ whose right angle lies at the origin to the \emph{conjugate} limit shape.

\begin{theorem}[cf.~\S\ref{Romik:mistake}] \label{theorem diff d k}
Let $\mathcal{D}_{d,k}$ denote the set of partitions $\mu$ into exactly $k$ parts, in which $\mu_i - \mu_{i-1} \geq d$ for $i=1,\ldots, k$.
Suppose $k \sim z \sqrt{n}$\. , and $d \geq 1$.
Let $D_{d,z}(x)$ denote the limit shape of $\mathcal{D}_{d,k}$.
Then we have
\begin{equation}
\label{diff d inv1} D^{-1}_{d,z}(x) \, = \,
\frac{w}{c} \. \log \left[ \left(e^{\frac{c\, z}{w}\, d} - e^{\frac{c\, z}{w}\, (d-1)}\right) \frac{e^{-\frac{c\, d}{w} x}}{1-e^{-\frac{c}{w} x}}\right], \qquad x>0\. ,
\end{equation}
where $D_{d,z}(D_{d,z}^{-1}(x) = x$.
\end{theorem}
\begin{proof}
Let  $\mathcal{D}_{n,d,k}$ denote the set of minimal difference $d$ partitions of $n$ into exactly $k$ parts.
Let $\mathcal{P}_{n,k}$ denote the set of unrestricted partitions of $n$ into exactly $k$ parts.
We have $\ts |\mathcal{D}_{n,d,k}| \ts = \ts |\mathcal{P}_{n-d\binom{k}{2},k}|$,
so we must tilt the distribution of the random size of the partition to have expected value
$n-d\binom{k}{2}$ and exactly $k$ parts.

By conjugation, each partition in $\mathcal{P}_{n,k}$ corresponds to an unrestricted partition with largest part at most $k$.
Hence, for $k \sim z \sqrt{n}\. ,$ $z >0$, we may generate the conjugate limit shape of $\mathcal{P}_{n,k},$ which we denote by $g_z$, directly using the method of Section~\ref{sect prob}.
Let $g_z(t)$ denote the limit shape of $\mathcal{P}_{n,k}$ for $k=z\sqrt{n}$, $z>0$.  We have
\[g_z(t) \, = \, \int_0^{z} \. \frac{e^{-c\, y}}{1-e^{-c\, y}}\. dy \, 
= \, \frac{1}{c}\.\ln(1 - e^{-c\, z}) \. - \. \frac{1}{c}\ln(1 - e^{-c\, t}), \qquad 0<t \leq z,
\]
where $c$ is the constant defined by the equation $\ts\int_0^{z}\ts g(t)\ts dt = 1.$
In~\cite{Romik}, it is shown that the constant $c$ satisfies
\[
c^2 \. = \. \text{\rm Li}_2\left(1 - e^{-c\, z}\right).
\]

However, rather than tilt the distribution of a random partition to have expected size~$n$, we instead tilt the distribution so that the expected size of the partition is $\ts n - d\binom{k}{2}$.
This changes the limit shape from $g_z(t)$ to a limit curve, say $h_{z,w}(t)$, given by
\[
h_{z,w}(t) \, = \, \frac{w}{s}\.\ln(1 - e^{-s \frac{z}{w}}) \. - \. \frac{w}{s}\.\ln(1 - e^{-s \frac{t}{w}}), \quad t >0\.,
\]
where $s$ is such that
\[ s^2 \, = \, \text{\rm Li}_2\left(1 - e^{-s\, \frac{z}{w}}\right).
\]
The final step is to add the triangle of slope $d$, i.e., the curve $d(z - t)$, to $h_{z,w}(t)$.
Rearranging the terms yields equation~\eqref{diff d inv}.
\end{proof}

Theorem~\ref{diff:d:theorem} then follows by the result below, c.f.~\S\ref{Romik:mistake}.

\begin{prop}[\cite{Romik}]
A random partition in $\mathcal{D}_d$ has number of parts asymptotic to $\gamma \sqrt{n}$, where $\gamma$ is given by equation~\eqref{gamma}.
\end{prop}

% ***********************************************Further Examples***********************************************
\bigskip\section{Further examples}
\label{sect further}

% Even with restrictions
\subsection{Partitions into an even number of parts $\leq k$}\label{further:example}

This section presents a bijection which requires one additional transformation before applying the conjugation transformation, see Figure~\ref{stretch}.

\begin{figure}
\includegraphics[scale=0.77]{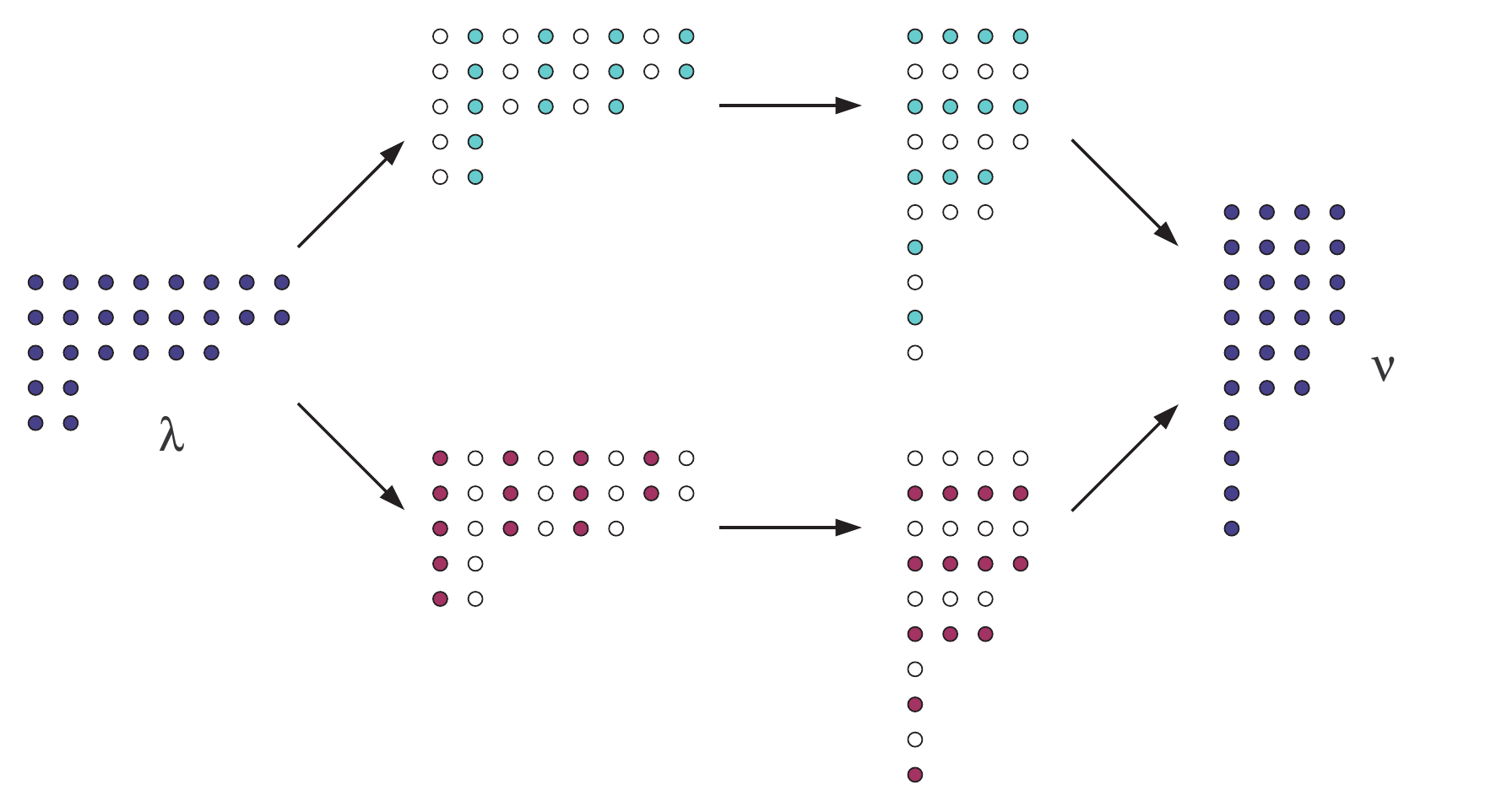}
\caption{The geometric bijection between $\A_{n,k}$, the set of partitions into even parts such that the largest part has size $\leq 2k$, and $\B_{n,k}$, the set of partitions into even parts such that the number of parts is $\leq k$.}
\label{stretch}
\end{figure}

\begin{theorem}[\cite{PakNature}] \label{Ank}\label{Bb:theorem}
Let $\B_{n,k}$ denote the set of partitions into even parts such that the number of parts is $\leq k$.
 Suppose $k/\sqrt{n}\to b>0$, and let $\B_b = \cup_n \B_{n, k}$.  Then the limit shape of $\B_b$ is given by
\[
G(t) := \int_t^{2b} \frac{2e^{- c\, 2y}}{1-e^{- c\, 2y }} \, dy, \qquad t>0,
\]
where $c$ is the constant such that $\int_0^\infty G(t)\, dt = 1$.
\end{theorem}

Let $\A_{n,k}$ denote the set of partitions into even parts such that the largest part has size $\leq 2k$.
The map $\vartheta_n: \A_{n,k} \to \B_{n,k}$ given in \cite[Section 5]{PakNature} is defined as follows: divide each part of $\lambda\in \B_{n,k}$ by 2, double the resulting multiplicities, and take the conjugate.
\begin{theorem}[\cite{Pak}]
For each $n \geq 1$ and $1 \leq k \leq n$, the map $\vartheta_n$ is a bijection between sets $\A_{n,k}$ and $\B_{n,k}$.
\end{theorem}

\begin{theorem}
Bijection $\vartheta_n$ is an \MPbns.
\end{theorem}
\begin{proof}
The bijection $\vartheta_n$ corresponds to the transformation matrix
\begin{equation}
\label{even V}
 V(i,j)\. = \. \begin{cases} 2, & \ i\geq 1 \ \text{ and } \ j=2i+2k \\
 0 & \ \text{otherwise} \end{cases}, \qquad k \geq 0.
 \end{equation}
Furthermore, this transformation $V$ can also be realized as a composition of two transformations, $V = V_1 \ts V_2$, where
\[
V_1(i,j) \. = \. \begin{cases} 1, & \ i \geq 1, \ j\ge i \\
0, & \text{ otherwise,} \end{cases}\qquad
\]
\[
V_2(i,j) \. = \. \begin{cases} 2, & \ i \geq 1 \ \text{ and } \ j=2i \\
 0, & \ \text{otherwise}\ts. \end{cases}\qquad \qquad
\]
We have $V_1$ is the conjugate transformation, and $V_2$ is the transformation \ts
$(\lambda/2\cup \lambda/2)$ \ts defined above.
\end{proof}

\begin{lemma}
Suppose \ts
$k/\sqrt{n} \to b>0$.  Let $\ts\A_b = \cup_n \A_{n, k}\ts$ and let $\ts\B_b = \cup_n \B_{n, k}$.
The bijection $\ts\varphiv: \A_b \to \B_b$ \ts is \stable\ with $r=1$, \ts $B=2$, scaling function \ts
$\a = \sqrt{n}$,
and $\ts K(t,y, \phi) = 2$ \ts for $\ts 0\leq t \leq y \leq 2b,$ and~0 otherwise.
\end{lemma}

\smallskip

\begin{theorem}[\cite{Romik}] \label{Ab:theorem}
Suppose $k/\sqrt{n}\to b>0$, and let $\A_b = \cup_n \A_{n, k}$.
The limit shape of $\A_b$ under scaling function $\a = \sqrt{n}$ is given by
\[
F(t) := \int_{t/2}^{b} \frac{e^{- c\, 2y}}{1-e^{- c\, 2y }} \, dy,
\]
where $c$ is the constant such that $\int_{0}^{2b} F(t)\, dt = 1$.
\end{theorem}

In fact, as written, the constant $c$ in Theorem~\ref{Ab:theorem} and the constant $c$ in Theorem~\ref{Bb:theorem} coincide.

% Generalization
\subsection{A generalization to \S\ref{further:example}}
\label{further:example:generalization}
For any $m \geq 2$ and $r \geq 1$, let $A_{n,k,m,r}$ denote the set of partitions of $n$ into parts from the set $U$ with $u_j = m^r\, j^r$, with largest part having size at most $m^r\, k.$
Let $B_{n,k,m,r}$ denote the set of partitions of $n$ into parts from the set $U$ with $u_j = m^r\, j^r$ and number of summands $\leq k$.
The bijection $\vartheta_n : A_{n,k,m,r} \to B_{n,k,m,r}$ is as follows: divide each part of $\lambda\in \B_{n,k,m,r}$ by $m^r$, multiply the resulting multiplicities by $m^r$, and take the conjugate.
Define $\A_{b,m,r} = \cup_n \A_{n, k,m,r}$ and similarly $\B_{b,m,r} = \cup_n \B_{n,k,m,r}$.
\begin{theorem}
Suppose $k/n^{r/(r+1)} \to b>0$.  Let $\A_{b,m,r} = \cup_n \A_{n, k,m,r}$ and similarly let $\B_{b,m,r} = \cup_n \B_{n,k,m,r}$.

\begin{enumerate}
\item {\rm(cf.~\S\ref{history})} The limit shape of $\A_{b,m,r}$ under scaling function $\a = n^{r/(r+1)}$ is given by
\[ F_{m,r}(t) = \int_{(t/m)^{1/r}}^{b} \frac{e^{- c\, m^r\, y^r}}{1-e^{- c\, m^r\, y^r }} \, dy, \qquad 0 < t \leq m^r\, b^r,
\]
where $c$ is a constant such that $\int_{0}^{m^r\,b^r} F_m(t)\, dt = 1$.

\item The bijection $\varphiv: \A_{n,k_n,m,r} \to \B_{n,k_n,m,r}$ is \stable\ with $B=m^r$, scaling function $\a = n^{r/(r+1)}$, and $K(t,y, \phi) = m^r$ for $0\leq t \leq y \leq m^r\, b^r,$ and 0 otherwise.

\item The limit shape of $\B_{b,m,r}$ under scaling function $\a = n^{1/(r+1)}$ is given by
\[G_{m,r}(t) = \int_t^{m^r\, b^r} m^r \frac{e^{- c\, m^r\, y^r}}{1-e^{- c\, m^r\, y^r }} \, dy,  \qquad 0 < t \leq m^r\, b^r,\]
where $c$ is the constant such that $\int_0^{m^r\, b^r} G_{m,r}(t)\, dt = 1$.
\end{enumerate}
\end{theorem}

% ***********************************************Probabilistic Setting***********************************************
\bigskip\section{Probabilistic setting}
\label{sect prob}

\subsection{Probabilistic definitions for unrestricted partitions}
In this section we introduce the probabilistic tools used in the proofs of the theorems.
Denote by $\Lambda_n$ a random partition of size $n$, where each partition is equally likely.
Similarly, denote by $C_i(n)$ the number of parts of size $i$ in the random partition $\Lambda_n$,
so that $\sum_{i=1}^n i\, C_i(n) = n$ is satisfied with probability~1.  The diagram function
(see \S\ref{diagram functions}) for a random partition of size~$n$ is given by
\begin{equation}
\label{random df}
D_{\Lambda_n}(t) \, = \, \sum_{i \geq t} C_i(n)\ts, \quad \text{for all } t>0.
\end{equation}

\label{convergence lemma}

Let $\{X_i(x)\}$ denote a sequence of \emph{independent}, but not identically distributed, geometrically distributed random variables with parameters $1-x^i$, $i\geq 1$, where $\p(X_i(x) \geq k) = x^{ik}$, $k\geq 0$, $0<x<1$.
We write $X_i \equiv X_i(x)$ and $C_i \equiv C_i(n)$ when the parameters are understood by context.

The diagram function $D_{\Lambda_n}(t)$ of Equation \eqref{random df} is constrained to have fixed total area for a given $n$, corresponding to a partition of fixed size $n$, for $n\geq 1$.  The random variables $X_i$ and $C_i$ are related by, for all $0<x<1$,
\begin{equation}
\label{equal in distribution}
\left((X_1, \ldots, X_n) \, \left\vert \, \sum_{i=1}^n i X_i = n\right)\right. \stackrel{D}{=} (C_1, \ldots, C_n), \qquad \text{for each $n \geq 1$,}
\end{equation}
see \cite{Fristedt}; see also \cite{IPARCS}.
Here, equality means in distribution.

  We next introduce a sequence of random variables $\{Z_i(x)\}_{i \geq 1}$, where $Z_i(x)$ is defined as the value of $X_i(x)$ conditional on $X_i(x) < a$, $i \geq 1$, $a\geq 1$, $0<x<1$; that is,
\[ Z_i(x) \stackrel{D}{=} (X_i(x) | X_i(x) < a), \qquad \mbox{ where $i \geq 1$.} \]
We also allow for the case when $a=\infty$, in which case we have $Z_i(x) \stackrel{D}{=} X_i(x)$.
Denote by $\Lambda$ a random partition of fixed size $n$ from an unrestrictedly smooth or restrictedly smooth set of partitions with parameters $r \geq 1$ and $B>0$.
Let $M_i(n)$ denote the number of parts of size $i$ in the random partition $\Lambda$, $i\geq 1$.  Then $M_i(n) = 0$ for $i\notin U$ and $M_i(n) < a$,  $i\in U$.
An analogous result to \eqref{equal in distribution} holds, namely, for all $0<x<1$,
\begin{equation}
\label{equal in distribution U}
\left(\bigl(Z_{u_1}, Z_{u_2}, \ldots \bigr) \left\vert \sum_{i\in U} i Z_i = n\right)\right. \stackrel{D}{=} \left(M_{u_1}, M_{u_2}, \ldots \right), \qquad \text{for each $n \geq 1$.}
\end{equation}

For a given scaling function $\a$ and constant $c$, define the \emph{tilting parameter}
\begin{equation}
\label{eq x}
x(n) := \exp\left(-c / \a\right), \qquad n \geq 1.
\end{equation}

    We define the \emph{independence diagram function} by
\[ S_x(t) = \sum_{k\in U, k\geq t} Z_k(x),\qquad  t> 0,~0<x<1.
\]
It has the same form as the diagram function $D_{\Lambda_n}$, except it is a sum of independent random variables.
  The tilting parameter $\xofn$ tilts the distribution, and we choose $c$ so that the expected area of $S_{\xofn}$ is $n$.  In this particular case, the tilt given in Equation \eqref{eq x} also maximizes the probability that the area of $S_{\xofn}$ is $n$ (see \cite[Section~5]{IPARCS}).

We also define the  \emph{scaled independence diagram function} as
\begin{equation}\label{shat}
\Shat_x(t) \. = \. \frac{\a}{n} \. S_x\left(\a t\right), \qquad t> 0,~0<x<1.
\end{equation}

% Concentration of the independence process
\subsection{Concentration of the independence process}\label{concentration}
In this section, we formalize the intuition that the weak notion of a limit shape, i.e., the pointwise convergence in expectation of the scaled independence diagram function, coincides with the limit shape of a certain set of integer partitions.

\begin{theorem}[\cite{Yakubovich}, cf.~\S\ref{history}]
\label{distribution is sufficient}
Let $\Shat_x$ denote the scaled independence diagram function corresponding to a set of partitions $\P_U$ which is either restrictedly smooth or unrestrictedly smooth.
Suppose there exists a piecewise continuous function $\Phi: \Rplus~\longrightarrow~\Rplus$, such that,
\[\mathbb{E}\left[\Shat_{x}(t)\right]\ \longrightarrow\ \Phi(t) \qquad \text{ for each continuity point \,\. $t>0$}.\]
Then the limit shape of $\A$ is given by $\Phi$.
\end{theorem}

The same kind of result holds in our setting as well.
Note that an \MMb or \MPb $\varphiv: \A\to\B$ induces a continuous, bijective map $\Tv:L_+^1(\R_+)\longrightarrow L_+^1(\R_+)$.  For diagram functions, we have
\[ \Tv D_\lambda(t) \. = \. \Tv \sum_{k\geq t} m_k(\lambda) \. = \, \sum_k v({ t ,k})\, m_k(\lambda) \,
= \, \begin{cases} \rho_{\varphi_{\bf v}(\lambda)}(t), & \text{\MPbns} \\
D_{\varphiv(\lambda)}(t), & \text{\MMbns}\end{cases}
\]

\begin{theorem}\label{distribution:sufficient:bijection}
Let $\Shat_x$ denote the scaled independence diagram function corresponding to a set of partitions $\P_U$ which is either restrictedly smooth or unrestrictedly smooth with parameters $r \geq 1$ and $B>0$.
Suppose there exists a piecewise continuous function $\Phi: \Rplus~\longrightarrow~\Rplus$, such that,
\[\qquad \qquad  \mathbb{E}\left[\Shat_{x}(t)\right]\ \longrightarrow\ \Phi(t), \quad \text{ for each continuity point \,\. $t>0$}.\]
Suppose further there exists a set of partitions $\B$ and an \MMb or \MPb $\varphiv:\A\longrightarrow\B$ that is \stable.
Suppose there exists a piecewise continuous function $\Phi_K: \Rplus \longrightarrow \Rplus$, such that,
\[\qquad \qquad \mathbb{E}\left[\Tv \Shat_{x}(t)\right]\ \longrightarrow\ \Phi_K(t), \quad \text{ for each continuity point \,\. $t>0$}.\]
Then the limit shape of $\B$ is given by $\Phi_K$.
\end{theorem}

The proof is straightforward, but requires several technical lemmas, so we present it in the section below.
The analogous theorem, valid for geometric transformations, is stated below.

\begin{theorem}[\cite{PakGeo}]\label{geo distr}
Let $\Shat_x$ denote the scaled independence diagram function corresponding to a set of partitions $\P_U$ which is either restrictedly smooth or unrestrictedly smooth.
Suppose there exists a piecewise continuous function $\Phi: \Rplus~\longrightarrow~\Rplus$, such that,
\[\qquad \qquad \mathbb{E}\left[\Shat_{x}(t)\right]\ \longrightarrow\ \Phi(t) \qquad \quad \quad  \text{ for each continuity point \,\. $t>0$}.\]
Suppose further  there exists a set of partitions $\B$ and a geometric bijection $\varphiv:\P_U\longrightarrow\B$.
Then the limit shape of $\B$ exists and can be found via the application of transformations in Theorem~\ref{transformations}.  In addition, each geometric transformation is asymptotically stable.
\end{theorem}
\begin{proof}
Each geometric transformation is a continuous injection in the space $L_+^1(\R_+)$, and thus the contraction principle applies to each transformation individually.
\end{proof}

% Proof of bijection theorem
\subsection{Proof of Theorem~\ref{distribution:sufficient:bijection}}

The proof has two steps.
The first establishes an exponential concentration, and the second is an example of ``overpowering the conditioning" (see~\cite[Section~10]{IPARCS}); that is, that the conditioning on the event $\{\sum_{i=1}^n i\, Z_i = n\}$ is not too strong.

Let $\Phi(t)$ denote the pointwise limit of the expected value of a scaled independence diagram function.  A well-known and key calculation is the following:
\begin{equation}\label{LD calculation}\begin{aligned}
 \Pr_n\left( \left| \Dhat_\lambda(t) - \Phi(t)\right| > \epsilon\right) &
\, = \, \Pr_n\left( \left| \Shat_\lambda(t) - \Phi(t)\right| > \epsilon\, \middle|\, \sum_{i\in U} i\, Z_i = n\right)  \\
 & \, \leq \, \frac{\Pr_n\left( \left| \Shat_\lambda(t) - \Phi(t)\right| >
\epsilon\right)}{\Pr\left(\sum_{i\in U} i\, Z_i = n\right)}\..
\end{aligned}\end{equation}
This formalizes the intuition that it is sufficient to have a large deviation which is little-o of the probability of hitting the target.
First, we note that it is sufficient to find a concentration inequality for the \emph{independence} diagram, as long as we have some bound on the probability of hitting the target of interest.
Fortunately, in the case of integer partitions, much is already known.

\begin{theorem}[\cite{LD, Yakubovich}] \label{large deviation}
Suppose $\Shat_x(t)$ is the independence diagram function for a set of partitions $\P_U$ that is either unrestrictedly smooth or restrictedly smooth with parameters $r\geq 1$ and $B>0$.  Assume $\e \Shat_x(t) \to \Phi(t)$ for all continuity points $t>0$ of $\Phi$.  Then for every $\epsilon>0$, there exists $\delta := \delta(\epsilon, t) \in\left(0,\frac{1}{2r+2}\right)$ such that for all $n$ large enough we have
\begin{equation}
\label{large deviations}
\p\left( \big|\Shat_x(t) - \Phi(t)\big| > \epsilon \right) \, \leq \,  \exp\left(- c_{t,\delta}\, n^{\frac{1}{2(1+r)}}\right),
\end{equation}
where $c_{t,\delta}\in (0,\infty)$ is a constant that does not depend on $n$.
\end{theorem}

\smallskip

\begin{theorem}[see, e.g., {\cite[Corollary~11]{Yakubovich}}; cf.~\S\ref{history}]\label{local bound}
Suppose $U$ is \smooth with parameters $r \geq 1$ and $B>0$, and $x = \exp\left(-d(r,B)/ n^{r/(1+r)}\right)$; or, $U$ is \rsmooth with parameters $r \geq 1$ and $B>0$ for some $a<\infty$, and $x = \exp\left(-d(r,B,a)/ n^{r/(1+r)}\right)$.
Then, there exists an $n_0 \geq 1$ such that
\begin{equation}\label{LCLT bound}
\p\left( \sum_{i\in U} i Z_i = n\right)  \, \geq\,  n^{-\gamma},  \qquad \mbox{ for all \ \. $n \geq n_0,$}
\end{equation}
where $\gamma > 0$ is a constant that does not depend on $n$.
\end{theorem}

\smallskip

\begin{lemma}
Under the assumptions of Theorem~\ref{distribution:sufficient:bijection},
for every $\epsilon>0$, there exists \[\delta~:=~\delta(\epsilon, t)~\in~\left(0,\frac{1}{2r+2}\right)\]
 such that for all $n$ large enough, we have
\begin{equation}
\label{large deviations 2}
\p\left( \big|\Tv \Shat_y(t) - \Phi_K(t)\big| > \epsilon \right) \leq \exp\left(- d_{t,\epsilon}\, n^{\frac{1}{2(1+r)}}\right),
\end{equation}
where $d_{t,\epsilon}\in (0,\infty)$ is a constant that does not depend on $n$.
\end{lemma}

\begin{proof}
Equation~\eqref{large deviations 2} follows by the contraction principle for continuous bijections, see~\cite[Th.~4.2.1]{LDBook}.
\end{proof}

% LD for dependent process
\subsection{Large deviation principle for the dependent process}\label{sect:LD}
In \cite[\S6.2]{LD}, a large deviation principle (LDP) is presented for a general class of measures over a process of \emph{independent coordinates}, which is stronger than Theorem~\ref{large deviation}.
In particular, the authors of~\cite{LD} let $\{c_k\}_{k=1}^\infty$ denote a sequence of non-negative integers, and say that $\{c_k\}$ is of type $(q,b)\in (0,\infty)\times (0,\infty)$ whenever
\[
\lim_{\epsilon\to 0} \. \lim_{L\to\infty} \. \epsilon^{-1} L^{-q} \. \sum_{k=L}^{(1+\epsilon)L} \. c_k \, = \, b\ts.
\]
In this more general context, $c_k$ denotes the number of different kinds of parts of size $k$; for example, we can give each part of size 1 any of $c_1$ colors, each part of size~2 can be any of $c_2$ colors, etc.
In this paper, we consider only the case $c_k \in \{0,1\}$ for all $k$.  In particular, for a sequence $u_k \sim B\, k^r$, this corresponds to
\[ c_k = \begin{cases} 1 & k\in U \\ 0 & \text{otherwise} \end{cases}\, , \qquad \text{$c_k$ is of type $\left( \frac{1}{r}, \frac{1}{r\, B^{1/r}}\right)$}. \]
Then, as in \cite{LD}, we let $m_{q,b}$ denote the positive, $\sigma-$finite measure on $[0,\infty)$ which has density $dm_{q,b}/dt = b\, t^{q-1}$.  Then the corresponding limit shapes given in \cite{LD}, namely, $\Psi_{q,b}(t)$ and $\Psi^s_{q,b}(t)$, correspond to our $\Phi(t; r,B)$ and $\Psi(t;r,B)$.

In our setting, the LDP is an asymptotic expression, both upper and lower bounds, governing the exponential rate at which a sequence of measures, say $\Pr_n$, assigns probability to a given measurable event as $n \to \infty$.
For example, a typical event is of the form $\mathcal{X}_\epsilon := \{|X-\mu| \geq \epsilon\}$, and we say that measure $\Pr_n$ satisfies the LDP with rate $\a$ and good rate function $I(\mathcal{X}_\epsilon)$ if it has the form
\[ \lim_{n\to\infty} \frac{1}{\a} \log \Pr_n(|X_n-\mu| \geq \epsilon) = -I(\mathcal{X}_\epsilon).\]
This is equivalent to
\[ \Pr_n(|X_n-\mu| \geq \epsilon) \sim \exp\left( -\a I(\mathcal{X}_\epsilon)\right), \]
and is stronger than the inequality in Theorem~\ref{large deviation}.
There are further details which pertain to the topology of convergence, a weakening of the limit as a $\liminf$ and $\limsup$,  and the precise form of the rate function $I$.
These details are relevant in order to distinguish between limit shapes of Andrews partitions with unrestricted multiplicities of allowable part sizes, where the expected number of parts grows like $O(\a \log(n))$, and Andrews partitions with bounded multiplicities of allowable part sizes, where the expected number of parts grows like $O(\a)$.
In the former case, we only obtain pointwise convergence.  In the latter case, the LDP is obtained in the topology of uniform convergence.
We are only able to prove a LDP when the allowable part sizes have bounded multiplicity, see Theorem~\ref{our LDP} below.

Let $J\subset[0,\infty)$ denote an interval, and let $D(J)$ be the space of all functions $f : J\to \R$ that are left-continuous and have right limits.
Let $\mathcal{AC}_{\infty}$ denote the subset of $D([0,\infty))$ of non-increasing absolutely continuous functions $f$ which satisfy $\lim_{t\to\infty} f(t) = 0$, and let $\mathcal{AC}_{\infty}^{[-1,0]} \subset \mathcal{AC}_{\infty}$ denote the set of functions with derivatives which are Lebesgue-a.e. in the interval $[-1,0]$.
We write the Lebesgue decomposition of $f$ as $f(t) = f_{ac}(t) + f_s(t)$, where $f_{ac}$ denotes the absolutely continuous part of $f$, and $f_s$ denotes the singular part.
Let $g'$ denote the derivative of a function $g$ with respect to $t$.

We write $d$ instead of $d(r,B,a)$ just for equation~\eqref{two param restricted LD} below.
Define
\begin{equation}\label{two param restricted LD} \widehat{I}_{r,B,a}(\nu, f) := d\left(1 - \int_0^\infty t(-f'(t))dt\right) + \int_0^\infty \widehat{H}_a\left(-\frac{df}{dm_{q,b}}\,,\, -\frac{d\Phi(r,B,a)}{dm_{q,b}}\right)dm_{q,b}\, ,\end{equation}
for all \[ f\in \mathcal{AC}_\infty  \quad \text{and} \quad  \int_0^\infty (-t)\, df(t) \leq \nu, \]
and $\widehat{I}_{r,B,a}(\nu, f) = \infty$, otherwise.
Here $\widehat{H}_a(f,g) = f \log(f/g) - f + g$ denotes the relative entropy between functions $f$~and~$g$.

\begin{theorem}\label{our LDP}
Suppose the set of partitions $\P_U$ is \rsmooth with parameters $r\geq 1$ and $B>0$ and $a<\infty$.  Then the diagram functions $\Dhat_\lambda$ satisfy the large deviation principle in $D[0,\infty)$ (equipped with the topology of uniform convergence), with speed $n^{1/(1+r)}$ and good rate function $\widehat{I}_{r,B,a}(1,f)$, given by equation~\eqref{two param restricted LD}.
\end{theorem}

For partitions without any restrictions on allowable part sizes, there is an area transformation that allows one to conclude that the LDP for the independent process also applies to the dependent process.  When restrictions like $u_k \sim B\, k^r$ are imposed, this area transformation is no longer defined.  We now supply an appropriate area transformation and prove the above claim, \emph{which only covers the case when the multiplicity of each part size is bounded}.  Our treatment closely mimics \cite[\S5]{LD}.

\begin{proof}[Proof of Theorem~\ref{our LDP}]
Let $N(\lambda)$ denote the size of the partition $\lambda$.
We define the transformation $F_{r,B,a,n}(\lambda)$ by
\begin{enumerate}
\item If $N(\lambda) = n$, then $F_{r,B,a,n}(\lambda) = \lambda$.
\item If $0\leq N(\lambda)=k < n$, then
	\begin{enumerate}
	\item if there exists a set of partitions of size $k$ consisting of parts from $U \setminus \lambda$, with each part size having multiplicity strictly less than $a$, for $a=2,3,\ldots,$ then $F_{r,B,a,n}(\lambda)$ adds the parts of the partition which has the least number of parts to $\lambda$; in the case of a tie, we choose the partition which is lower in lexicographic order;
	\item If there does not exist a set of partitions of size~$k$ consisting of parts from $U \setminus \lambda$, then $F_{r,B,a,n}(\lambda)$ is the partition of $n$ in $\P_U$ with the least number of parts.
	\end{enumerate}
\end{enumerate}

We now note the following properties of the transformation $\lambda \longmapsto  F_{r,B,a,n}(\lambda)$:
\begin{enumerate}
	\item For some fixed $c_3$ that does not depend on $n$, and for $n$ large enough, we have
	\[0 \leq \varphi_{F_{r,B,a,n}(\lambda)}(i) - \varphi_\lambda(i) \leq c_3 ( n-N(\lambda))^{1/(1+r)}\]
	for all $\lambda\in \P_U$.
	\item For any $\lambda$ with $N(\lambda) = n$, we have
		\[ \left|F_{r,B,a,n}^{-1}(\lambda)\right| \leq \sum_{k=1}^n p_U(n-k); \]
	\item Let $\eta \in (0,1)$, which we will choose later.   For any $\lambda_0 \in \P_U$ with $N(\lambda_0) = n$, and for $n$ sufficiently large, we have
	\begin{align}
	\nonumber \Pr\big( \lambda \in &\, F_{r,B,a,n}^{-1}(\lambda_0),  \ 1-\delta < n^{-1}N(\lambda) < 1-\delta / n^{1-\eta}\big) \\
	\nonumber	& \leq\ \sum_{k>(1-\delta)n}^{k<(n-\delta\, n^{\eta})} p_U(n-k)\, x^{k-n}\,\Pr(\lambda = \lambda_0) \\
	\nonumber	& \leq\ n\, p_U(\delta n)\ts \exp\left(c\ts \delta\ts n^{1/(r+1)}\right)\ts \Pr(\lambda = \lambda_0) \\
	\label{ineq}	& \leq\ n^{\gamma_1} \exp\left( (\delta n)^{1/(r+1)}\right) \exp\left(c\, \delta\, n^{1/(r+1)}\right)\, \Pr(\lambda = \lambda_0),
	\end{align}
\noindent where $\gamma_1$ is some real--valued constant which does not depend on $n$.  In addition,
\begin{align}
	\nonumber \Pr\big( \lambda \in &\,  F_{r,B,a,n}^{-1}(\lambda_0),  \ 1-\delta/n^{1-\eta} < n^{-1}N(\lambda) < 1\big) \\
	\nonumber	& \leq\ \sum_{k>(n-\delta\, n^{\eta})}^{n}\, p_U(n-k) x^{k-n}\, \Pr(\lambda = \lambda_0) \\
	\nonumber	& \leq\ n\, p_U(\delta n^\eta)\ts \exp\left(c\, \delta\, n^{\eta-r/(r+1)}\right)\ts \Pr(\lambda = \lambda_0) \\
	\nonumber	& \leq\ n^{\gamma_2} \exp\left( \left(\delta n^\eta\right)^{1/(r+1)}\right)) \exp\left(c\, \delta\, n^{\eta-r/(r+1)}\right) \Pr(\lambda = \lambda_0),
	\end{align}	
	\ignore{\begin{align}
	\nonumber \Pr( \lambda \in F_{r,B,a,n}^{-1}(\lambda_0), & \ 1-\delta/n^{1-\eta} < n^{-1}N(\lambda) < 1)
		 \leq \sum_{k>(n-\delta\, n^{\eta})}^{k<n} p_U(n-k) x^{k-n}\Pr(\lambda = \lambda_0) \\
	\nonumber	& \leq n\, p_U(\delta n^{\eta}) \exp(c\, \delta\, n^{\eta-r/(r+1)}) \Pr(\lambda = \lambda_0) \\
	\nonumber	& \leq n^\gamma \exp( (\delta n^{\eta})^{r/(r+1)}) \exp(c\, \delta\, n^{1/(r+1)}) \Pr(\lambda = \lambda_0),
	\end{align}}
where $\gamma_2$ is some real--valued constant which does not depend on $n$.
\end{enumerate}

\nin
Next, we define the sets
\[\begin{array}{ll}
B_{\phi,\delta} & = \bigl\{f\in D[0,\infty) : \|f-\phi\|_\infty < \delta \bigr\}, \\ \\
\widehat{B}_{\phi, \delta,\eta}^1 & = \bigl\{(\nu, f) : \| f - \phi\|<\delta, (1-\delta) < \nu < (1-\delta/n^{1-\eta})\bigr\}, \\ \\
\widehat{B}_{\phi, \delta,\eta}^2 & = \bigl\{(\nu, f) : \| f - \phi\|<\delta, (1-\delta/n^{1-\eta}) < \nu < 1\bigr\},\\ \\
\widehat{B}_{\phi, \delta} & = \widehat{B}_{\phi, \delta,\eta}^1 \cup \widehat{B}_{\phi, \delta,\eta}^2 = \bigl\{(\nu, f) : \| f - \phi\|<\delta, (1-\delta) < \nu < 1\bigr\}.
\end{array}
\]
Define $\delta_r := 2\, c_3\, \delta^{1/(1+r)}$.  Then $n$ can be taken sufficiently large, so that we have
\[ \Pr\left( \left( N(\lambda)/n, \Dhat_\lambda\right) \in \widehat{B}_{\phi, \delta} \right)  \leq  \Pr\left( F_n(\lambda) \in B_{\phi, \delta_r}\right).  \]

Thus, for such sufficiently large $n$, we have
\begin{align*}
 \Pr\bigl( (n^{-1}N(\lambda), \Dhat_n) & \in \widehat{B}_{\phi, \delta,\eta}^1\bigr) \, = \,
 \sum_{\lambda_0 \in B_{\phi,\delta_r}, N(\lambda_0) = n} \Pr\left( F_{r,B,a,n}(\lambda) \ts = \ts
 \lambda_0, \bigl(n^{-1}N(\lambda),\Dhat_\lambda\bigr) \in \widehat{B}_{\phi,\delta,\eta}^1\right) \\
   & \hspace{-1cm} \leq \, \sum_{\lambda_0 \in B_{\phi,\delta_r}, N(\lambda_0) = n} \Pr\left( \lambda \in F_{r,B,a,n}^{-1}(\lambda_0), n^{-1}N(\lambda) \in (1-\delta, 1-\delta/n^\eta)\right) \\
   & \hspace{-1cm} \leq \, n^{\gamma_1} \. \exp\left[ (\delta n)^{1/(r+1)}\right] \.
   \exp\bigl(c\, \delta\, n^{1/(r+1)}\bigr) \. \Pr(n^{-1}N(\lambda)=1, \Dhat_\lambda \in B_{\phi,\delta_r}),
 \end{align*}
and
\begin{align*}
 \Pr\bigl( (n^{-1}N(\lambda), \Dhat_n) & \in \widehat{B}_{\phi, \delta,\eta}^2\bigr) \, = \,\sum_{\lambda_0 \in B_{\phi,\delta_r}, N(\lambda_0) = n} \.
 \Pr\left( F_{r,B,a,n}(\lambda) \, = \,\lambda_0, \bigl(n^{-1}N(\lambda),\Dhat_\lambda\bigr) \in \widehat{B}_{\phi,\delta,\eta}^2\right) \\
   & \hspace{-.75cm} \leq \sum_{\lambda_0 \in B_{\phi,\delta_r}, N(\lambda_0) = n} \Pr\left( \lambda \in F_{r,B,a,n}^{-1}(\lambda_0), n^{-1}N(\lambda) \in (1-\delta/n^\eta,1)\right) \\
   & \hspace{-.75cm} \leq \, n^{\gamma_2} \. \exp\left[ (\delta n^{\eta})^{1/(r+1)}\right] \. \exp(c\, \delta\, n^{\eta - r/(r+1)}) \.
   \Pr(n^{-1}N(\lambda)=1, \Dhat_\lambda \in B_{\phi,\delta_r}).
 \end{align*}
\ignore{\begin{align*}
 \Pr\left( \big( \frac{N(\lambda)}{n}, \Dhat_n\right) & \in \widehat{B}_{\phi, \delta,\eta}^2\big)  \leq \Pr\left( N(\lambda) \in (n-\delta n^\eta, n]\ \right) \\
   & \leq \Pr\left( \frac{N(\lambda)-n}{n^{1/(2(1+r))}} \in (-\delta n^{\eta-1/(2(1+r))},0) \right), \\
   & \lesssim \Pr\left( \mathcal{N}(0,1) \in (-\delta n^{\eta-1/(2(1+r))},0)\right) \\
   & \lesssim \frac{1}{\sqrt{2\pi}} \delta\, n^{\eta-1/(2(1+r))}. \\
 \end{align*}}

\nin
Thus, for any $\eta \in (0,1)$, we have
\[ \liminf_{n\to\infty} \frac{1}{n^{1/(1+r)}} \. \log \Pr\left(N(\lambda)/n = 1, \Dhat_\lambda \in B_{\phi, \delta_r}\right) \,
\geq  \, -\inf_{(\nu, \psi)\in \widehat{B}_{\phi,\delta}} \widehat{I}_{r,B,a}(\nu, \psi)  - c\, \delta -  \delta^{1/(r+1)},
\]
where $ \widehat{I}_{r,B,a}(\nu, f)$ is given by equation~\eqref{two param restricted LD}.

The rest of the proof is straightforward; the only ingredient previously missing was the area transformation $F_{r,B,a,n}$ and the inequality given by equation~\eqref{ineq}, which holds under the assumption that $U$ is \rsmooth with parameters $r \geq1$, $B>0$, and any \ts $2 \leq a<\infty$.
\end{proof}
Finally, we note that our transformation $F_{r,B,a,n}$ does not apply for integer partitions into unrestricted part sizes, as our proof relies on the total number of parts in a typical partition to be $O(n^{1/(1+r)})$ as $n \to \infty$.

% ***********************************************Main Theorems***********************************************
\bigskip\section{Proofs of transfer theorems}
\label{sect main}

\subsection{Notation}
For $r \geq 1, B>0$ and $a \geq 2$, let $Z_i,$ $i\ge 1$, $c = d(r,B,a)$, and $x=\xofn$ be defined as in \S \ref{convergence lemma}.
Define $E_a(x)~:=~\e[Z_i(x)]$, which has explicit form given by
\begin{equation*}
\large E_a(x) = \frac{x+2x^2+3x^3+\ldots (a-1)\,x^{a-1}}{1+x+x^2+x^3+\ldots+x^{a-1}} \ \ \  a \geq 2.
\end{equation*}
Note that $\phi(y; r, B, a) = E_a\left(e^{-c\, B\, y^r}\right)$.
Similarly, when multiplicities are unrestricted, we define $E(x) := \frac{x}{1-x}.$

We are now ready to state and prove our general theorem regarding limit shapes.
We assume for the rest of the section that the set $U$ is such that $\P_U$ is \smooth or \rsmooth with parameters $r\geq 1$, $B>0$, $a \geq 2$.

\subsection{Proof of theorems~\ref{special unrestricted} and~\ref{special distinct}}
\label{proof special distinct}
Let $y_k$ be such that as $n\to \infty$, we have $B\, y_k^r = u_k / \a$, $\Delta y_k \sim (1/B\a)^{1/r}$.  Then as $n\to\infty$, we have
\begin{align*}
\mathbb{E}\left[\Shat_x(t)\right] & = \frac{n^{r/(1+r)}}{n} \sum_{k\in U: k\geq t \a} \e [Z_k]  \\
& \sim \,
\begin{cases}
\sum_{y_k \geq (t/B)^{1/r}} E\left(e^{-c\, B\, y_k^r}\right) \Delta y_k \to \Phi(t; r, B), & \text{for Theorem~\ref{special unrestricted}}, \\
\sum_{y_k \geq (t/B)^{1/r}} E_a\left(e^{-c\, B\, y_k^r}\right) \Delta y_k \to \Phi(t; r, B,a), & \text{for Theorem~\ref{special distinct}}.
\end{cases}
\end{align*}
By Theorem~\ref{distribution is sufficient}, this calculation is sufficient to establish the limit shape.

% Proof of theorems
\subsection{Proof of theorems~\ref{MT unrestricted} and~\ref{thm dist}}
\label{proof special distinct transform}
We continue to use the notation from \S \ref{proof special distinct}.
In the case of an \MPbns, we have
\[T \widehat{S}_x(t)  = \frac{\b}{n} \sum_{k \geq 1} v({ \b\ts t, u_k})\, Z_{u_k}.\]
Taking expectation, and applying equation~\eqref{v nice} and equation~\eqref{v nicea}, as $n\to \infty$ we have
\begin{align*}\label{B inv}
\mathbb{E}\left[T\Shat_x(t)\right] & = \frac{\b}{n} \sum_{y_k \geq 0} v({\b t,\a B\, y_k^r})\, \mathbb{E}\left[Z_{B\th y_k^r \a}\right] \\
& \longrightarrow   \begin{cases}
	\int_0^\infty K\left(t,y, E\left(e^{-c\, B\, y^r}\right) \right)\, dy, & \text{for Theorem~\ref{MT unrestricted},} \\
 	\int_0^\infty K\left(t,y, E_a\left(e^{-c\, B\, y^r}\right) \right)\, dy, & \text{for Theorem~\ref{thm dist}.}
	\end{cases}
\end{align*}
In the case of an \MMbns, similarly by equation~\eqref{v nice2} and equation~\eqref{v nice2a}, as $n\to \infty$ we have
\begin{align*}
\mathbb{E}\left[T \widehat{S}_x(t)\right]  & = \frac{\b}{n} \sum_{i \geq t \b} \sum_{k \geq 1} v({ \b t, u_k})\, \e Z_{u_k} \\
 & \longrightarrow \begin{cases}
 	 \int_0^\infty K\left(t,y, E\left(e^{-c\, B\, y^r}\right)\right)\, dy, & \text{for Theorem~\ref{MT unrestricted},}  \\
 	 \int_0^\infty K\left(t,y, E_a\left(e^{-c\, B\, y^r}\right)\right)\, dy, & \text{for Theorem~\ref{thm dist}.}
	\end{cases}	
\end{align*}
By Theorem~\ref{distribution:sufficient:bijection}, the proof is complete.

% ***********************************************Remarks***********************************************
%\bigskip\section{Historical remarks}
\bigskip \section{Equivalence of ensembles}
\label{sect remarks}

\subsection{ } %Historical remark
\label{history}
\ignore{Intuitively, a limit shape is a curve representing a law of large numbers of a joint distribution of component sizes in a random combinatorial structure.
Our definition of a limit shape is made precise in Section~\ref{diagram functions}, and is a limit \emph{in probability} as $n\to \infty$, of the part sizes of a uniformly random integer partition of size~$n$.
It turns out that the weaker notion of a limit shape, a limit \emph{in expectation} of independent random part sizes which generate a random integer partition of \emph{random size}, i.e., approximately of size~$n$, agrees with this stronger definition in the context of integer partitions.
Thus, there is rigorous reasoning behind the intuition, which is made precise in Section~\ref{sect prob}.

We start with two pivotal results in the field of limit shapes of integer partitions: the classical limit shape, and the limit shape of partitions into distinct parts. }
The probabilistic model utilized for integer partitions goes back to Khinchin~\cite{Khinchin}.
The explicit form of the limit shape curve $\Phi(t)$ for unrestricted integer partitions was derived via heuristic arguments by Temperley~\cite{Temperley}.
He applied the method of steepest descents to the generating function for $p(n)$, the number of partitions of size~$n$, in order to describe the equilibrium state of crystalline structures.
Several decades later, Kerov and Vershik~\cite{Kerov} found and proved the limit shape for integer partitions under the \emph{Plancherel measure}, which weights each partition of size~$n$ by the squared dimension $(f^\la)^2$ of the corresponding irreducible representation of~$S_n$.  Its formula (rotated by $45^\circ$)
is given by
\[ \Omega(t) = \begin{cases}
\frac{2}{\pi} \left(t\, \arcsin \frac{t}{2} + \sqrt{4-t^2}\right), & |t| \leq 2 \\
|t| & |t| \geq 2.
\end{cases}
\]
We refer to~\cite{R4} for a thorough and well written proof of this result and its applications.

Let us mention that the final section of~\cite{Kerov} states without proof that the approach used also gives the limit shape of integer partitions under the uniform measure. A followup paper by Vershik~\cite{Vershik} states the strong notion of a limit shape for uniformly random integer partitions, but leaves out the essential details which imply the strong version from the weak one, sometimes referred to as ``equivalence of ensembles" in the statistical mechanics literature.

In parallel to the probabilistic notions of a limit shape, there was an early interest in understanding the joint distribution of part sizes, which appears to begin with the work of Erd\H{o}s and Lehmer on the largest part size~\cite{ErdosLehner}, later extended by Erd\H{o}s and Szalay~\cite{ErdosSzalay} to the largest $t$ part sizes, where $t>0$ is fixed.
A significant improvement occurred with the work of Fristedt~\cite{Fristedt}, who extended many but not all of the results to the largest $o(n^{1/4})$ part sizes.

Pittel fully resolved the strong notion of a limit shape in~\cite{PittelShape}.
The proof starts by confirming a conjecture of Arratia and Tavare~\cite{IPARCS} governing the total variation distance between the joint distribution of component sizes and the joint distribution of appropriately chosen independent random variables, precisely those defined in Section~\ref{sect prob}, and essentially extending Fristedt's results to the largest $o(n^{1/2})$ part sizes\footnote{Note that Fristedt used the \emph{Prokhorov} metric whereas Pittel used the total variation distance metric.}.
The \emph{total variation distance} between two distributions, say $\mathcal{L}(X)$ and $\mathcal{L}(Y)$, is defined as
\[
d_{TV}\bigl(\mathcal{L}(X), \mathcal{L}(Y)\bigr) \, = \, \sup_{A \subset \R} \, \bigl| \Pr(X \in A) - \Pr(Y \in A)\big|\ts,
\]
where $A$ is a Borel subset of~$\R$.
There are two initial, complementary results, governing the smallest part sizes and the largest part sizes.
\begin{theorem}{\rm(\cite[Th.~1]{PittelShape})}
 In the definitions from Section~\ref{sect prob}, let $x = e^{-c/\sqrt{n}}$, with $c = \pi/\sqrt{6}$,
 and suppose $Z_i \equiv Z_i(x)$ is a geometric random variable with parameter $1-x^i$, for $i\ge 1$,
 where all $Z_i$ are mutually independent.
Let $C_i(n)$ denote the number of parts of size~$i$ in a random integer partition of size~$n$.
Then:
\begin{align*}
& d_{TV}\bigl( \mathcal{L}(C_1(n), C_2(n), \ldots, C_{k_n}(n)), \mathcal{L}(Z_1, Z_2, \ldots, Z_{k_n})\bigr)  \to 0, \qquad \frac{k_n}{\sqrt{n}} \to 0, \\
& d_{TV}\bigl( \mathcal{L}(C_{k_n}(n), C_{k_n+1}(n), \ldots, C_{n}(n)), \mathcal{L}(Z_{k_n} Z_{k_n+1}, \ldots, Z_{n})\bigr) \to 0, \qquad \frac{k_n}{\sqrt{n}} \to \infty,\\  & \hskip6.cm \mbox{ and } \ \ k_n \. < \. \frac{\sqrt{3/2}}{\pi} \. \sqrt{n} \.\log(n)\..
\end{align*}
\end{theorem}
Thus, one may treat both small and large part sizes as approximately independent, which was utilized in~\cite{Pittel}, for example, to resolve several conjectures for integer partitions involving statistics localized on the largest part sizes.
Pittel also showed in~\cite[Th.~1]{PittelShape}, that if we include in the joint distribution of part sizes those parts which are at most $t \sqrt{n}$, for some $t>0$, then this joint distribution converges to a non-trivial limit which is the total variation distance between two appropriately defined normal distributions.
However, total variation distance is a considerably strong metric, since one can take any measurable function of the random variables and the resulting total variation distance is guaranteed not to increase, whereas a limit shape is one particular statistic.
Thus, by a more detailed analysis of the component sizes, specifically for the limit shape statistic, Pittel was able to show that the limit shape of random integer partitions is indeed close in probability to the limit shape $\Phi$ indicated by the weaker notions of convergence, thus formally establishing the equivalence of ensembles.

The reason why the weaker notion of the limit shape, i.e., convergence in expectation of random-size partitions with independent part sizes, coincides with the strong notion of the limit shape, i.e., convergence in probability of integer partitions of a fixed size, is due to the exponential concentration of the limit shape statistic.
This has traditionally been handled either via a martingale technique, or via more direct methods like Markov's inequality, and demonstrates an unsurprising exponential concentration around the expected value which could be considered implicit in the work of Vershik~\cite{Vershik}.
As noted in \S \ref{sect:LD}, an LDP is not just an upper bound on the rate of exponential concentration, as is often all that is needed to prove the strong notion of a limit shape from the weak notion, it is also an asymptotic expression for the rate of concentration.

While Pittel's work proved the limit shape of unrestricted partitions, the large deviation results by Dembo, Vershik, and Zeituni in~\cite{LD} provided the limit shape of integer partitions into distinct parts, and made their connection with the LDP explicit.
In addition, \cite[Section~6.2]{LD} states the form of the limit shape for partitions with restrictions of the form $u_k = B k^r,$ and provides the LDP for the \emph{independent} process in both cases of unrestricted multiplicities and distinct part sizes.
Combining the LDP with classical asymptotic enumeration formulas, e.g., those in for example~\cite{HR, Wright, Ingham}, it is straightforward to derive the equivalence of ensembles under these restrictions via the inequality in Equation~\eqref{LD calculation}.
We have provided the LDP for the \emph{dependent} process with bounded multiplicities, i.e., the process governing random partitions of fixed size~$n$, in \S \ref{sect:LD}.

\subsection{ }
\label{credit}
It is difficult to credit limit shape results like Theorem~\ref{special unrestricted} precisely, owing to the connection between asymptotic enumeration and concentration, and the precise conditions which are assumed and utilized.
The weak form of the limit shape in Theorem~\ref{special unrestricted} is easily calculated via the statistical models in Khinchin~\cite{Khinchin}, and more explicitly by Kerov and Vershik~\cite{Kerov}; the same ones utilized by Fristedt~\cite{Fristedt} and Pittel~\cite{PittelShape}.
The limit shape then follows by concentration and a local central limit theorem, or equivalently, asymptotic enumeration, via the inequality in Equation~\eqref{LD calculation}; see also~\cite{Goh, rth}.

Ingham~\cite[Th.~2]{Ingham} provided quite general asymptotic enumeration formulas for partitions satisfying the following conditions:
Let $N(u)$ denote the number of elements in $U$ less than or equal to $u$.  Suppose there are constants $L>0$, $\theta >0$, and function $R(u)$ such that
\[ N(u) = L u^\theta + R(u), \]
and for some $b>0$ and $c>0$ we also have
\[ \int_0^u \. \frac{R(v)}{v}\, dv \. = \. b \log u + c + o(1). \]
Define
\[ \alpha \. = \. \frac{\beta}{1+\beta}\qquad M \. = \. \bigl[L \ts\beta\ts\Gamma(\beta+1)\ts\zeta(\beta+1)\bigr]^{1/\beta}. \]
Then, \emph{as long as $p_U(n)$ is monotonically increasing}, we have
\begin{equation}\label{eq:Ingham}
p_U(n) \. \sim\.  \left(\frac{1-\alpha}{2\pi}\right)^{1/2}\ts e^cM^{-(b-\frac{1}{2})\alpha}
\ts u^{(b-\frac{1}{2})(1-\alpha)-\frac{1}{2}}\ts e^{\alpha^{-1}(Mu)^\alpha}.
\end{equation}
Similarly, under the slightly weaker assumption that
\[
\int_0^u R(v)\,dv \. = \. b\, u + o(u),
\]
and again assuming \emph{$\ts p_U^d(n)$ \ts is monotonically increasing}, with \ts
$M^\ast := (1-2^{-\beta})^{1/\beta}$, we have
\[
p_U^d(n) \. \sim \. \left(\frac{1-\alpha}{2\pi}\right)^{1/2} \ts 2^b \ts (M^\ast)^{\alpha/2}
\ts u^{\frac{1}{2}\alpha - 1}\ts e^{\alpha^{-1}(M^\ast u)^\alpha}\ts.
\]

Erd\H{o}s and Bateman in~\cite[Th.~6]{ErdosBateman} simplified the
monotonically increasing assumption for $p_U(n)$ to be more simply
that the greatest common divisor of $U$ is~1.
It is unknown whether or not a corresponding condition exists for
partitions into distinct parts, and it would be interesting
to find such necessary and sufficient conditions.

Roth and Szekeres provide a partial answer, at least one detailed enough for our purposes.
They demand the elements $u_k$ satisfy more generally
\begin{equation}
\tag{U1$'$}\label{U1:prime} \lim_{k\to\infty} \frac{\log u_k}{\log k} = r \in [1,\infty).
\end{equation}
In addition, they also assume the following condition, which is sufficient for the asymptotic enumeration results to hold:
\begin{align}
\tag{U2}\label{U2} & \text{$J_k = \inf_\alpha \left\{ \frac{1}{\log k} \sum_{i=1}^k [u_i \alpha]^2 \right\} \to \infty \ \ {\rm as} \ \ k\to \infty,$}
\end{align}
where $[x]$ denotes the closest integer to $x$ and the infinum is taken over $\alpha \in \left(\frac{1}{2u_k}, \frac{1}{2}\right]$.
They also show that when $r$ is small, Condition~$\eqref{U2}$ is not necessary.
\begin{prop}[Roth and Szekeres~\cite{RothSzekeres}]
\[\text{\eqref{U1:prime}} \text{\rm\ and $r<\frac{3}2$ } \implies \text{\eqref{U2}}.\]
\end{prop}

There are also several classes of sequences $U$ presented in the introduction in~\cite{RothSzekeres} which cover many cases of interest; namely,
\begin{quote}
\begin{itemize}
\item[(i)] $u_k = p_k$ where $p_k$ is the $k^{th}$ prime number.
\item[(ii)] $u_k = f(k)$, where $f(x)$ is a polynomial which takes only integral values for integral $x$ and has the property that corresponding to every prime $p$ there exists an integer $x$ such that $p \nmid f(x)$.
\item[(iii)] $u_k = f(p_k)$, where $f(x)$ is a polynomial which takes only integral values for integral $x$ and has the property that corresponding to every prime $p$ there exists an $x$ such that $p \nmid x f(x)$.
\end{itemize}
\end{quote}
In particular, they specialize their main theorem in the case when $u_k = a_r k^r + a_{r-1} k^{r-1} + \ldots + a_0$ and $u_k$ satisfies (ii), which prevents any gaps in $p_U^d(n)$; this is Theorem~\ref{RothSzekeres}.
We have dropped the final condition that for every prime $p$ there exists an integer $x$ such that $p \nmid f(x)$, requiring instead that should such a $p$ exist then we take $n$ to infinity through multiples of the gcd of $U$, as is required for the corresponding limit shape.

One could also argue that Theorem~\ref{special unrestricted} follows from Pittel's analysis~\cite{PittelShape}, in that it requires no new ideas other than generalizing the analysis appropriately.
However, there are certain estimates utilized which must be carefully adapted, and a priori require some kind of technical conditions similar to the ones initially imposed by Ingham and others.

One may also be tempted to attribute Theorem~\ref{special unrestricted} to Dembo, Vershik, and Zeituni~\cite{LD}.
In~\cite[Section 6.2]{LD}, they state explicitly an LDP for an independent process of random variables, which is sufficient to prove concentration, but as stated the sequence $U$ consists of multiples of perfect powers.
This is similar to the work in Goh and Hitczenko~\cite{Goh}, whose extra technical assumption is the existence of four $u_j$'s which are relatively prime.
Canfield, Corteel, and Wilf~\cite{rth} provide all of the ingredients for Theorem~\ref{special unrestricted}, even including the restriction that $u_k$ is the image of a polynomial, which is the special case we have adopted.

\subsection{}
We have decided to describe the growth conditions on sets $U$ in terms of the growth of~$u_k$, where $k\ge 1$.
Meinardus~\cite{Meinardus}, Granovski, Stark and Erlihson~\cite{Stark2008}, Hwang~\cite{Hwang},
and Bogachev~\cite{Bog}, for example, impose conditions directly on the corresponding generating functions.
Yakubovich~\cite{Yakubovich}, Goh and Hitczenko~\cite{Goh}, and Ingham~\cite{Ingham},
on the other hand, have described growth conditions on the inverse function \ts
$N(k) = \#\{j : u_j < k\}$.  Our choice is natural for this setting since our sets of
partitions and bijections have traditionally been written in the form of allowable part sizes.

\subsection{ }\label{Ingham Form}
The original form of the growth conditions in Ingham's \cite[Th.~2]{Ingham}
is particularly elegant, even though it is not ideal for our setting.
Rather than requiring a condition on the greatest common divisor of allowable part sizes,
the result is instead stated in terms of sums of partition numbers.  Formally,
let $p_U(n)$ denote the number of partitions of $n$ into parts from the set~$U$.
Ingham defines for real $u>0$
\[ 
P(u) \. = \. \sum_{k<u} p_U(k)\ts.
\]
Then, he proves an asymptotic result for $P(u)$ \emph{and}
\[ 
P_h(n) \. = \. \frac{P(n) - P(n-h)}{h}\., \qquad h>0\ts, 
\]
under the condition that $P_h(n)$ is increasing for every fixed $h>0$.
When \ts {\tt gcd}$(U) = a$, we can choose $h = a$.
Then, \emph{for every} $u>0$, the number $P_h(n)$ is precisely one of $p_U(\ell)$ for some
$\ell \in [u-h,u)$, and, since $(u-h) \sim u \sim \ell$, it thus captures
the asymptotic rate of growth of $p_U(n)$, without needing to specify,
as we have done, that $n\to \infty$ through multiples of~$a$.

\subsection{}
The conditions of Meinardus~\cite{Meinardus} and Yakubovich~\cite{Yakubovich}
are related to Condition~\eqref{U2}, and apply more generally to partitions
where each part size can have multiple versions (for example, three different
types of 1s).  These conditions are used by Hwang~\cite{Hwang}, for example,
where the distribution of the number of summands (taken without multiplicities)
follows a central and local limit theorem.
There is also a central limit theorem proved by Madritsch and Wagner in~\cite{Mad},
concerning partitions whose base-$b$ representation only contains digits from some
given set.  This type of condition is natural given the fact that the bits in a
geometric random variable are independent, a property which was also exploited in,
 e.g.,~\cite{PDC} for the purpose of random sampling of integer partitions.

\subsection{ }\label{fixed:summands}
The asymptotic number of partitions with part sizes restricted to be less than some $t\sqrt{n}$, was studied by Romik~\cite{Romik}.
The limit shape of the conjugate set of partitions, i.e., the set of partitions with a fixed number of summands $t \sqrt{n}$, was studied by Yakubovich and Vershik~\cite{VershikYakubovich}.
This is another example, see \S\ref{Romik:example}, where a bijection involving conjugation is used to obtain the limit shape of a non-multiplicative set of restrictions.
A more detailed saddle point analysis counting the number of partitions of $n$ into exactly $m$ summands under further restrictions like those in~\S\ref{further:example:generalization} are contained in~\cite{Haselgrove}.

\subsection{ } \label{Romik:mistake}
Another proof of Theorem~\ref{theorem diff d k} appears in the last section of the unpublished
preprint~\cite{RomikUnpublished}.  This was used to obtain~\cite[Th.~1]{RomikUnpublished}.
Unfortunately, there appears to be a mistake, as the formula given for the limit shape does
not have unit area.

\subsection{ }
In this paper we consider only $\Pr_n = 1/|\A_n|$, the uniform distribution;
 see, however, \cite{Kerov} for limit shapes under Plancherel measure, and
 e.g.~\cite{JackMeasures} for other types of measures on partitions.

% Section - Final remarks and open problems

\bigskip

\section{Final remarks and open problems}
\label{final:remarks}

\subsection{ }\label{natural:generalization}
The next natural generalization of our theorems would be the one considered by Canfield and Wilf~\cite{CanfieldWilf}: partitions of size~$n$ with part sizes in set $U$ and with multiplicities in set $M$.  In our setting $M = \{0,1,2,\ldots\}$, and in \S\ref{Romik:example} we have $M = \{0,2,3,\ldots\}$.
It would be interesting to investigate conditions on $M$ and $U$ such that the limit shape exists.

\subsection{} \label{final:remarks-open}
Our approach does not extend to compute the limit shape for classes of partitions for which
no bijection to Andrews class partitions is known.  One notable example is the partitions
into powers of~$2$, which are equinumerous with certain \emph{Cayley compositions} (see~\cite{deBruijnPP,KP2}
and~\cite[\href{http://oeis.org/A000123}{A000123}]{OEIS}).  In this case the usual notion
of a limit shape does not exist and an alternative notion is required.

\subsection{}
Fristedt's method~\cite{Fristedt} in Section~\ref{sect prob} defines a random sampling algorithm for Andrews class partitions of a fixed size~$n$.
Namely, we sample independent geometric random variables $Z_1, Z_2, \ldots, Z_n$, where $Z_i$ has parameter $1-x^i$, $0<x<1$, $i \geq 1$, and check whether $\sum_{i=1}^n i Z_i = n$, repeating the sampling of $Z_1, Z_2, \ldots, Z_n$ until the condition is satisfied.
The expected number of times that we resample is given by $O(1/\sigma_n)$, where $\sigma_n := \sqrt{\mbox{Var}(\sum_{i=1}^ni\,Z_i)}.$

An application of probabilistic divide-and-conquer (PDC), introduced in~\cite{PDC}, see also~\cite{PDCDSH}, uses von Neumann's rejection sampling~\cite{Rejection} to obtain a speedup.
The algorithm is to sample repeatedly from $Z_2, Z_3, \ldots, Z_n$ and $U$, uniform on $(0,1)$, until
$$U \. \leq \. \frac{ \Pr(Z_1 = n - \sum_{i=2}^n i Z_i) }{ \Pr(Z_1 = 0)}\..
$$
The expected number of times that we resample this procedure is $O(n^{1/4})$, see~\cite[Table~1]{PDC}, a considerable speedup.

More generally, consider the set of partitions with parts in the set $U$, where $u_k \sim B k^r$.  The PDC deterministic second half algorithm to generate a random integer partition of size~$n$ with parts in the set $U$ is to sample repeatedly rom $Z_{u_2}, Z_{u_3}, \ldots$ and $U$, uniform on $(0,1)$, until
$$U \. \leq \. \frac{ \Pr(Z_{u_1} = n - \sum_{i=2}^n i Z_{u_i}) }{ \Pr(Z_{u_1} = 0)}\..
$$
The speedup in this case is given by
\[ \text{speedup} = \left(\max_k P(Z_{1} = k)\right)^{-1} \, = \,  P(Z_{u_1} = 0)^{-1} \, \sim \, \frac{1}{1-x^{u_1}} \, \sim \,  \frac{c}{u_1} \a \, = \, O\bigl(n^{r/(1+r)}\bigr)\ts.
\]
Hence, the expected number of rejections before a single accepted sample via the PDC deterministic second half algorithm is asymptotically $O\bigl(n^{(2+r)/(2+2r)}\bigr)$.

\subsection{}
There are also bijections defined as linear transformations by Corteel and Savage~\cite{CS}, and Corteel, Savage, and Wilf~\cite{CSW}, which correspond to partitions which satisfy various inequality conditions on the part sizes, in a sense generalizing the set $\mathcal{C}^r$ of partitions with nonnegative $r^{th}$ differences in Section~\ref{sect r}.
A further generalization was demonstrated by Pak~\cite{PakGeo}; most of those examples involve sets $U$ where $u_k$ grows exponentially in $k$, and so our theorems do not apply.

Of particular note is that our linear transformations are only assumed to act as Markov operators on the diagram function, so it is not necessary that the coefficients of the matrix transformation $\bv$ each be nonnegative; matrix $\bv$ simply needs to map positive coordinates to positive coordinates.

\subsection{}
For self-conjugate partitions,
it would be interesting to define a coupling over the set of random variables
governing the part sizes, although this has many apparent difficulties.
We hope to return to this problem in the future.

\subsection{ }
One should be careful when defining maps between partitions.
Consider the sets of partitions $\A = \{k^k 1^{k^2}\}$, $\B = \{k^{2k}\}$, $k\geq 1$.
\begin{prop}
For each $n \geq 1$, we have
\begin{equation}
|\A_n| = |\B_n| =  \begin{cases} 1 & n = 2k^2, \ k \geq 1\., \\
 0 & \text{otherwise\..} \end{cases}
 \end{equation}
\end{prop}
We showed earlier in Remark~\ref{vanishing:mass} that the limit shape of $\A$ under scaling function $\sqrt{n}$ does not exist, since the scaled diagram function converges to $a(x) = \frac{1}{\sqrt{2}}$ for $0 \leq x \leq \frac{1}{\sqrt{2}}$, which does not have unit area.
The limit shape of $\B$ exists and is equal to $b(x) = \frac{1}{2}$ for $0\leq x \leq 2$.  Consider the (non-bijective) transformation $\pi: 1 \to k$, which sends partitions in $\A$ to $\B$.
Even though there is a transformation, and each set of partitions tends to some limiting curve, there is no connection implied about the limit curves since the transformation is not even a bijection.

\vskip.5cm
\subsection*{Acknowledgements}
The authors are grateful to Ken Alexander, Richard Arratia, Quentin Berger,
Amir Dembo, Alejandro Morales, Greta Panova, Dan Romik, Bruce Rothschild,
Richard Stanley and Damir Yeliussizov for helpful discussions.
The second author was partially supported by the NSF.

% \bibliographystyle{plain}
%\bibliography{GLSBib}

%\vskip1.1cm

% **********************  END OF PAPER  ***********************

\end{document}